\documentclass[a4paper,10pt]{article}
\usepackage{amssymb} 
\usepackage{mathtools} 
\usepackage{amsthm} 
\usepackage[pdftex]{graphicx}
\usepackage[pdftex]{color}
\usepackage{natbib}
\usepackage{subcaption}
\usepackage{scalerel}
\usepackage{placeins}
\usepackage[shortlabels]{enumitem}
\RequirePackage[colorlinks,citecolor=blue,urlcolor=blue]{hyperref}
\usepackage{geometry}
\newtheorem{Def}{Definition}[section]
\newtheorem{Notation}[Def]{Notation}
\newtheorem{As}[Def]{Assumption}
\newtheorem{Lem}[Def]{Lemma}
\newtheorem{Prop}[Def]{Proposition}
\newtheorem{Theorem}[Def]{Theorem}
\newtheorem{Cor}[Def]{Corollary}

\newtheorem{Test}[Def]{Hypothesis Test}
\newtheorem{Rm}[Def]{Remark}
\newtheorem{Int}[Def]{Interpretation}

\DeclareMathOperator{\argmin}{\mbox{\rm argmin}}

\newcommand{\eps}{\varepsilon}

\newcommand{\iid}{\operatorname{\stackrel{i.i.d.}{\sim}}}

\newcommand{\inD}{\operatorname{\stackrel{\mathcal{D}}{\to}}}

\newcommand{\mun}{\widehat{\mu}_n}
\newcommand{\nun}{\widehat{\nu}_n}
\newcommand{\Vn}{\widehat{V}_n}
\newcommand{\gsim}{\:\raisebox{0.1\baselineskip}{\vstretch{0.7}{\gtrsim}}\:}
\DeclarePairedDelimiter{\ceil}{\lceil}{\rceil}
\binoppenalty=\maxdimen
\relpenalty=\maxdimen

\begin{document}
  \title{M-Variance Asymptotics and Uniqueness of Descriptors}
  \author{Benjamin Eltzner\footnote{Max Planck Institute for Multidisciplinary Sciences, G\"ottingen, \texttt{benjamin.eltzner@mpinat.mpg.de}\newline \phantom{......}University of G\"ottingen, \texttt{benjamin.eltzner@mathematik.uni-goettingen.de}}}
  \maketitle

\begin{abstract}
  Asymptotic theory for M-estimation problems usually focuses on the asymptotic convergence of the sample descriptor, defined as the minimizer of the sample loss function. Here, we explore a related question and formulate asymptotic theory for the minimum value of sample loss, the M-variance. Since the loss function value is always a real number, the asymptotic theory for the M-variance is comparatively simple. M-variance often satisfies a standard central limit theorem, even in situations where the asymptotics of the descriptor is more complicated as for example in case of smeariness, or if no asymptotic distribution can be given as can be the case if the descriptor space is a general metric space. We use the asymptotic results for the M-variance to formulate a hypothesis test to systematically determine for a given sample whether the underlying population loss function may have multiple global minima. We discuss three applications of our test to data, each of which presents a typical scenario in which non-uniqueness of descriptors may occur. These model scenarios are the mean on a non-euclidean space, non-linear regression and Gaussian mixture clustering.
\end{abstract}

\section{Introduction}

Asymptotic theory is a cornerstone of statistics and an important underpinning for the use of asymptotic distribution quantiles in hypothesis tests. For this reason, asymptotic theory has been generalized to a variety of different settings. Here we focus on M-estimators, which are a fairly broad class of estimators. The motivation of our approach stems from the Fr\'echet mean on manifolds but our results are applicable to a variety of different scenarios as illustrated by our data applications. In order to include descriptors like covariance matrices of normal distributions we do not restrict ourselves to linear spaces but develop asymptotic theory for non-Euclidean data and descriptor spaces. Asymptotic theory for such settings has been developed among others by \cite{HL98,BP05,H11a,BL17}. Notably, \cite{EH19} even extend asymptotic theory to probability measures featuring a lower asymptotic rate than $n^{-1/2}$, a phenomenon termed ``smeariness''.

\emph{M-estimators}, to which we restrict our attention here, are descriptors defined as minimizers of a loss function, the \emph{Fr\'echet function}. In all of the following, we call the minimum values of the population and sample Fr\'echet function the \emph{population} and \emph{sample M-variance}, respectively. In the case of the Fr\'echet function being the expected square distance, the M-variance is also known as the \emph{variance}. This should not be confused with the trace of the CLT-covariance, which is in the general case an unrelated quantity. We show a Central Limit Theorem (CLT) for the M-variance, building on the CLT for the Fr\'echet variance in general metric spaces given by \cite{DM18} and making heavy use of empirical process theory as presented in \cite{vdVW96,vdV00}. We show that a CLT can be formulated for the more general case that the Fr\'echet function does not have a unique global minimum but still a finite number of global minima in a straight-forward manner.

In the literature on the asymptotics of M-estimators on non-Euclidean data spaces, it is usually assumed that the population value of the descriptor is unique. However, while \cite{Sturm03} showed that the population Fr\'echet mean is unique on spaces with non-positive Alexandrov curvature, this need not hold on positively curved spaces, as illustrated by straight forward examples like two uniform modes centered at antipodal points of a circle or a uniform distribution on the equator of a sphere. On the other hand, real data always consists only of finitely many values, therefore one might ask whether the sample descriptor is unique under more general conditions. \cite{AM14} showed that for a sample from a continuously distributed random variable the sample Fr\'echet mean is almost surely unique.

Given this result, one may be tempted to rely on the sample mean as a meaningful data descriptor, since it is unique under general conditions. Furthermore, it is conceivable that the result can be generalized to a broad class of M-estimators. However, it is not clear, whether a sample with a unique descriptor may have been drawn from a population with non-unique descriptor. If this is the case, the full descriptor set of the population is not adequately reflected by the sample descriptor and conclusions from the sample descriptor may be invalid.

Here, we approach this problem and devise a hypothesis test for non-uniqueness of the population descriptor.
The null hypothesis is that the population descriptor is not unique, since wrongly assuming a unique population descriptor could lead to wrong results. As an example, consider a population with two population descriptor values $\mu^1$ and $\mu^2$ and two samples with unique sample descriptors, one near $\mu^1$ and the other near $\mu^2$, then a two-sample $T^2$-type test for equality of the descriptor might wrongly reject the null hypothesis because of the tacit underlying assumption of a unique population descriptor. To achieve an asymptotic size result for the hypothesis test, the generalized asymptotic theory for the M-variance formulated before is used.

We now give a brief sketch of our proposed hypothesis test to motivate the program of this paper. The null hypothesis of our test is that there are a number $m>1$ of population descriptors $\mu^i$. Let $\Vn^i$ denote the sample M-variance in the neighborhood of any $\mu^i$ and let $W^{ij}$ be the CLT variance such that $\sqrt{n} (\Vn^i - \Vn^j) \to N(0,W^{ij})$. For any $i$ consider the index sets $I_i := \{1, \dots, i-1, i+1, \dots, m\}$. Then the type I error $\zeta_n$ is given by the probability that the sample M-variance $\Vn^i$ in the neighborhood of any $\mu^i$ is smaller by at least the normal quantile $q^{ij}_n \left(\frac{\alpha}{2}\right) := \frac{\sqrt{W^{ij}}}{\sqrt{n}} q \left(\frac{\alpha}{2}\right)$ than all other $\Vn^j$, where $q$ is the standard normal quantile function,
\begin{align*}
  \zeta_n(\alpha) := \sum_{i=1}^m \mathbb{P} \left( \forall j \in I_i : \, \left(\Vn^i - \Vn^j\right) \le q^{ij}_n \left(\frac{\alpha}{2}\right) \right) \, .
\end{align*}
As first step of our hypothesis test, we take a set of $B$ $n$-out-of-$n$ bootstrap samples and calculate bootstrap M-variances $V_{n,n}^{*,i}$ in the neighborhoods of the $\mu^i$. We define probabilities of bootstrap M-variances under the condition that $\Vn^i - \Vn^j \approx q^{ij}_n \left(\frac{\alpha}{2}\right)$ for all $j$,
\begin{align*}
  \xi_n(\alpha) := \lim_{\eps \to 0} \mathbb{P} \left(\exists j \in I_i : \, V_{n,n}^{*,i} - V_{n,n}^{*,j} > 0 \, \middle| \, \forall j \in I_i : \, \frac{1}{q^{ij}_n \left(\frac{\alpha}{2}\right)} \left(\Vn^i - \Vn^j\right) \in \left[1, 1 + \eps \right] \right) \, .
\end{align*}
Thus, $\xi_n$ is the probability that despite the fact that $\Vn^i - \Vn^j = q^{ij}_n \left(\frac{\alpha}{2}\right)$ for all $j$, which is the boundary of the region defining the type I error, the bootstrap M-variance $V_{n,n}^{*,i}$ is not the smallest, but there is some $V_{n,n}^{*,j} < V_{n,n}^{*,i}$. The probability is defined over all samples of size $n$ and all corresponding bootstrap samples of size $n$ and is therefore a deterministic value. This means that the bootstrap minimizer $\mu_{n,n}^*$ will not be close to the global sample descriptor $\mun = \mun^i$ but to the local sample descriptor $\mun^j$. In other words, under the null hypothesis there are multiple clusters of bootstrap descriptors.

The second step of our hypothesis test is therefore to determine whether there are multiple clusters in the set of bootstrap descriptors $\{\mu_{n,n}^{*,i}\}_i$ and count the number $k_{\text{far}}$ of bootstrap descriptors, which are in another cluster than the global sample descriptor $\mun$. For this, we use the multiscale test by \cite{DW08}, so we can use their uniform size and power results to show an asymptotic size result and conjecture an asymptotic power result for our test. To do this, we only consider the distances of the bootstrap descriptors from the sample descriptor, thereby translating the hypothesis test into a search for modes in a set of one-dimensional values.

The third and last step of our test determines whether $k_{\text{far}}$ is small enough to reject the null hypothesis. For this, we use the asymptotic result shown in Theorem \ref{theo:quantileLLN},
\begin{align*}
  \lim_{n \to \infty} 2 \xi_n(\alpha) \ge \alpha \ge \lim_{n \to \infty} \zeta_n(\alpha)
\end{align*}
which means that the test rejects  at level $\alpha$ if $2 k_{\text{far}} /B < \alpha$, and we can conclude that the descriptor is unique to the chosen level. Otherwise, the possibility that the mean is not unique cannot be rejected.

We present three data examples to show the wide range of application scenarios of our test. The first data set concerns the mean on a circle and is therefore very close to the motivating ideas presented above. The second application is to non-linear least-squares regression, formulated as an M-estimation problem. Such curve fit estimators are very common in the natural sciences and are therefore highly relevant. We find that, although the parameter space looks innocuous, curve fits are found to be non-unique in some cases. The third application concerns a Gaussian mixture which is used for clustering. This method and related centroid clustering methods are very commonly used and the question which number of clusters should be assumed to describe a data set is a well known subject of discussion. If one assumes the point of view that an unambiguous cluster result is essential, our test for non-uniqueness can be used as a model selection tool.

It has been repeatedly suggested to us in private communication that instead of \textit{uniqueness} we should use the term \textit{identifiability} of the descriptor. However, in the case of non-uniqueness, the descriptor of interest can be regarded as set-valued and the set of descriptors might be identifiable in the sense that consistent estimators can be given by using clusters of bootstrap estimators. A detailed discussion of identifiability of set-valued descriptors is deferred to future research but since it seems plausible that such a limited form of identifiability is preserved in some cases, we use the less absolute term \textit{non-unique} instead of \textit{non-identifiable} in this paper.

\section{M-Variance Asymptotics on a Metric Space}\label{section:preliminaries}

In this section we discuss the CLT for the M-variance, which plays a crucial role for our considerations, in the very general setting of metric spaces. Especially, we will show that even when an estimator exhibits a lower asymptotic rate than $n^{-1/2}$, see e.g. \cite{EH19}, the asymptotic rate for the M-variance will remain $n^{-1/2}$ under fairly generic circumstances.

In all of the following let $Q$ be a topological space called the \textit{data space} and $\mathcal{P}$ a separable topological space with continuous metric function $d$ called the \textit{parameter space}. $\Omega$ is a silently underlying probability space as usual. Let $\rho : \mathcal{P} \times Q \to \mathbb{R}$ a continuous function, $X : \Omega \to Q$ a $Q$-valued random variable and $X_1, \dots, X_n \iid X$ a random sample. Define Fr\'echet functions and descriptors for samples and populations
\begin{align}
  F(p) &= \mathbb{E}[\rho(p, X)] & E &= \left\{ p \in \mathcal{P} :F(p) = \inf_{p \in \mathcal{P}} \limits F(p) \right\} \\
  F_n(p) &= \frac{1}{n} \sum_{j=1}^{n} \limits \rho(p, X_j) & E_n &= \left\{ p \in \mathcal{P} :F_n(p) = \inf_{p \in \mathcal{P}} \limits F_n(p) \right\}\, .
\end{align}
The elements of $E_n$, if it is non-empty, are commonly called \emph{M-estimators}. The term \emph{M-estimator} has been widely used in the field of robust statistics. In the field of Statistics on non-Euclidean spaces the term \emph{generalized Fr\'echet means} has been introduced by \cite{H11a} for M-estimators where the elements of $\mathcal{P}$ define geometrical objects, for example geodesics of $Q$. In the present article, we are not specifically concerned with robust M-estimators nor with geometric objects, so we use the term \emph{M-estimator} in its general meaning.

Before continuing the exposition, we fix some notation.
\begin{Notation}\phantom{}
  \begin{enumerate}[(i)]
    \item For a point $p \in \mathcal{P}$ and $\eps > 0$ let $\mathcal{B}_\eps(p) = \{p' \in \mathcal{P} : d(p, p') < \eps \}$.
    \item For a point $x \in \mathbb{R}^p$ and $\eps > 0$ let $B_\eps(x) = \{x' \in \mathbb{R}^p : |x'-x| < \eps \}$.
  \end{enumerate}
\end{Notation}

To define the asymptotics of the estimator $E_n$, it is first necessary to make sure that it is consistent, i.~e. it converges to $E$ for large $n$. We will use a rather general formulation of consistency, introduced by \cite{BP03} and based on earlier work by \cite{Z77}. Consistency has been derived by \cite{BP03,H11a} under a few technical conditions on $\rho$ as well as $\mathcal{P}$ satisfying separability and the cluster points of the sample Fr\'echet mean sets satisfying the Heine Borel property. In the present article we do not expand on this subject. We will simply require $E_n$ to be consistent in all of the following.

\begin{As}[Bhattacharya-Patrangenaru strong consistency (BPSC)]$ $\label{as:BPSC}\\
  The estimator $E_n$ is Bhattacharya-Patrangenaru strongly consistent (BPSC), i.e. for every $\eps > 0$ there is a $\Omega_0 \subseteq \Omega$ with $\mathbb{P}(\Omega_0) = 1$ such that
  \begin{align*}
    \forall \omega \in \Omega_0 \quad \exists n_0 > 0 \quad \forall n \ge n_0 \quad \forall \widehat{\mu}_n \in E_n \quad \exists \mu \in E : \quad d (\widehat{\mu}_n, \mu) \le \eps
  \end{align*}
\end{As}
The proofs of BPSC given by \cite[Theorem 2.3]{BP03} and \cite[Theorem A.4]{H11a} do not assume that $E$ contains exactly one element, but that $E \neq \emptyset$. Therefore, the existing theory is readily applicable in our setting. The assumption of BPSC is fundamental to all arguments presented here and it is unlikely that it can be relaxed without requiring significant restrictions in return.

We will approach the case that the population descriptor is not unique but we will assume throughout this article that the population descriptor set consists of a finite number of points.

\begin{As}\label{as:finite}
  The set of population descriptors comprises $m$ points $E = \{\mu^1, \dots, \mu^m\} \subset \mathcal{P}$ and we define $\delta > 0$ such that $\mathcal{B}_\delta(\mu^j) \cap \mathcal{B}_\delta(\mu^k) = \emptyset$ holds for any two points in $E$.
\end{As}

This assumption excludes cases, where the population descriptor set contains a submanifold of descriptor space. This can happen for example due to symmetry in the system leading to minima sets given by orbits of symmetry transformations. One simple example is the set of Fr\'echet means on a sphere with a point mass on the north pole and a point mass on the south pole, which is a circle of constant latitude. A generalization to such cases requires a different approach and is beyond the scope of the article.

In the following, we will briefly consider the general case that $\mathcal{P}$ is simply a metric space, not necessarily a manifold. The setting we use is very similar to the one used by \cite{DM18}. However, for general M-variance, we need the following additional assumption that the function $\rho$ exhibit some regularity beyond what is needed for strong consistency.

\begin{As}[Almost Surely Locally Lipschitz]\label{as:Lipschitz0}
  Assume $\delta>0$ from Assumption \ref{as:finite} and that for every $\mu^i \in E$, there is a measurable function $\dot{\rho}^i: Q \to \mathbb{R}$ satisfying $\mathbb{E} [\dot{\rho}^i(X)^2] < \infty$ and a metric $d_i$ of $\mathcal{P}$, such that the following Lipschitz condition 
  \begin{align*}
    |\rho(p_1, X) - \rho(p_2, X)| \le \dot{\rho}^i(X) d_i(p_1, p_2) \quad\mathrm{ a.s.}
  \end{align*}
  holds for all $p_1, p_2\in \mathcal{B}_\delta(\mu^i)$.
\end{As}

This assumption is always satisfied for the Fr\'echet variance corresponding to the case $\mathcal{P} = Q$ and $\rho(p,x) = d(p,x)^2$, as can be seen from \cite{DM18}. More precisely, that article employs the additional restriction that the space $Q$ be compact but since we rely on the consistency Assumption \ref{as:BPSC}, the Lipschitz condition is only required to hold locally.

Finally, one needs a so-called entropy condition. To state this assumption, we need to introduce two important concepts

\pagebreak[2]
\begin{Def}\phantom{}
  \begin{enumerate}[(i)]
    \item For a class of functions $\mathcal{F}$ from $Q$ to $\mathbb{R}$, an envelope function $F: Q \to \mathbb{R}$ is any function, such that for any $f \in \mathcal{F}$ and $q \in Q$, $|f(q)| \le F(q)$.
    \item For a totally bounded metric space $(\mathcal{P}, d)$ and a size $\eps$ the \emph{covering number} $N(\eps, \mathcal{P}, d)$ is defined as the minimal number of $\eps$ balls needed to cover $\mathcal{P}$.
  \end{enumerate}
\end{Def}

\begin{As}[Entropy bound]\label{as:entropy}
  For the classes $\mathcal{F}^i := \Big\{ \rho^i(p, \cdot) - \rho^i(\mu^i, \cdot), p \in \mathcal{B}_\delta(\mu^i) \Big\}$ with envelope functions $2 \delta \dot{\rho}^i$ the following entropy bound holds
  \begin{align*}
    \lim_{\delta \to 0} \limits \delta \int_0^1 \limits \sqrt{1 + \log N(\eps \delta, \mathcal{B}_\delta(\mu^i), d_i ) } \, d\eps = 0 \, .
  \end{align*}
\end{As}

This bound is called an entropy bound, because the integral can be expressed in terms of the covering entropy defined in Definition \ref*{def:cover-number} in the appendix. The bound is fairly general as discussed in \cite{DM18}. Lemma \ref*{lem:entropy} and Theorem \ref*{theo:emp-proc-bound} show that a slightly altered entropy bound holds in the finite dimensional manifold setting discussed below. The local Lipschitz and entropy assumptions are discussed in more detail in Appendix~\ref*{section-supp:aux-metric}.

Now, we have introduced all the tools necessary to extend the result of \cite{DM18} to general M-estimators.

\begin{Theorem}[CLT for the M-variance on metric spaces]\label{theo:CLT-m-variance-metric-spaces}
  Under Assumptions \ref{as:BPSC}, \ref{as:finite}, \ref{as:Lipschitz0} and \ref{as:entropy} we have, for a measurable selection $\mun^i \in \left\{ x \in \mathcal{B}_\delta(\mu^i) : F_n(x) = \inf_{y \in \mathcal{B}_\delta(\mu^i)} \limits F_n(y) \right\}$
  \begin{align*}
    \sqrt{n} \big(F_n(\mun^i) - F(\mu^i)\big) \inD \mathcal{N}\left(0, \textnormal{Var}\left[\rho(\mu^i, X)\right]\right) \, .
  \end{align*}
\end{Theorem}

\begin{proof}
  From Proposition \ref*{prop:CLT-metric} we have
  \begin{align*}
    \sqrt{n} \big(F_n(\mun^i) - F(\mu^i)\big) &= \sqrt{n} \big(F_n(\mu^i) - F(\mu^i)\big) + o_P(1)\\
    &= \frac{1}{\sqrt{n}} \sum_{j=1}^n \left( \rho(\mu^i, X_j) - \mathbb{E}[\rho(\mu^i, X)] \right) + o_P(1) \, .
  \end{align*}
  The result follows using the standard CLT.
\end{proof}

\section{M-Variance Asymptotics on a Manifold}\label{section:asymptotics}

\subsection{Preliminaries}

We will now restrict to the case that $\mathcal{P}$ additionally has a Riemannian manifold structure and is finite dimensional. In that case, we can drop Assumption \ref{as:entropy}, since it will follow from a suitably modified form of Assumption \ref{as:Lipschitz0}. We now introduce some notation and assumptions for the case that the population descriptor is not unique as formulated in Assumption \ref{as:finite}, i.e. $E = \{\mu^1, \dots, \mu^m\}$, in order to use them in later sections.

\begin{As}[Local Manifold Structure] \label{as:local-manifold}
  For every $\mu^i \in E$, assume that there is an $2\leq r_i \in \mathbb{R}$ and a neighborhood $\widetilde{U}_i$ of $\mu^i$ that is a $p$-dimensional Riemannian manifold, $p\in \mathbb{N}$, such that with a neighborhood $U_i$ of the origin in $\mathbb{R}^p$ the exponential map $\exp_{\mu^i} : U_i \to \widetilde{U}_i$, $\exp_{\mu^i}(0)= \mu^i$, is a $C^{\ceil{r_i}}$-diffeomorphism. For $x \in U_i$ and $q\in Q$ we introduce the notation
  \begin{align*}
    \tau^i &: (x,q) \mapsto \rho (\exp_{\mu^i}(x), q)\,,\\
    G^i &: x \mapsto F(\exp_{\mu^i}(x))\,, & G^i_n &: x \mapsto F_n(\exp_{\mu^i}(x))\,.
  \end{align*}
\end{As}
This assumption allows for slightly more general descriptor spaces than manifolds, like stratified spaces. However, in such spaces one is restricted to random variables which take their descriptor values in the manifold part of the space. Stratified spaces, in which the mean is generically taken on the manifold part, are called \emph{manifold stable} and \cite{H_meansmeans_12} showed that landmark shape spaces enjoy this property. The local manifold assumption serves mostly as a convenience and it allows for the generalization of assumptions \ref{as:Lipschitz} and \ref{as:Lipschitz2} to H\"older continuity.

\begin{As}[Almost Surely Locally H\"older Continuous and Differentiable at Descriptor]\label{as:Lipschitz}
  Assume that for every $\mu^i \in E$
  \begin{enumerate}[(i)]
    \item the gradient $\dot{\tau}^i_0 (X) := \mathrm{grad}_q \tau^i(q,X)|_{q=0}$ exists almost surely;
    \item there is a measurable function $\dot{\tau}^i: Q \to \mathbb{R}$ satisfying $\mathbb{E} [\dot{\tau}^i(X)^2] < \infty$ and some $0 < \beta \le 1$ such that the following H\"older condition
    \begin{align*}
      |\tau^i(x_1, X) - \tau^i(x_2, X)| \le \dot{\tau}^i(X) \| x_1 - x_2 \|^\beta \quad\mathrm{ a.s.}
    \end{align*}
    holds for all $x_1,x_2\in U_i$.
  \end{enumerate}
\end{As}
This Assumption holds for many loss functions including for the $L_p$-minimizer, principal geodesics and maximum likelihood estimators. The assumption boils down to requiring a finite second moment of the random variable in case of the mean on Euclidean space. A generalization of this type of condition using so-called quadruple inequalities was discussed by \cite{Schoetz2019}.

Using Assumptions \ref{as:local-manifold} and \ref{as:Lipschitz}, one can replace the entropy condition with a similar condition. This is detailed in Lemma \ref*{lem:entropy} in Appendix~\ref*{subsec-supp:entropy}.

\subsection{A CLT for Fr\'echet M-Variance and M-Variance Differences}

The first asymptotic result we show is the asymptotic behavior of the M-variance. This result will be used repeatedly throughout the following.

\begin{Def}
  For a population descriptor $\mu^i$ the set of all \emph{$\delta$-local sample $i$-descriptors} is defined by
  \begin{align*}
    E^i_n = \left\{ x \in B_\delta(0) \subset \mathbb{R}^p : G_n^i(x) = \inf_{y \in B_\delta(0)} \limits G_n^i(y) \right\} \, .
  \end{align*}
  Elements from this set are denoted $\nun^i \in E^i_n$. In a slight abuse of terminology, we will also call $\mun^i := \exp_{\mu^i}(\nun^i)$ a $\delta$-local sample $i$-descriptor.
\end{Def}

According to the result of \cite{AM14}, the $\delta$-local sample $i$-descriptors are almost surely unique in case of the Fr\'echet mean if the random variable $X$ possesses a density with respect to Lebesgue measure in charts. However, to account for more general cases, where they may not be unique, we will state the following results in terms of measurable selections of $\delta$-local sample $i$-descriptors.

We denote the vector of $\tau^i(0,X)$ and its covariance as
\begin{align*}
  \tau(0,X) :=& \begin{pmatrix} \tau^1(0, X) \\ \vdots \\ \tau^m(0, X) \end{pmatrix}\\
  \textnormal{Cov}[\tau(0,X)] :=& \mathbb{E} \left[ \big(\tau(0,X) - \mathbb{E}[\tau(0,X)]\big) \big(\tau(0,X) - \mathbb{E}[\tau(0,X)]\big)^T \right]
\end{align*}
Furthermore, we introduce the shorthand notations
\begin{align}
  \Vn^i :=& G^i_n(\nun^i) & \Vn :=& \begin{pmatrix} \Vn^1 \\ \vdots \\ \Vn^m \end{pmatrix} & V :=& G^1(0) = \dots = G^m(0) & \mathbf{1}_m :=& \begin{pmatrix} 1 \\ \vdots \\ 1 \end{pmatrix} \label{eq:v-def}
\end{align}
for the sample M-variance $\Vn^i$ and the population M-variance $V$.

\begin{Theorem}[M-Variance Asymptotics] \label{theo:CLT-m-variance}
  Under Assumptions \ref{as:BPSC}, \ref{as:finite}, \ref{as:local-manifold} and \ref{as:Lipschitz}, using the notation from Equation \eqref{eq:v-def} we have for measurable selections $\nun^i$ of $\delta$-local sample $i$-descriptors for all $i \in \{1, \dots, m\}$ 
  \begin{align*}
    \sqrt{n} \left( \Vn - V \mathbf{1}_m \right) \inD \mathcal{N}\Big( 0, \textnormal{Cov}[\tau(0,X)] \Big) \, .
  \end{align*}
\end{Theorem}

\begin{proof}
  We use the fact that $G^i(0) = G^j(0)$ for all $i$ and $j$ and Proposition \ref*{prop:CLT-manifold2} to get
  \begin{align*}
    \sqrt{n} \left( \Vn - V \mathbf{1}_m \right) &= \sqrt{n} \begin{pmatrix} G^1_n(\nun^1) - G^1(0) \\ \vdots \\ G^m_n(\nun^m) - G^m(0) \end{pmatrix} = \sqrt{n} \begin{pmatrix} G^1_n(0) - G^1(0) \\ \vdots \\ G^m_n(0) - G^m(0) \end{pmatrix} + o_P(1)\\
    &= \frac{1}{\sqrt{n}} \sum_{j=1}^n \begin{pmatrix} \tau^1(0, X_j) - \mathbb{E}[\tau^1(0, X)] \\ \vdots \\ \tau^m(0, X_j) - \mathbb{E}[\tau^m(0, X)] \end{pmatrix} + o_P(1) \, .
  \end{align*}
  The claim follows from the standard multivariate Euclidean CLT.
\end{proof}

Appendix~\ref*{subsec-supp:aux-clt} expands on the proof and the interpretation of this theorem.

\begin{Cor}[M-Variance Difference Asymptotics] \label{cor:CLT-m-variance2}
  Under Assumptions \ref{as:BPSC}, \ref{as:finite}, \ref{as:local-manifold} and \ref{as:Lipschitz} we have, for measurable selections $\nun^i$ of $\delta$-local sample $i$-descriptors for any $i, j \in \{1, \dots, m\}$
  \begin{align*}
    \sqrt{n} \big(\Vn^i - \Vn^j \big) \inD \mathcal{N}\Big( 0, \textnormal{Var}[\tau^i(0,X) - \tau^j(0,X)] \Big)
  \end{align*}
\end{Cor}

\begin{proof}
  The claim follows immediately from Theorem \ref{theo:CLT-m-variance} where we use $G^i(0) = G^j(0) = V$, the definition of the multivariate normal and
  \begin{align*}
    (e_i -e_j)^T \textnormal{Cov}[\tau(0,X)] (e_i -e_j) &= \textnormal{Var}[\tau^i(0,X) - \tau^j(0,X)] \, .
  \end{align*}
  \vspace*{-2\baselineskip}\\
  \phantom{}
\end{proof}

For the following it is necessary that the covariance matrix in the CLT $\textnormal{Cov}[\tau(0,X)]$ has full rank. This is true for a wide range of population measures and descriptors. A detailed investigation of this topic is beyond the scope of this paper, so we introduce a new assumption to this effect.

\begin{As}\label{as:covar-rank}
  The covariance matrix $\textnormal{Cov}[\tau(0,X)]$ has full rank.
\end{As}

Note that this assumption ensures that also $\textnormal{Var}[\tau^i(0,X) - \tau^j(0,X)] > 0$ for any $i \neq j$.

\subsection{Asymptotics for M-Variance Bootstrap and Covariance}

We are working towards a hypothesis test based on the bootstrap, therefore we need to establish asymptotic results for bootstrap estimates as well. In all of the following, recall $\mun^i := \exp_{\mu^i}(\nun^i)$.

\begin{Def}[Bootstrap Fr\'echet function]
  For a sample $X_1, \dots, X_n \iid X$
  let $X_1^*, \dots X_k^*$ be a bootstrap sample which is drawn with replacement. Let $X^*$ be the random variable associated with the empirical measure $\mathbb{P}_n$. For measurable selections $\mun^i$ of $\delta$-local sample $i$-descriptors we define
  \begin{align*}
    \tau_n^{*,i} &: (x,q) \mapsto \rho(\exp_{\mun^i}(x), q)\,,\\
    G^{*,i}_n &: x \mapsto \mathbb{E}[\tau_n^{*,i}(x,X^*)] = \frac{1}{n} \sum_{j=1}^n \limits \tau_n^{*,i}(x,X_j)\,, & G^{*,i}_{n,k} &: x \mapsto \frac{1}{k} \sum_{j=1}^k \limits \tau_n^{*,i}(x,X_j^*)\,.
  \end{align*}
\end{Def}

\begin{Rm}
  The notation $\tau_n^{*,i}$ is chosen in analogy to the notation for bootstrap quantities, since it is evaluated on the bootstrap random variable $X^*$. In the definition of $G^{*,i}_n$ and everywhere else below, the expected value $\mathbb{E}[\cdot]$ as well as the covariance $\textnormal{Cov}[\cdot]$ if applied on a function of $X^*$ are understood to use the distribution measure of this random variable. Note that in the notation chosen here, we have $\tau_n^{*,i}(0,q) = \tau^i(\nun^i, q)$ and in consequence $G^{*,i}_n (0) = G^i_n (\nun) = \Vn^i$ as equivalent notation for the sample M-variance.
\end{Rm}

\begin{Def}
  For a population descriptor $\mu^i$ and corresponding measurable selection of $\delta$-local sample $i$-descriptor $\mun^i$ let $U_i := \log_{\mun^i} (\mathcal{B}_\delta(\mu^i))$. Then the set of all \emph{$\delta$-local bootstrap $i$-descriptors} is defined by
  \begin{align*}
    E^{*,i}_{n,k} = \left\{ x \in U_i : G^{*,i}_{n,k}(x) = \inf_{x \in U_i} \limits G^{*,i}_{n,k}(x) \right\} \, .
  \end{align*}
  Elements from this set are denoted $\nu^{*,i}_{n,k} \in E^{*,i}_{n,k}$.
\end{Def}

Denoting the vector of $\tau_n^{*,i}(0,X^*)$ and its covariance as
\begin{align*}
  \tau_n^*(0,X^*) :=& \begin{pmatrix} \tau_n^{*,1}(0,X^*) \\ \vdots \\ \tau_n^{*,m}(0,X^*) \end{pmatrix} \qquad \tau(\nun,X^*) := \begin{pmatrix} \tau^1(\nun^1,X^*) \\ \vdots \\ \tau^1(\nun^m,X^*) \end{pmatrix}\\
  \textnormal{Cov}[\tau_n^*(0,X^*)] :=& \mathbb{E} \left[ \big(\tau_n^*(0,X^*) - \mathbb{E}[\tau_n^*(0,X^*)]\big) \big(\tau_n^*(0,X^*) - \mathbb{E}[\tau_n^*(0,X^*)]\big)^T \right]\\
  =& \mathbb{E} \left[ \big(\tau(\nun,X^*) - \mathbb{E}[\tau(\nun,X^*)]\big) \big(\tau(\nun,X^*) - \mathbb{E}[\tau(\nun,X^*)]\big)^T \right]\\
  =& \frac{1}{n} \sum_{j=1}^n \limits \tau(\nun, X_j)\tau(\nun, X_j)^T - \frac{1}{n^2} \sum_{j=1}^n \limits \sum_{k=1}^n \limits \tau(\nun, X_j)\tau(\nun, X_k)^T
\end{align*}
and using notation
\begin{align}
  V^{*,i}_{n,k} :=& G^{*,i}_n(\nu^{*,i}_{n,k}) & V_{n,k}^{*} :=& \begin{pmatrix} V_{n,k}^{*,1} \\ \vdots \\ V_{n,k}^{*,m} \end{pmatrix}  \label{eq:v-star-def}
\end{align}
for the bootstrap M-variance, we get

\begin{Cor}[Bootstrap M-Variance Asymptotics] \label{cor:CLT-boot-m-variance}
  Under Assumptions \ref{as:BPSC}, \ref{as:finite} and \ref{as:local-manifold} and assuming that the population measure $P$ and the empirical measure $\mathbb{P}_n$ both satisfy Assumption \ref{as:Lipschitz}, we have, for measurable selections $\nun^i$ of $\delta$-local sample $i$-descriptors for all $i \in \{1, \dots, m\}$ in the limit $k \to \infty$,
  \begin{align*}
    \sqrt{k} \left( V_{n,k}^{*} - \Vn \right) &\inD \mathcal{N}\Big( 0, \textnormal{Cov}[\tau_n^*(0,X^*)] \Big) \, ,\\
    \sqrt{k} \left( \left( V^{*,i}_{n,k} - V^{*,j}_{n,k} \right) - \left( \Vn^i - \Vn^j \right) \right) &\inD \mathcal{N}\Big( 0, \textnormal{Var}\left[\tau^{*,1}_n(0, X^*) - \tau^{*,2}_n(0, X^*)\right] \Big) \, .
  \end{align*}
  Note that by requiring $\mathbb{P}_n$ to satisfy Assumption \ref{as:Lipschitz} and since the covariance depends explicitly on the sample, this result is dependent on the sample.
\end{Cor}

\begin{proof}
  The claims follow immediately from Theorem \ref{theo:CLT-m-variance} and Corollary \ref{cor:CLT-m-variance2}.
\end{proof}

\noindent Next, we are looking for an asymptotic result for $\textnormal{Cov}[\tau_n^*(0,X^*)]$ and $\textnormal{Cov}[\tau(0,X)]$.
\begin{Def}[Product Fr\'echet function]
  For a sample $X_1, \dots, X_n \iid X$, define
  \begin{align*}
    \tau^{ij} &: (x^1, x^2, X) \mapsto \tau^{i}(x^1,X)\tau^{j}(x^2,X)\\
    G^{ij} &: (x^1, x^2) \mapsto \mathbb{E}[\tau^{ij}(x^1,x^2,X)] \,, & G^{ij}_{n} &: (x^1, x^2) \mapsto \frac{1}{n} \sum_{k=1}^n \limits \tau^{ij}(x^1,x^2,X_k)\,.
  \end{align*}
\end{Def}

\begin{As}[Almost Surely Locally H\"older Continuous and Differentiable at Descriptor]\label{as:Lipschitz2}
  Let Assumption \ref{as:Lipschitz} hold for both the population measure $P$ and the empirical measure $\mathbb{P}_n$ and, denoting $\tau^i_{\max}(X) := \max_{x \in \mathcal{P}^i} \limits \tau^i(x, X)$, assume that for every $\mu^i, \mu^j \in E$ the function
  \begin{align*}
    \dot{\tau}^{ij}(X) := \tau^i_{\max}(X) \dot{\tau}^j(X) + \tau^j_{\max}(X) \dot{\tau}^i(X)
  \end{align*}
  is measurable and satisfies $\mathbb{E} [\dot{\tau}^{ij}(X)^2] < \infty$.
\end{As}

\begin{Rm}
  From Assumption \ref{as:Lipschitz2} one can see, using $\tau^i_{\max}$ from Assumption \ref{as:Lipschitz2}, that the following H\"older condition
  \begin{align*}
    |\tau^{ij}(x^i, x^j, X) - \tau^{ij}(y^i, y^j, X)| \le& \tau^i_{\max}(X) \left| \tau^j(x^j, X) - \tau^j(y^j, X) \right|\\
    & + \tau^j_{\max}(X) \left| \tau^i(x^i, X) - \tau^i(y^i, X) \right|\\
    \le& \tau^i_{\max}(X) \dot{\tau}^j(X) \left| x^j- y^j \right|^\beta + \tau^j_{\max}(X) \dot{\tau}^i(X) \left| x^i- y^i \right|^\beta \quad\mathrm{ a.s.}\\
    \le& \dot{\tau}^{ij}(X) \left| \begin{pmatrix} x^i \\ x^j \end{pmatrix} - \begin{pmatrix} y^i \\ y^j \end{pmatrix} \right|^\beta \quad\mathrm{ a.s.}
  \end{align*}
  holds for all $x^i,y^i\in U_i$ and $x^j,y^j\in U_j$.
  
  Assumption \ref{as:Lipschitz2} includes a moment condition which amounts to finiteness of the sixth moment in case of the mean on Euclidean space. Again, the assumption holds in many cases including $L_p$ minimizers, principal geodesics and maximum likelihood estimators. Like Assumption \ref{as:Lipschitz}, it holds if the $\tau^i$ are differentiable in the first argument and their derivatives are bounded in neighborhoods of the $\mu^i$. If Assumption \ref{as:Lipschitz} is generalized by using quadruple inequalities as presented by \cite{Schoetz2019}, Assumption \ref{as:Lipschitz2} would have to be adjusted accordingly.
\end{Rm}

For our setting, we do not need a CLT for $\textnormal{Cov}[\tau(0,X)]$. Instead we restrict attention to a simpler result which is sufficient for our purposes.

\begin{Theorem}[M-Variance Covariance Asymptotics] \label{theo:varvarCLT}
  Under Assumptions \ref{as:BPSC}, \ref{as:finite}, \ref{as:local-manifold}, \ref{as:covar-rank} and \ref{as:Lipschitz2}, for a measurable selection of $\delta$-local sample $i$-descriptors $\nun^i$ and for any vector $v \in \mathbb{R}^m$ there is a $W_{v} \in \mathbb{R}^+$ such that
  \begin{align*}
    \sqrt{n} \Big(v^T \textnormal{Cov}[\tau_n^*(0,X^*)] v - v^T \textnormal{Cov}[\tau(0,X)] v\Big) \inD \mathcal{N} \left(0,W_{v} \right) \, ,
  \end{align*}
  with the variance given by
  \begin{align*}
    W_v = \textnormal{Var}\left[\left(v^T \tau(0, X) \right)^2\right] - 4 \|v\|_1 V C^v + 4 \|v\|_1^2 V^2 \textnormal{Var}\left[v^T \tau(0, X)\right] \, .
  \end{align*}
\end{Theorem}

\begin{proof}
  The proof can be found in Appendix~\ref*{subsec-supp:cov-clt}.
\end{proof}

Based on Corollaries \ref{cor:CLT-m-variance2} and \ref{cor:CLT-boot-m-variance} and Theorem \ref{theo:varvarCLT} we now get a convergence result for quantiles. For any $i,j \in \{1, 2, \dots, m\}$ with $i \neq j$ let $e_{ij} := (e_i -e_j)$, where $e_i$ denote canonical basis vectors as usual, and
\begin{align*}
  W^{ij} &= e_{ij}^T \textnormal{Cov}[\tau(0,X)] e_{ij} = \textnormal{Var}[\tau^i(0,X) - \tau^j(0,X)]\\
  \widehat{W}_n^{ij} &= e_{ij}^T \textnormal{Cov}[\tau_n^*(0,X^*)] e_{ij} = \textnormal{Var}\left[\tau^{*,1}_n(0, X^*) - \tau^{*,2}_n(0, X^*)\right] \, .
\end{align*}
Recall that $W^{ij}$ appears as asymptotic variance $\sqrt{n} \big(\Vn^i - \Vn^j \big) \inD \mathcal{N}\left( 0, W^{ij} \right)$ in Corollary \ref{cor:CLT-m-variance2} and analogously $\sqrt{k} \left( \left( V^{*,i}_{n,k} - V^{*,j}_{n,k} \right) - \left( \Vn^i - \Vn^j \right) \right) \inD \mathcal{N}\left( 0, \widehat{W}_n^{ij} \right)$ in Corollary \ref{cor:CLT-boot-m-variance}. Let $q$ be the standard normal quantile function and use the shorthand $q_n^{ij}\left(\frac{\alpha}{2} \right) := \frac{\sqrt{W^{ij}} q\left(\frac{\alpha}{2} \right)}{\sqrt{n}}$. In all of the following, we will always consider $\alpha \le 1$ and thus $q_n^{ij}\left(\frac{\alpha}{2} \right) \le 0$. With this notation, we define the following sets,
\begin{align*}
  A_{n,ij}^{\alpha,\eps} :=& \left\{x \in \mathbb{R}^m : (1 + \eps)q_n^{ij}\left(\frac{\alpha}{2} \right) \le x_i-x_j \le q_n^{ij}\left(\frac{\alpha}{2} \right) \right\}  & A_{n,i}^{\alpha,\eps} &= \bigcap_{\substack{j= 1 \\ j \neq i}}^m \limits A^{\alpha,\eps}_{n,ij}\\
  A_{n,ij}^{\alpha} :=& \left\{x \in \mathbb{R}^m : x_i-x_j \le q_n^{ij}\left(\frac{\alpha}{2} \right) \right\} & A_{n,i}^\alpha &= \bigcap_{\substack{j= 1 \\ j \neq i}}^m \limits A^\alpha_{n,ij} \, .
\end{align*}
Then we have the following result

\begin{Theorem}[Convergence of Quantiles] \label{theo:quantileLLN}
  Under Assumptions \ref{as:BPSC}, \ref{as:finite}, \ref{as:local-manifold}, \ref{as:covar-rank} and \ref{as:Lipschitz2} there exists a sequence $\{\eps_n\}$ such that $\lim_{n\to\infty}\limits\eps_n=0$ and for $n \to \infty$
  \begin{align}
    \lim_{n \to \infty} \limits \sum_{i=1}^m \mathbb{P} \left(\Vn \in A_{n,i}^{\alpha} \right) - \alpha & \le 0 & \textnormal{for } & m \in \{2,3\}, \, \alpha \in [0,1] \label{eq:quant1}\\
    \lim_{n \to \infty} \limits \sum_{i=1}^m \mathbb{P} \left(\Vn \in A_{n,i}^{\alpha} \right) - \alpha & \le 0 & \textnormal{for } & m \ge 4, \, \alpha \textnormal{ sufficiently small} \label{eq:quant1b}\\
    \lim_{n \to \infty} \limits 2 \cdot \mathbb{P} \left(V_{n,n}^{*} \notin A_{n,i}^1 \, \middle| \, \Vn \in A_{n,i}^{\alpha,\eps_n} \right) - \alpha & \ge 0 \label{eq:quant2} & \textnormal{for } & m \ge 2, \, \alpha \in [0,1]  \, .
  \end{align}
  The probability in equation \eqref{eq:quant2} is understood as a probability over all samples and corresponding bootstrap samples. Thus it is a deterministic value.
\end{Theorem}

\begin{proof}
  The proof can be found in Appendix~\ref*{subsec-supp:quantiles}.
\end{proof}

\begin{Rm} \label{rm:quantiles}
  Theorem \ref{theo:quantileLLN} can be understood in terms of a hypothesis testing framework. Assume the null hypothesis that there are a number $m>1$ of global minimizers $\mu^i$. Then the sum $\sum_{i=1}^m \limits \mathbb{P} \left(\Vn \in A_{n,i}^\alpha \right)$ in equation \eqref{eq:quant1}, denoting the total probability that any $\Vn^i$ is smaller by at least the absolute value of the asymptotic $\frac{\alpha}{2}$ quantile than all other $\Vn^j$, can be understood as type I error probability.
  
  Equation \eqref{eq:quant2} then describes an asymptotically conservative (i.e. overestimating the error of first kind) estimator of the error of first kind. For $m=2$ the bounds are sharp and thus the test has size $\alpha$. The probability $\lim_{\eps \to 0} \limits \mathbb{P} \left(V_{n,n}^{*} \notin A_{n,i}^1 \, \middle| \, \Vn \in A_{n,i}^{\alpha,\eps} \right)$ is the probability that the bootstrap variance $V_{n,n}^{*,i}$ is not the smallest, but there is some $V_{n,n}^{*,k} < V_{n,n}^{*,i}$ despite conditioning on the fact that $\Vn^i$ is smaller by the $\frac{\alpha}{2}$ quantile than all other $\Vn^j$, which is the boundary of the region defining the type I error as explained above. This means that the bootstrap minimizer $\mu_{n,n}^*$ will not be close to the sample minimizer $\mun = \mun^i$ but to the local minimizer $\mun^k$.
\end{Rm}

\section{Hypothesis Test}

Many data analysis tools, including $T^2$-like tests and principal component analysis require the population descriptor $\mu$ to be unique. It is therefore useful to be able to determine with confidence whether the descriptor $\mu$ is unique. Wrongly assuming a unique $\mu$ leads to wrong assumptions about the distributions of test statistics and misplaced confidence in estimated values. This error is therefore more problematic than wrongly assuming a non-unique descriptor. Therefore, with Assumptions \ref{as:finite}, we design a test with null hypothesis and alternative as follows
\begin{align*}
  H_0 \, &: \, |E| \ge 2 & H_1 \, &: \, |E| = 1 \, .
\end{align*}

\begin{Rm}
  At first glance, it may seem as if the null hypothesis is considerably larger than the alternative. Indeed, the space describing tuples of multiple descriptors has higher dimension than the space describing a unique descriptor. In the present setting, where any finite number of descriptors in permissible in the null hypothesis, the space even has infinite dimension.
  
  However, these considerations of parameter space dimension are misleading if one is trying to determine, whether the null hypothesis or the alternative is larger. Since the test gives a result about a probability measure, the relevant question is whether probability measures with unique population descriptors are more abundant or more rare than those with non-unique population descriptors in the model under consideration.
  
  In the models considered here, one may intuitively expect that almost all probability measures have unique population descriptors. However, a rigorous formulation and proof of such a hypothesis, which would require a considerably more complex result in the vain of \cite{AM14} but for population descriptors, appears currently elusive. 
\end{Rm}

To approach the definition of the test statistic, we recall Remark \ref{rm:quantiles}. We can thus get an asymptotically conservative test statistic as follows. Determine the sample descriptor $\mun$, and simulate a large number $B$ of n-out-of-n bootstrap samples whose estimated descriptor values are denoted as $\big\{\mu^{*1}_{n,n}, \dots, \mu^{*B}_{n,n}\big\}$. Then, count the number $n_{\textnormal{far}}$ of $\mu^{*j}_{n,n}$ which are ``far from'' $\mun$. To make this last step precise, note that according to Assumption \ref{as:BPSC} the set $\big\{\mu^{*1}_{n,n}, \dots, \mu^{*B}_{n,n}\big\}$ always contains one cluster which includes $\mun$ but under the null hypothesis Assumption \ref{as:covar-rank} ensures that it contains additional clusters. Our goal is to identify additional clusters and count the number of points therein. To this end, we can define
\begin{align}
  d_j := d (\mun, \mu^{*j}_{n,n}) + X_j \label{eq:dists1}
\end{align}
with the metric $d$ of $\mathcal{P}$ and $X_j \iid N(0,n^{-1/2})$, even if $\mathcal{P}$ does not have a manifold structure.

Equation \eqref{eq:dists1} provides a way to map the bootstrap descriptors to $\mathbb{R}^+$. Next, we apply the multiscale test described by \cite{DW08} to identify slopes in the set $\big\{d_1, \dots, d_B\}$. This test is designed to estimate intervals containing rising and falling slopes of a one-dimensional probability density based on an i.i.d. sample from said density. The density in the present case is implicitly defined as the density of distances $d_j$ determined via bootstrap estimated from the data sample. The choice to project the bootstrap descriptors to one dimension and use the method by \cite{DW08} is not arbitrary, but it is made because this method provides uniform size and power properties, which we use in our own size and power considerations.

Figure \ref{fig:cluster-illustration} illustrates typical distributions of $d_j$. It shows that the cluster containing $\mun$ is characterized by the falling slope closest to $0$. However, this falling slope may be preceded by a rising slope due to the spherical volume element such that the first slope starting from $0$ is due to the cluster containing $\mun$. Thus, we look for the first interval containing a rising slope which begins at a higher value of $d_j$ than the earliest end point of an interval containing a falling slope. Call the first point of this rising slope interval $d_+$.

\begin{figure}[ht!]
  \centering
  \subcaptionbox{Dimension $p=1$}[0.45\textwidth]{\includegraphics[width=0.45\textwidth]{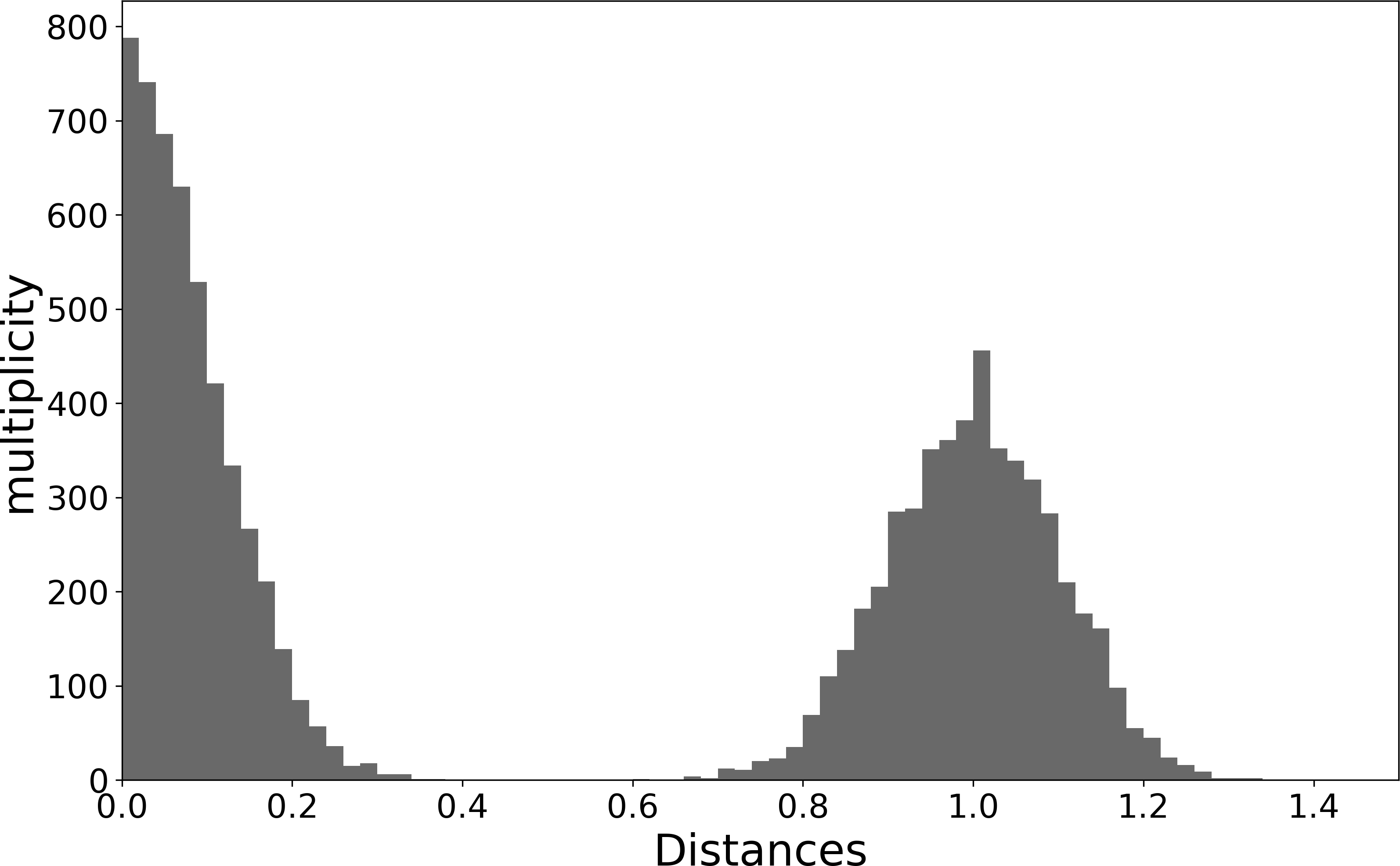}}
  \hspace*{0.02\textwidth}
  \subcaptionbox{Dimension $p=2$}[0.45\textwidth]{\includegraphics[width=0.45\textwidth]{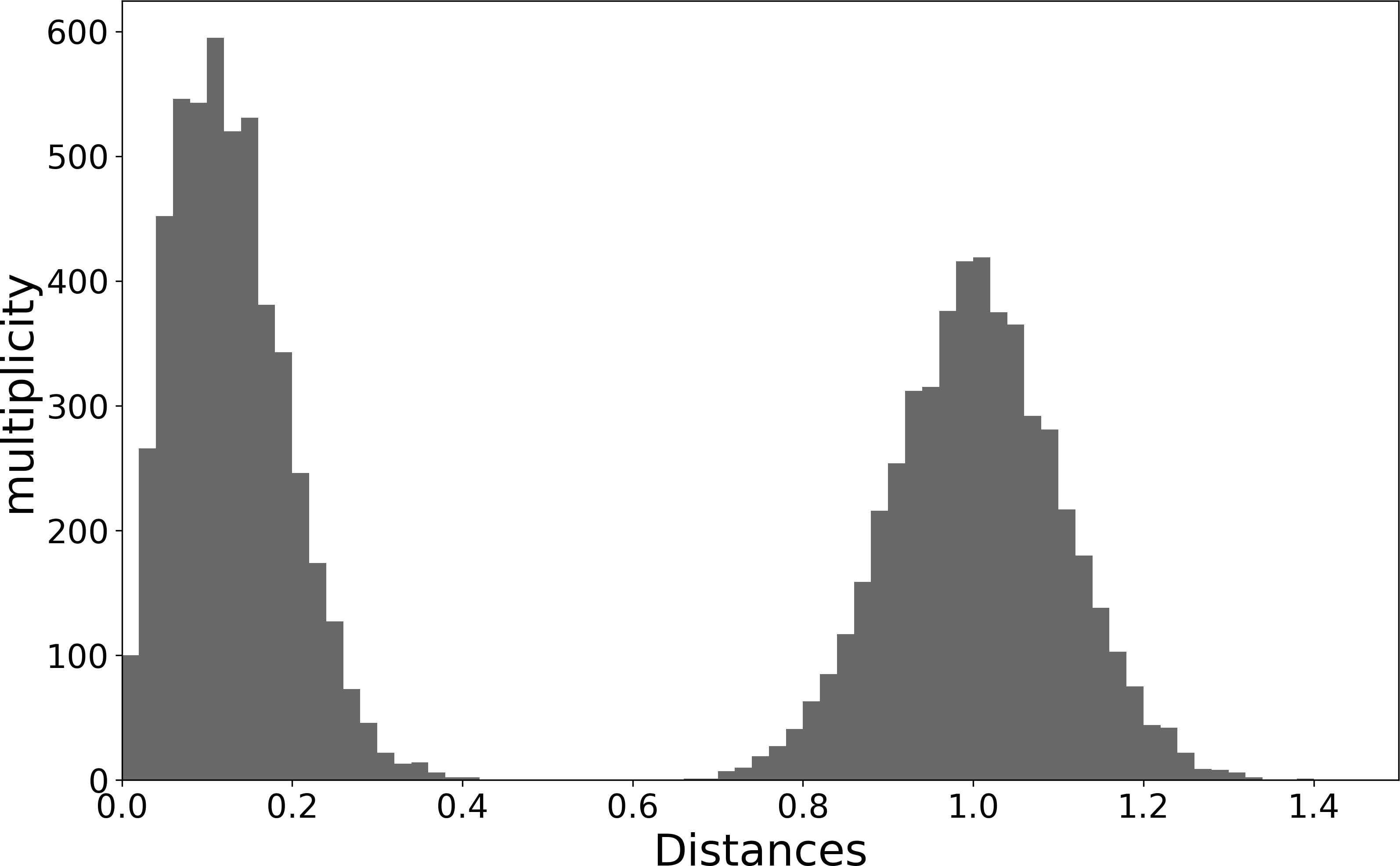}}
  \caption{Typical distributions of $d_j$ in case that the population Fr\'echet function has two minima. In the one-dimensional case, the cluster containing $\mun$ appears as a falling slope and the second cluster as a full mode. However, in the two-dimensional case, the cluster containing $\mun$ also appears as a full mode with rising and falling slope due to the volume element.
    \label{fig:cluster-illustration}}
\end{figure}

We define the hypothesis test for non-uniqueness as follows.

\begin{Test} \label{test:non-unique}$ $\\
  \textnormal{\textbf{Null hypothesis and alternative:}}
  \begin{itemize}
    \item[] $H_0 \, : \, |E| \ge 2$, $\quad H_1 \, : \, |E| = 1$.
  \end{itemize}
  \textnormal{\textbf{Test statistic:}}
  \begin{itemize}
    \item[] $T_d := \frac{2}{B} |\{d_j : d_j > d_+\}|$.
  \end{itemize}
  \textnormal{\textbf{Rejection region and p-value:}}
  \begin{itemize}
    \item[] Reject if $T_d < \alpha$, p-value $p_d = \min\left\{ 1, T_d \right\}$.
  \end{itemize}
\end{Test}

\begin{Rm}
  As remarked above, the result by \cite{AM14}, which shows almost sure uniqueness of Fr\'echet means under general conditions, may lead one to expect that a similar result could be shown for general M-estimators. However, in case that the sample descriptor is not unique, we propose preferring the sample descriptor which is located within the largest cluster of bootstrap descriptors. To this end, one could perform the test with each of the sample descriptor and pick the smallest p-value, since that should place the sample descriptor in the largest cluster of bootstrap descriptors. In all simulations and applications presented here, the sample descriptor was always unique.
\end{Rm}

In order to state asymptotic size and power results for this hypothesis test, we introduce the notation $\alpha_\textnormal{MS}$ for the level of the multiscale test by \cite{DW08} and $\alpha_\textnormal{ND}$ for the level of the non-unique descriptors test. For clarity of our argument, we here reproduce Theorem~4.1~(a) from \cite{DW08} and an equivariance property stated in \cite{DW08}~Section~4.

\begin{Theorem}[\cite{DW08}~Theorem~4.1~(a)] \label{theorem:multiscale-power}
  Let $\{I_n \}_{n \in \mathbb{N}}$ be a sequence of bounded intervals and consider sequences $\{\delta_n \in (0,1] \}_{n \in \mathbb{N}}$ and $\{ C_n > 0 \}_{n \in \mathbb{N}}$ such that $C_n \sqrt{1- \log \delta_n} \to \infty$. Then consider a sequence of densities $f_n$ whose cumulative distribution functions satisfy $F_n(I_n) = \delta_n$ and $H(f_n,I_n) := \inf_{x \in I_n} \frac{|I_n|^2f_n'(x)}{\sqrt{F_n(I_n)}} > (C_n + \sqrt{24}) \sqrt{1- \log \delta_n} n^{-1/2}$. Then, the probability of the test by \cite{DW08} to detect a rising slope within $I_n$ goes to $1$ for $n \to \infty$.
\end{Theorem}

\begin{Lem}[Equivariance property stated in \cite{DW08}~Section~4] \label{lemma:multiscale-equivar}
  Assume $\mu \in \mathbb{R}$ and $\sigma > 0$, $I_{\mu,\sigma} := [\mu, \mu+\sigma]$ and define $f_{\mu,\sigma}(x) := \sigma^{-1} f_{0,1}(\sigma^{-1} (x - \mu))$, then\\ $H(f_{0,1},I_{0,1}) := \inf_{x \in I_{0,1}} \frac{|I_{0,1}|^2 f_{0,1}'(x)}{\sqrt{F_{0,1}(I_{0,1})}} = H(f_{\mu,\sigma},I_{\mu,\sigma})$.
\end{Lem}

\begin{Lem} \label{lemma:density-convergence}
  Let $X$ be an $\mathbb{R}$-valued random variable with smooth Lebesgue probability density function $f$ and assume $f$, $f'$ and $f''$ are $L^1$-bounded. Let $\{Y_n\}_{n \in \mathbb{N}}$ be a sequence of $\mathbb{R}$-valued random variables with cumulative distribution functions $G_n$ which converge uniformly in distribution to a random variable $Y$ with continuous distribution function $G$. Then the random variable $Z_n := X + Y_n$
  \begin{enumerate}[(i)]
    \item converges in distribution to $Z := X + Y$,
    \item has a $C^1$ Lebesgue probability density function $j_n : x \mapsto \int_{-\infty}^{\infty} f'(x -y) G_n (y) dy$,
    \item $j_n$ converges to $j : x \mapsto \int_{-\infty}^{\infty} f'(x - y) G (y) dy$ in the $L^{\infty}$-norm,
    \item $j'_n$ converges to $j'$ in $L^{\infty}$-norm.
  \end{enumerate}
\end{Lem}

\begin{proof} \phantom{ }
  \begin{enumerate}[(i)]
    \item The distribution functions of $Z_n$ and $Z$ can be expressed as
    \begin{align*}
      J_n :&~ x \mapsto \int_{-\infty}^{\infty} f(x -y) G_n (y) dy & J :&~ x \mapsto \int_{-\infty}^{\infty} f(x -y) G (y) dy
    \end{align*}
    Thus we get
    \begin{align*}
      |J_n(x) - J(x)| =&~ \left|\int_{-\infty}^{\infty} f(x - y) (G_n(y) - G(y)) dy \right|\\
      \le&~ \int_{-\infty}^{\infty} f(x - y) \cdot |G_n(y) - G(y) | dy \le \sup_{y \in \mathbb{R}} |G_n(y) - G(y) | \to 0 \, ,
    \end{align*}
    Which proves the claim.
    \item One calculates $j_n(x) = J'_n(x) = \int_{-\infty}^{\infty} f'(x -y) G_n (y) dy$.
    \item Similarly to (i), we see
    \begin{align*}
      |j_n(x) - j(x)| =&~ \left|\int_{-\infty}^{\infty} f'(x - y) (G_n(y) - G(y)) dy \right|\\
      \le&~ \int_{-\infty}^{\infty} |f'(x - y)| \cdot |G_n(y) - G(y) | dy\\
      \le&~ \sup_{y \in \mathbb{R}} |G_n(y) - G(y) | \cdot \int_{-\infty}^{\infty} |f'(x - y)| dy \to 0 \, ,
    \end{align*}
    where we use the $L^1$-boundedness of $f'$.
    \item This follows analogously to (iii) by using the $L^1$-boundedness of $f''$. \qedhere
  \end{enumerate}
\end{proof}

In order to apply this last lemma, we need a regularity result for the asymptotic distribution of the distances $d (\mun, \mu^{*,i}_{n,n})$.

\begin{Lem} \label{lemma:continuous-distribution}
  Under Assumptions \ref{as:BPSC}, \ref{as:finite} with $m \ge 2$, \ref{as:local-manifold}, \ref{as:covar-rank}, \ref{as:Lipschitz2}, as well as Assumption \ref*{as:Taylor} in Appendix~\ref*{subsec-supp:aux-clt} holding for every $\mu^i$ with any $r = 2$, the distribution of $d (\mun, \mu^{*,i}_{n,n})$ conditioned on $\forall k \neq i : \, V^{*,i}_{n,n} < V^{*,k}_{n,n}$ converges uniformly in distribution to a random variable $D^i$ with continuous distribution function.
\end{Lem}

\begin{proof}
  Applying the generalized CLT from \cite{EH19}, we conclude that for every $i$, the distribution of $\log_{\mun^i}(\mu^{*,i}_{n,n})$ converges to the normal distribution. Conditioning on $\forall k \neq i : \, V^{*,i}_{n,n} < V^{*,k}_{n,n}$ preserves the continuity of the distribution since it cannot introduce point masses. Lastly, the continuity of the distribution function carries over to the $d (\mun, \mu^{*,i}_{n,n})$, since the asymptotic distribution of the $\mu^{*,i}_{n,n}$ is not concentrated on spheres around $\mun$. The uniform convergence then follows by Egorov's Theorem.
\end{proof}

Now, we state our theorem for the asymptotic test size, which is proved using the above Theorem and Lemmas.

\begin{Theorem}[Asymptotic Test Size] \label{theorem:size} $ $
  Under Assumptions \ref{as:BPSC}, \ref{as:finite} with $m \ge 2$, \ref{as:local-manifold}, \ref{as:covar-rank}, \ref{as:Lipschitz2}, as well as Assumption \ref*{as:Taylor} in Appendix~\ref*{subsec-supp:aux-clt} holding for every $\mu^i$ with any $r = 2$ and for sufficiently small $\alpha_\textnormal{ND}$, Test \ref{test:non-unique} satisfies in the limit $n \to \infty$, $B \to \infty$ and $\alpha_\textnormal{MS} \to 0$
  \begin{align*}
    \mathbb{P} (T_d \ge \alpha_\textnormal{ND})& \to C \ge 1 - \alpha_\textnormal{ND}
  \end{align*}
  with $C = 1 - \alpha_\textnormal{ND}$ for $m=2$.
\end{Theorem}
\begin{proof}
  From Theorem \ref{theo:quantileLLN} and Remark \ref{rm:quantiles} it becomes clear that, if $d_+$ can be determined asymptotically with probability $1$, these test statistics have asymptotically the correct size for $m=2$ and may be conservative otherwise. To show that the second mode in the distribution of $d_j$ is found with probability $1$ for $B \to \infty$ we apply Theorem \ref{theorem:multiscale-power} as well as Lemmas \ref{lemma:multiscale-equivar}, \ref{lemma:density-convergence}, \ref{lemma:continuous-distribution}.
  
  We now consider all minima of the loss function, corresponding to possible modes in the distribution of the $d_j$, separately. For each $i$, we thus condition on $\forall k \neq i : \, V^{*,i}_{n,n} < V^{*,k}_{n,n}$ in all of the following, i.e. we restrict attention to the points contributing to the $i$-th mode. Furthermore, we introduce the shorthand $\overline{d}^i := \mathbb{E} [d (\mun, \mu^{*,i}_{n,n})]$. The distribution of distances $d (\mun, \mu^{*,i}_{n,n})$ converges to $\overline{d}^i + \mathcal{O}_P(n^{-1/2})$ according to Lemma \ref{lemma:continuous-distribution}, so in particular $\overline{d}^i$ is well-defined. This allows us to define the sequence of random variables
  \begin{align*}
    Z^i_n := n^{1/2} \left( d (\mun, \mu^{*,i}_{n,n}) - \overline{d}^i + X \right)
  \end{align*}
  with $X \sim N(0,n^{-1/2})$. According to Lemma \ref{lemma:density-convergence}, each of the random variables $Z^i_n$ has a continuous first derivative of a density function, which converges uniformly to the first derivative of density function of the random variable $Z^i: = \lim_{n\to \infty} \limits Z^i_n$, which exists according to Lemma \ref{lemma:continuous-distribution}. For any interval $I^Z \subset \mathbb{R}$ the sequences $F^Z_n(I^Z)$ and $H(f_n,I^Z)$ then converge to constants. Therefore, using for example $C_n = n^{1/4}$ and $B \propto n$, the conditions for Theorem~\ref{theorem:multiscale-power} are satisfied by every interval $I^Z$ for which the density of $Z^i$ is strictly positive.
  
  Now, using Lemma \ref{lemma:multiscale-equivar}, the random variable
  \begin{align*}
    W^i_n :=  d (\mun, \mu^{*,i}_{n,n}) + X
  \end{align*}
  has the exact same value of the function $H$ introduced in Theorem~\ref{theorem:multiscale-power} as the random variable $Z^i_n$ under appropriate equivariant transformation of $I^Z$ to $I^W$ and $F^W_n(I^W) = F^Z_n(I^Z)$. In consequence, the conditions for Theorem~\ref{theorem:multiscale-power} are satisfied by every interval $I^W = [a,b]$ such that for the corresponding transformed interval $I^Z := [n^{1/2}(a - \overline{d}^i), n^{1/2}(b - \overline{d}^i)]$ the density of $Z^i$ is strictly positive.
  
  In consequence, the rising slope (and by completely analogous arguments also the falling slope) of every mode in the distribution of the $d_j$ will be identified by the multiscale test with asymptotic probability $1$ if one takes $\alpha_\textnormal{MS} \to 0$ sufficiently slowly such that the rejection region still goes to zero for $n \to \infty$, which means that the fraction of bootstrap descriptors with $d_j > d_+$ will asymptotically never be underestimated, leading to an at worst conservative test. The claim follows.
\end{proof}

The size result for Test~\ref{test:non-unique} builds on the power result for the multiscale test by \cite{DW08}. Conversely, a power result for Test~\ref{test:non-unique} builds on the asymptotic test size from \cite{DW08}. However, for Test~\ref{test:non-unique} to reject the null-hypothesis, the properly rescaled asymptotic distribution of $d_j$ must be unimodal. In analogy to the proof of Theorem~\ref{theorem:size}, define the random variable
\begin{align*}
  Z_n := n^{1/2} \left( d (\mun, \mu^{*}_{n,n}) - \mathbb{E} [d (\mun, \mu^{*}_{n,n})] + X \right) = n^{1/2} \left( \| \log_{\mun}(\mu^{*}_{n,n}) \| + X \right)
\end{align*}
with $X \sim N(0,n^{-1/2})$.

If the random variable $\| \log_{\mun}(\mu^{*}_{n,n}) \|$ to converge to a unimodal random variable, the probability that only one mode is found in the distribution of $Z_n$ to go to $1$ for $n \to \infty$. To this end we make an additional assumption.

\begin{As}\label{as:unimodal}
  Under Assumptions \ref{as:BPSC}, \ref{as:finite} with $m = 1$, \ref{as:local-manifold}, \ref{as:covar-rank}, \ref{as:Lipschitz2} Assumption \ref*{as:Taylor} in Appendix~\ref*{subsec-supp:aux-clt} with $r=2$ in the notation of Assumption \ref*{as:Taylor}, i.e. a unique non-smeary descriptor, the CLT covariance
  \begin{align*}
    \Sigma := \lim_{n \to \infty} \textnormal{Cov}[n^{1/2} \log_{\mu}(\mun)]
  \end{align*}
  has the property that for $Y \sim N(0,\Sigma)$ the random variable $|Y|$ has a unimodal density.
\end{As}

We conjecture that this assumption is tautological and the required unimodality holds for every positive semidefinite $\Sigma$. However, while this conjecture looks deceptively simple, it does not appear to be easily provable using previous results like \cite{Khinchine1938,OlshenSavage1970,Moschopoulos1985}.

\begin{Theorem}[Asymptotic Test Power] \label{theorem:power} $ $
  Under Assumptions \ref{as:BPSC}, \ref{as:finite} with $m = 1$, \ref{as:local-manifold}, \ref{as:covar-rank}, \ref{as:Lipschitz2}, \ref{as:unimodal}, Assumption \ref*{as:Taylor} in Appendix~\ref*{subsec-supp:aux-clt} with $r=2$ in the notation of Assumption \ref*{as:Taylor}, i.e. a unique non-smeary descriptor, and for $\alpha_\textnormal{ND} \in (0,1)$, Test~\ref{test:non-unique} satisfies in the limit $n \to \infty$, $B \to \infty$ and the level of the multiscale test $\alpha_\textnormal{MS} \to 0$
  \begin{align*}
    \mathbb{P} (T_d < \alpha_\textnormal{ND})& \to 1
  \end{align*}
\end{Theorem}

\begin{proof}
  According to Lemma \ref{lemma:density-convergence}, each of the random variables $Z_n$ has a continuous first derivative of a density function, which converges uniformly to the first derivative of density function of the random variable $Z: = \lim_{n\to \infty} \limits Z_n$, which exists according to the CLT. Now, using Lemma \ref{lemma:multiscale-equivar}, the random variable
  \begin{align*}
    W_n :=  d (\mun, \mu^{*}_{n,n}) + X
  \end{align*}
  has the exact same value of the function $H$ introduced in Theorem~\ref{theorem:multiscale-power} as the random variable $Z_n$ under appropriate equivariant transformation of $I^W = [a,b]$ to $I^Z := [n^{1/2}(a - \mathbb{E} [d (\mun, \mu^{*}_{n,n})]), n^{1/2}(b - \mathbb{E} [d (\mun, \mu^{*}_{n,n})])]$ whenever the density of $Z$ is strictly positive in $I^Z$. Thus it suffices to consider the random variables $Z_n$ and $Z$. Due to Assumption \ref{as:unimodal}, $Z$ is unimodal. Now, let $\zeta_n$ be the maximal deviation of the derivative of the density function of $Z_n$ from the derivative of the density function of $Z$. We use $B \propto n$  again and define for $\alpha_{\textnormal{MS}, n}$ the probability $p_n$ that a slope of steepness $\zeta_n$ is found at the level $\alpha_{\textnormal{MS}, n}$. Since $\zeta_n \to 0$, $p_n \to 0$ if $\alpha_{\textnormal{MS}, n}$ decreases slowly enough with $n$. The claim follows.
\end{proof}

\begin{Rm}
  Note that Theorems \ref{theorem:size} and \ref{theorem:power} exclude three cases, for which Test~\ref{test:non-unique} may have insufficient power.
  \begin{enumerate}
    \item If no asymptotic result can be shown, i.e. if either of the Assumptions \ref{as:BPSC}, \ref{as:local-manifold}, \ref{as:Lipschitz} or \ref{as:Lipschitz2} is violated, none of the considerations shown here are applicable.
    \item If the covariance of the M-variances does not have full rank, i.e. Assumption \ref{as:covar-rank} is violated, then there are vectors $v$ such that $v^T \textnormal{Cov}[\tau(0,X)] v = 0$, which undermines Theorem \ref{theo:varvarCLT}. In light of Theorem \ref{theo:quantileLLN}, one might weaken Assumption \ref{as:covar-rank}, and Theorem \ref{theo:varvarCLT} accordingly, to only require $\textnormal{Var}[\tau^i(0,X) - \tau^j(0,X)] > 0$ for any $i \neq j$. If this is violated, the M-variance of the $i$-th and $j$-th local Fr\'echet mean are always the same, so either both of them or neither are global sample descriptors.
    \item If \ref{as:finite} does not hold and the population descriptor is a continuous subset of parameter space, the test is not applicable. If an asymptotic distribution of descriptors on the population descriptor set can be shown, the properties of this distribution will determine the result of the test.
    \item If $m=1$ and Assumption \ref*{as:Taylor} in Appendix~\ref*{subsec-supp:aux-clt} holds with $r>2$, i.e. there is a unique smeary descriptor, our test is not guaranteed to reject non-uniqueness. In fact, on the circle and the torus the asymptotic distribution of smeary means was shown to be multimodal by \cite{HH15,Hun17} which means that the test will asymptotically not reject the null hypothesis. Furthermore, since in all cases of smeariness described so far, probability measures with smeary descriptors are always adjacent to measures with non-unique descriptors, the lack of power of the test in this case appears natural.
  \end{enumerate}
\end{Rm}

In the following section, we show simulations exemplifying the performance of Hypothesis Test \ref{test:non-unique} in terms of size and power for finite sample size $n$. In the final section we apply the test to various data examples to illustrate its range of applicability. The examples comprise the Fr\'echet mean on manifolds, non-linear regression and clustering. The interpretation of test results varies according to the data application and is highlighted in each example for clarity.

\section{Simulations}\label{section-supp:simulations}

In this section we provide numerical results for performance of Hypothesis Test \ref{test:non-unique} under the null hypothesis for $m=2$ and some alternatives close to the null. The simulations consider Fr\'echet means on $S^p$ with $p \in \{1,2,5\}$. For the tests in case of the null hypothesis, we use on $S^1$ a mixture of two wrapped normal distributions, denoted by $N_w$
\begin{align*}
  X \sim 0.5 N_{w}\left(0,\frac{\pi}{50}\right) + 0.5 N_{w}\left(\pi,\frac{\pi}{50}\right)
\end{align*}
which has means $\pm \frac{\pi}{2}$ and for higher dimensions $p>1$ we use, denoting the normal distribution by $N$
\begin{align*}
  Y &\sim N\left(0,\frac{10^{-6}}{p^2}\right) & Z &\sim U_{S^{p-1}} & X & \sim \begin{pmatrix} Y \\ \sqrt{1-Y^2} Z \end{pmatrix}
\end{align*}
which has means $(\pm 1, 0, \dots, 0)^T$, i.e. two opposite poles.

\begin{figure}[h!]
  \centering
  \includegraphics[width=0.4\textwidth]{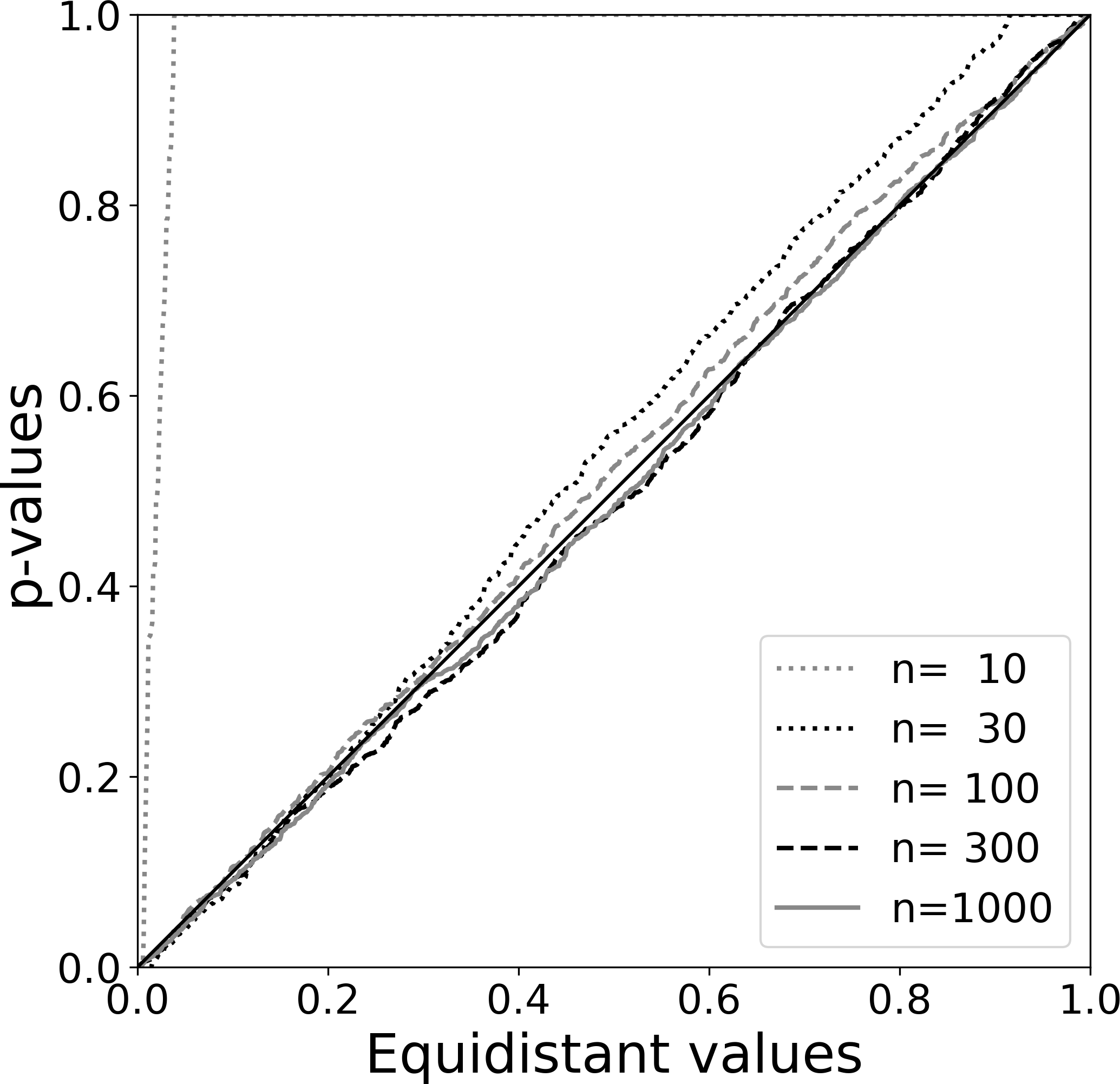}
  \caption{$T=1\,000$ sorted p-values for the null hypothesis in dimension $p=1$ plotted against equidistant values for sample sizes $n = 10,30,100,300,1000$. $B=10\,000$ were used for the test. While the test is noticeably conservative for $n=10$ it attains the theoretical size very well for $n \ge 30$.\label{fig:test-size-1d}}
\end{figure}

\begin{figure}[h!]
  \centering
  \subcaptionbox{$p_d$ for $p=2$}[0.4\textwidth]{\includegraphics[width=0.33\textwidth]{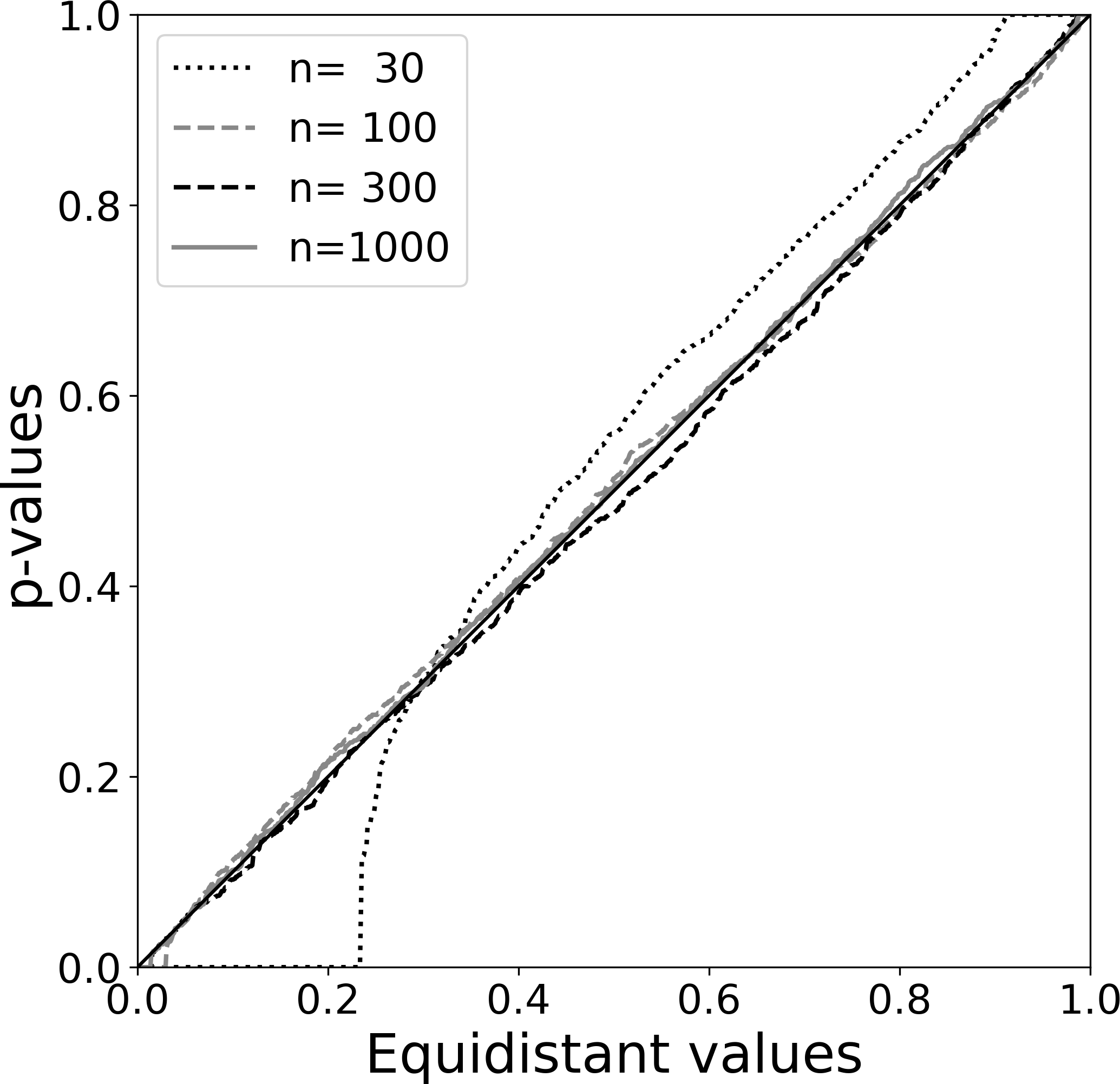}}
  \hspace*{0.02\textwidth}
  \subcaptionbox{$p_d$ for $p=5$}[0.4\textwidth]{\includegraphics[width=0.33\textwidth]{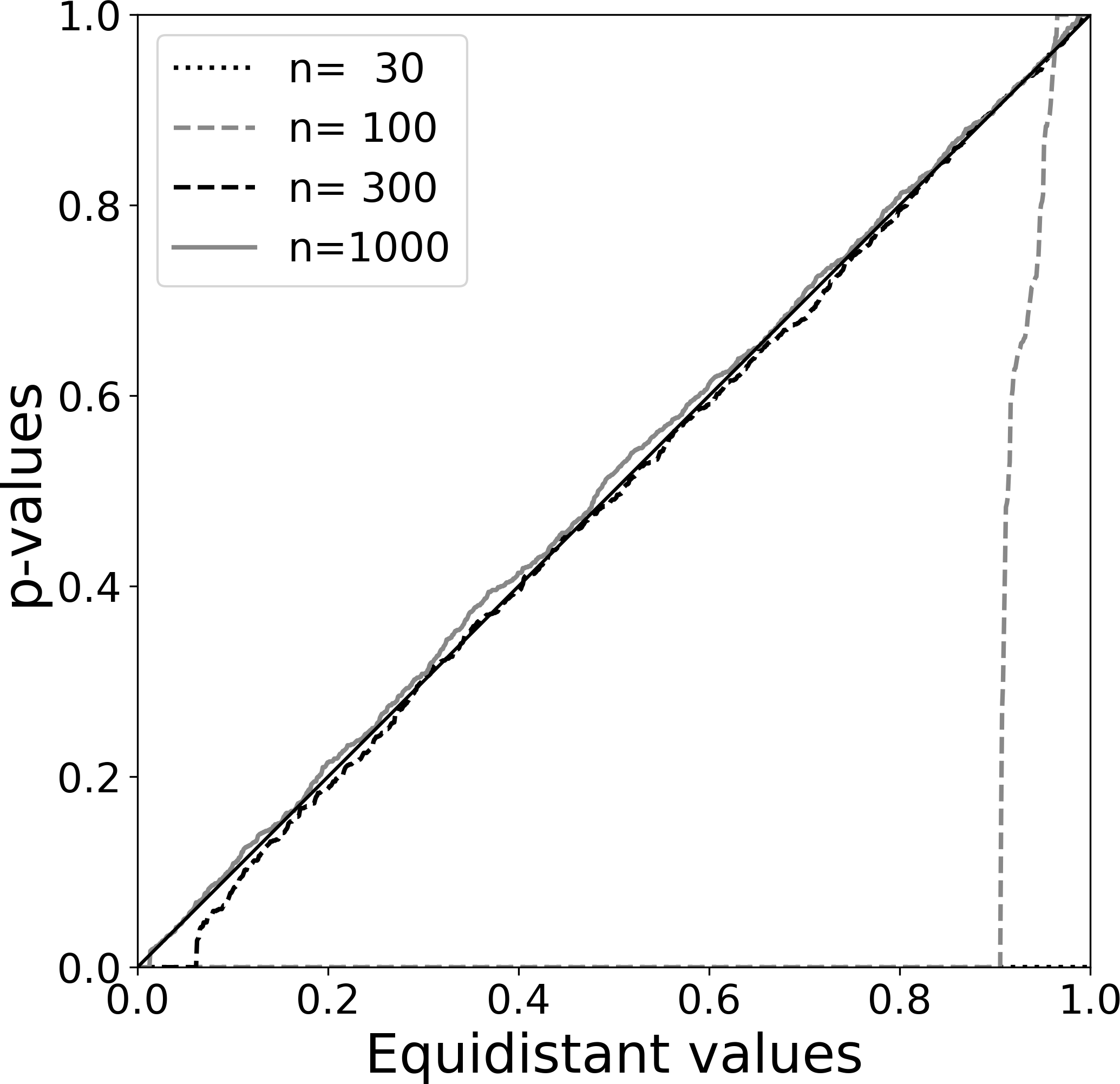}}
  \caption{$T=1\,000$ sorted p-values for the null hypothesis in dimensions $p=2,5$ plotted against equidistant values for sample sizes $n = 30,100,300,1000$. $B=2\,000$ were used for the test. For $p=2$ the test appears to achieve the correct size for $n \ge 100$, while for $p=5$ it is preferable to have $n \ge 1000$.
    \label{fig:test-size-23d}}
\end{figure}

For $p=1$ we use $B=10\,000$ and otherwise $B=2\,000$. We then perform tests on $T=1\,000$ samples for each sample size and display the resulting sorted p-valued in a pp-plot in Figures \ref{fig:test-size-1d}, and \ref{fig:test-size-23d}. The test approaches the size predicted by the asymptotics for large $n$, and in the one-dimensional situation the test achieves very good performance already for $n=30$.

Since the multiscale test by \cite{DW08} works best if modes consist of $\gsim 50$ points, p-values below $0.05$ are often underestimated to be $0$ when $B=1\,000$, therefore it is advisable to use $B=10\,000$ in data applications. Furthermore, we observe that for small sample sizes and $p>1$ many sample means are close to the equator, $e_1^T \mun \approx 0$. This effect becomes more pronounced with increasing dimension $p$. Therefore, the tests require increasing sample size with increasing dimension in our setting to avoid being too liberal. In our simulations $n \gsim 300$ is sufficient to achieve good performance. In other scenarios dimension dependence can vary but in general greater sample sizes are needed in higher dimension.

\begin{figure}[h!]
  \centering
  \includegraphics[width=0.6\textwidth]{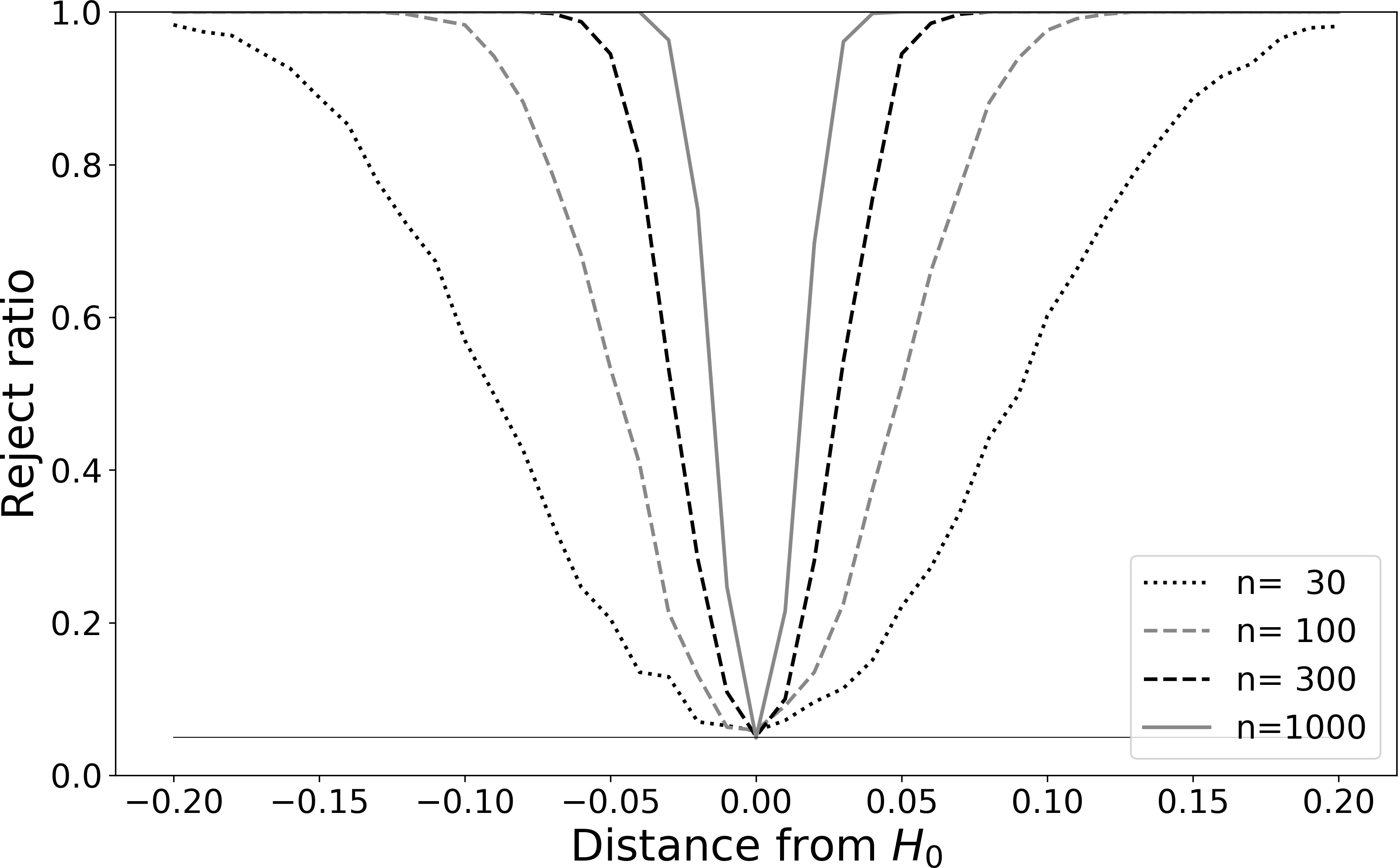}
  \caption{Ratio of rejecting tests for dimension $p=1$ using $T=1\,000$ tests to level $\alpha = 0.05$ with $B=2\,000$ bootstrap samples each for distributions close to the null hypothesis as defined in Equation \eqref{eq:test-alt-dists}. We see that the power increases rapidly with $n$ as expected and the correct size is attained at the null hypothesis for large enough $n$ with good results already for $n=30$.\label{fig:test-power-1d}}
\end{figure}

Next, we investigate the power of the test in dimension $p=1$ for distributions close to the null hypothesis. We use on $S^1$ a mixture of two wrapped normal distributions, denoted by $N_w$
\begin{align}
  X \sim 0.5 N_{w}\left(a,\frac{\pi}{50}\right) + 0.5 N_{w}\left(\pi-a,\frac{\pi}{50}\right) \label{eq:test-alt-dists}
\end{align}
where $a\in[-0.1,0.1]$. Note that for $a>0$ the mean is unique $\mu = \pi/2$, for $a<0$ the mean is unique $\mu = -\pi/2$ and for $a=0$ the mean is non-unique $\mu = \pm \pi/2$,.

The numerical results for the power of the test are displayed in Figure \ref{fig:test-power-1d}. The test can clearly discern the tested alternative from the null in the tested family of distributions. The power of the test increases appreciably with $n$, as one would expect. These numerical results illustrate that test is suited for its purpose and is for low dimensions already reliably applicable for fairly low $n$.

\section{Application to Data}

\subsection{Wind Directions in Rome}

As a first data example, we consider a data set consisting of hourly wind directions from the years 2000 to 2019 in Rome, Italy, from \cite{meteoblue}. As in \cite{HEH19}, where similar data sets were considered for Basel, Switzerland, and G\"ottingen, Germany, we consider the intrinsic mean of wind directions for every day. Additional examples for applications to centroid clustering are given in Appendix~\ref*{subsec-supp:turtles}.

\begin{table}[!ht]
  \centering
  \caption{Ratio of days for each year on which the test for non-uniqueness of the intrinsic mean of wind directions did not reject the null hypothesis of a non-unique mean at level $\alpha = 0.05$.}
  \begin{tabular}{l*{10}{|c}}
    year          & 2000   & 2001   & 2002   & 2003   & 2004   & 2005   & 2006   & 2007   & 2008   & 2009\\
    \hline
    $ p \ge 0.05$ & 11.2\% & 12.1\% & 10.4\% & 12.9\% & 11.2\% & 12.1\% & 10.1\% & 13.4\% & 12.0\% & 10.7\%
  \end{tabular}\\
  
  \begin{tabular}{l*{10}{|c}}
    year          & 2010   & 2011   & 2012             & 2013   & 2014   & 2015   & 2016   & 2017   & 2018   & 2019\\
    \hline
    $ p \ge 0.05$ & 10.7\% & 11.2\% & \phantom{1}9.0\% & 14.0\% & 12.1\% & 11.8\% & 11.2\% & 10.7\% & 11.0\% & 12.3\%
  \end{tabular}
  \label{tab:wind-rome}
\end{table}

The relative ratios of days for each year on which non-uniqueness of the intrinsic mean of the wind directions could not be ruled out are displayed in Table~\ref{tab:wind-rome}. As this potential non-uniqueness affected more than $10\%$ of days in all but one year, the resulting distributions of daily means is qualitatively very distinct from the distribution of hourly wind directions. As a result, if one considers bootstrap intrinsic means of hourly wind directions in comparison to bootstrap means of daily mean wind directions, one get starkly different results, as shown in Figure~\ref{fig:wind}.

\begin{figure}[h!]
  \centering
  \subcaptionbox{bootstrap means 2001}[0.30\textwidth]{\includegraphics[width=0.3\textwidth]{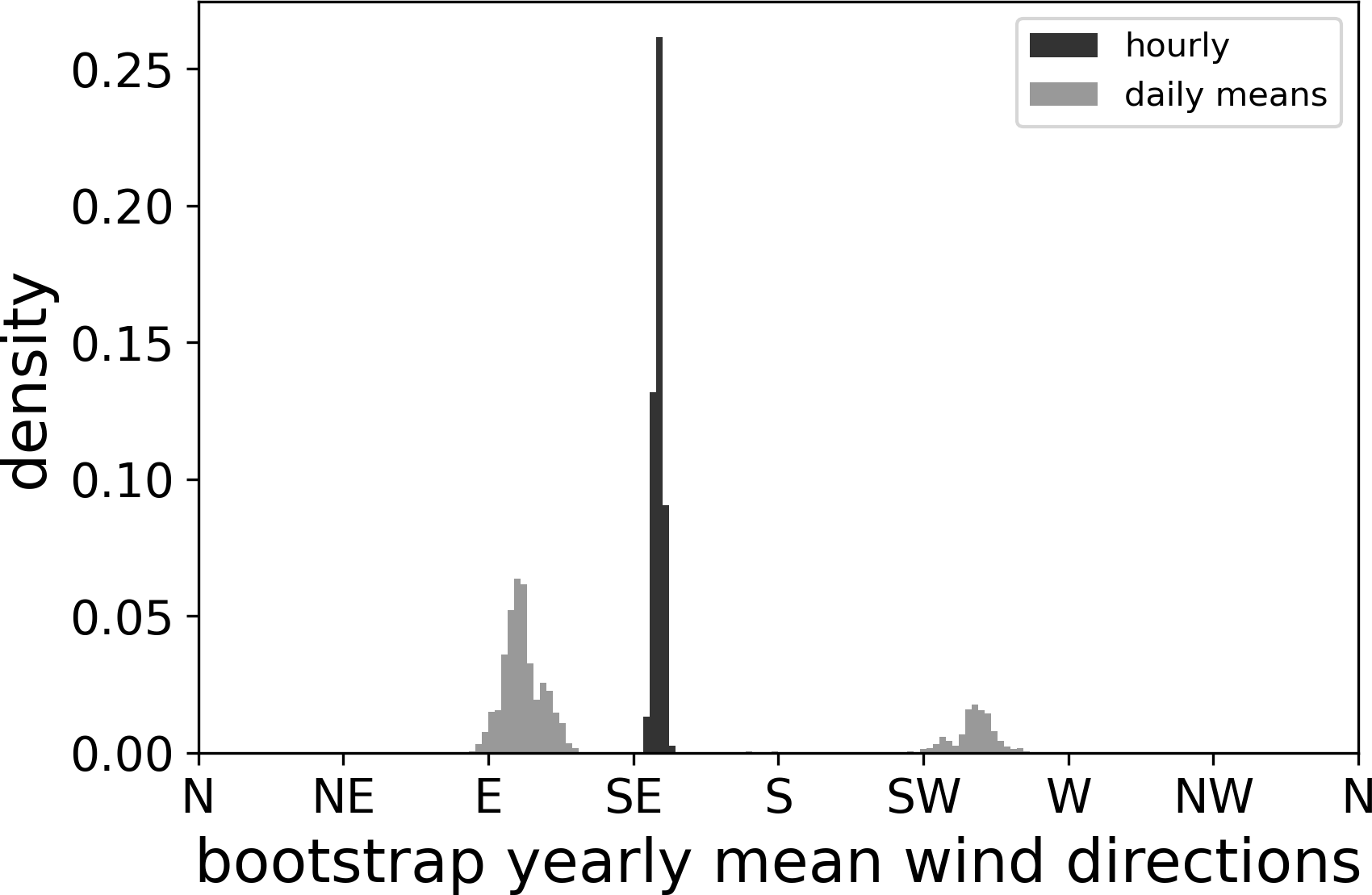}}
  \hspace*{0.02\textwidth}
  \subcaptionbox{bootstrap means 2019}[0.30\textwidth]{\includegraphics[width=0.3\textwidth]{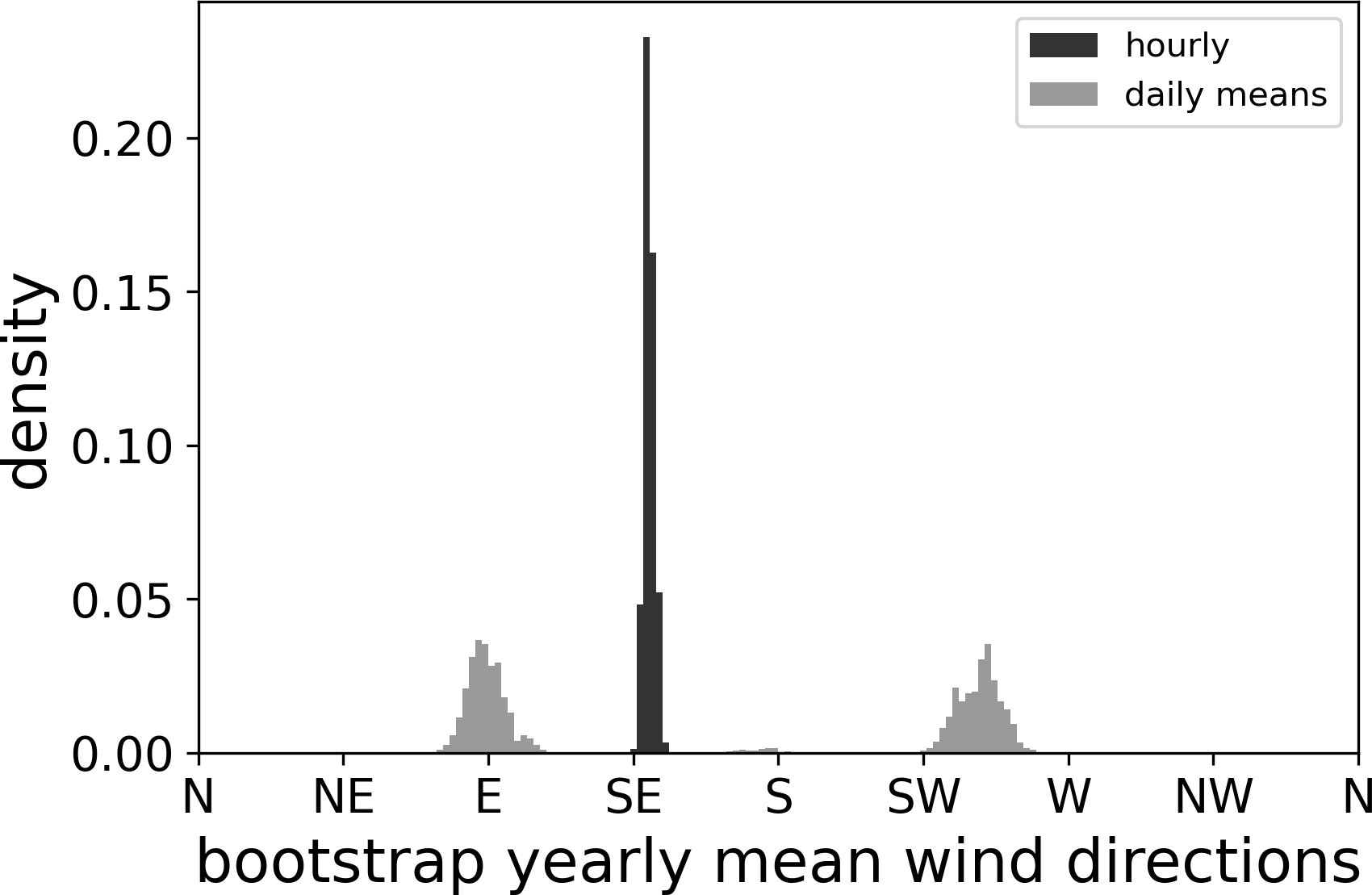}}
  \hspace*{0.02\textwidth}
  \subcaptionbox{bootstrap means 2011}[0.30\textwidth]{\includegraphics[width=0.3\textwidth]{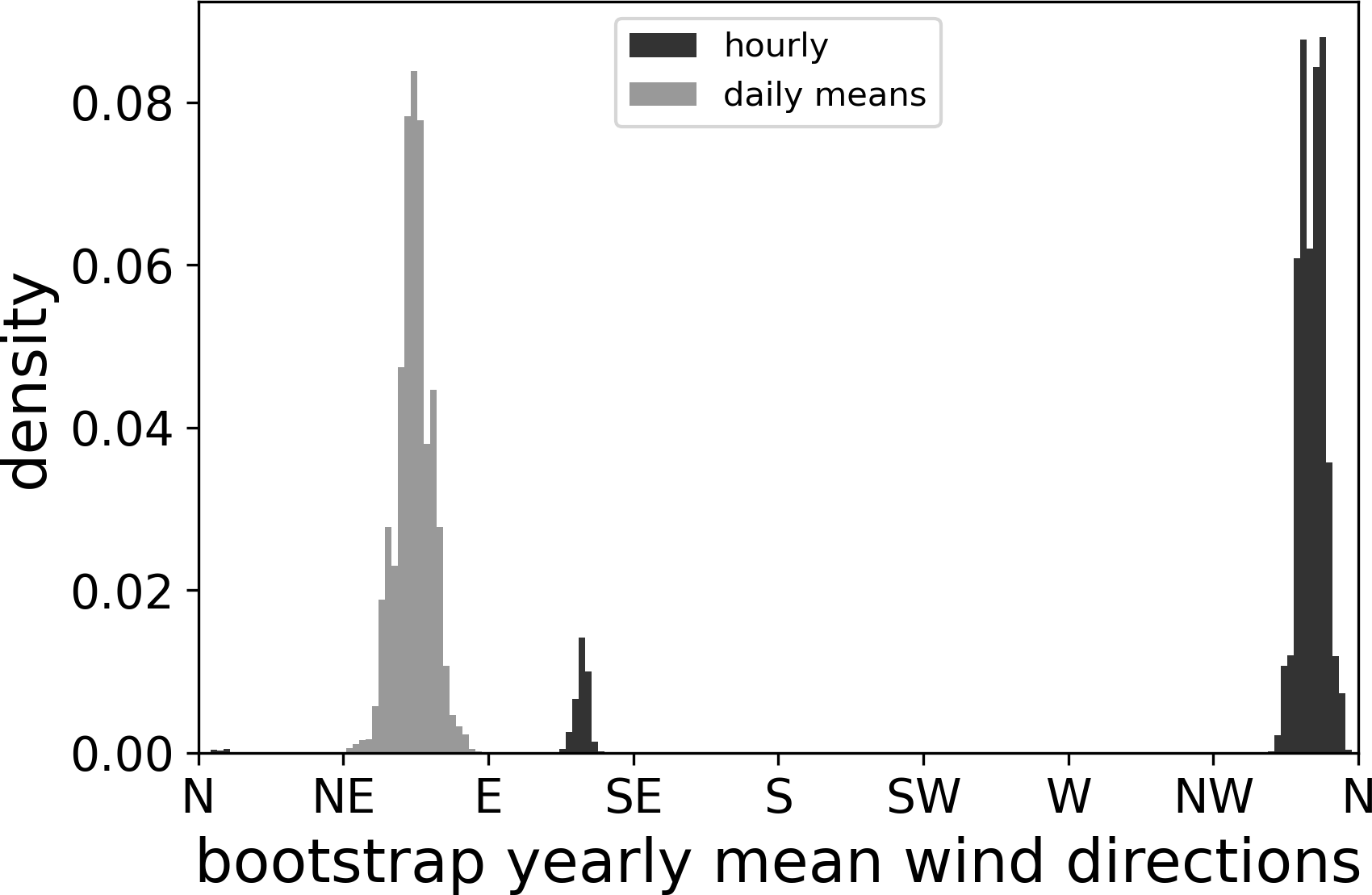}}
  \caption{Histograms of bootstrap means of hourly wind directions (dark) and daily mean wind directions (light). The years 2001 and 2019 were picked as they are typical years with a unimodal distribution for the hourly wind directions and a bimodal distribution for daily means. The year 2011 is the only case, in which the bootstrap means of the hourly wind directions exhibit more than one mode.\label{fig:wind}}
\end{figure}

\begin{Int}
  Similar data sets were investigated in \cite{HEH19} for Basel and G\"ottingen. In those cases, the daily means were generally unique and the distribution of daily mean wind directions was similar to the distribution of hourly wind directions. The data sets of daily means exhibited appreciable finite sample smeariness of the mean and nonuniqueness of the mean cannot be ruled out for most of the years. In the present data set for Rome, we see a different picture, where already on the scale of single days non-uniqueness of the mean direction is common. A closer examination of the data set, shown in Appendix~\ref*{subsec-supp:wind}, indicates that the wind in Rome has a pronounced daily pattern, as is typical for coastal regions, coming from the northeast (from land to sea) in the night and morning and from the southwest (from sea to land) in the afternoon and evening. This frequently leads to a bimodal distribution of wind directions over the day, such that nonuniqueness of the mean cannot be ruled out.
\end{Int}

\subsection{Platelets Spreading on a Substrate}

The second data set we are considering was first presented in \cite{PEK2019}. The data describe the total length $\ell$ of actin stress-fiber-like structures in blood platelets over time $t$ during spreading on a substrate. The expected behavior is an initial growth of fiber structures which is completed after a certain growth time. This leads to a regression problem with a more involved M-estimator. Consider two-dimensional data $x := (t,\ell)$, four parameters $\theta := (\theta_1, \theta_2, \theta_3, \theta_4) \in \mathbb{R} \times \mathbb{R}^+ \times \mathbb{R} \times \mathbb{R}$ and the loss function
\begin{align*}
  f(t, \theta) :=& \max \left\{0, (\theta_3 + \theta_4 t) \cdot \left(1 - \exp\left(\theta_1 - \frac{t}{\theta_2}\right)\right)\right\}\\
  \rho(x, \theta) :=& \left( \ell - f(t, \theta) \right)^2
\end{align*}
which amounts to a standard quadratic loss regression model with a somewhat unusual four-parametric regression function $f$. However, since this loss function is treated in an M-estimator setting, the theory and test laid out in this paper can be applied in this setting. The parameter $\theta_2$ describes the time scale of fiber growth and is therefore the parameter of greatest interest.

Since imaging on such small scales and such high time resolution is very challenging, images frequently feature low brightness and some blurring, which in turn poses a challenge for image analysis. The semi-automated line detection using the Filament Sensor \cite{EWGHR15} comes with some amount of artifacts which lead to a high variance of $\ell$. As a results, for some of the platelets several local minima of the sample M-variance exist and our test with $B=10\,000$ shows that these can sometimes not be rejected as valid minima of population M-variance.

The distance between two sets of parameters is defined via $L^2$ distance of the curves within the experimental time $[0,T]$, using a simple sum approximation to reduce calculation time.
\begin{align*}
  d(\theta, \widetilde{\theta}) := T^2 \int_0^T \limits | f(t, \theta) - f(t, \widetilde{\theta})|^2 dt \approx \sum_{i=1}^n | f(t_i, \theta) - f(t_i, \widetilde{\theta})|^2
\end{align*}

\begin{figure}[h!]
  \centering
  \subcaptionbox{Test does not reject with $p_d = 0.2270$}[0.45\textwidth]{\includegraphics[width=0.45\textwidth]{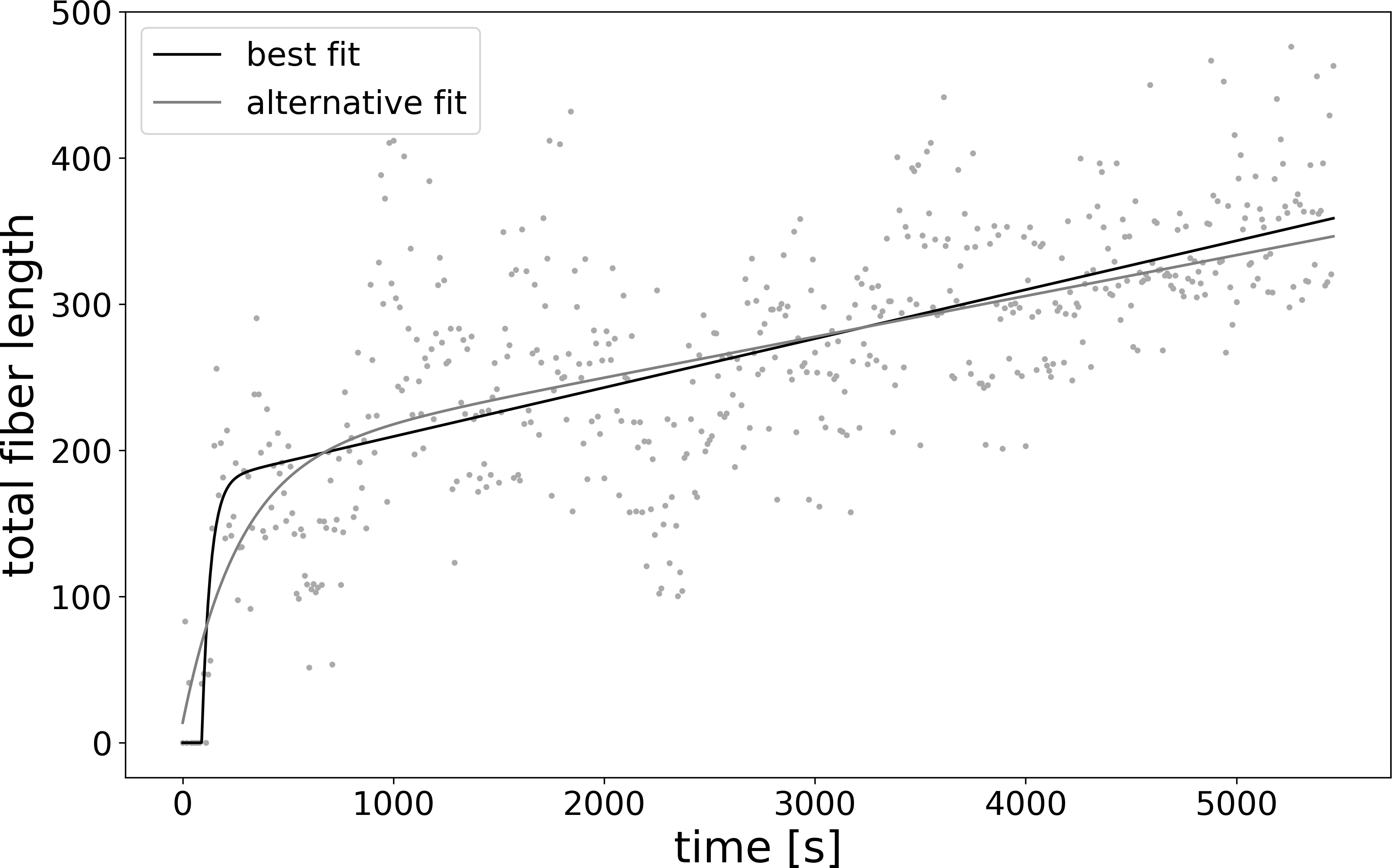}}
  \hspace*{0.02\textwidth}
  \subcaptionbox{Test rejects with $p_d = 0.0438$}[0.45\textwidth]{\includegraphics[width=0.45\textwidth]{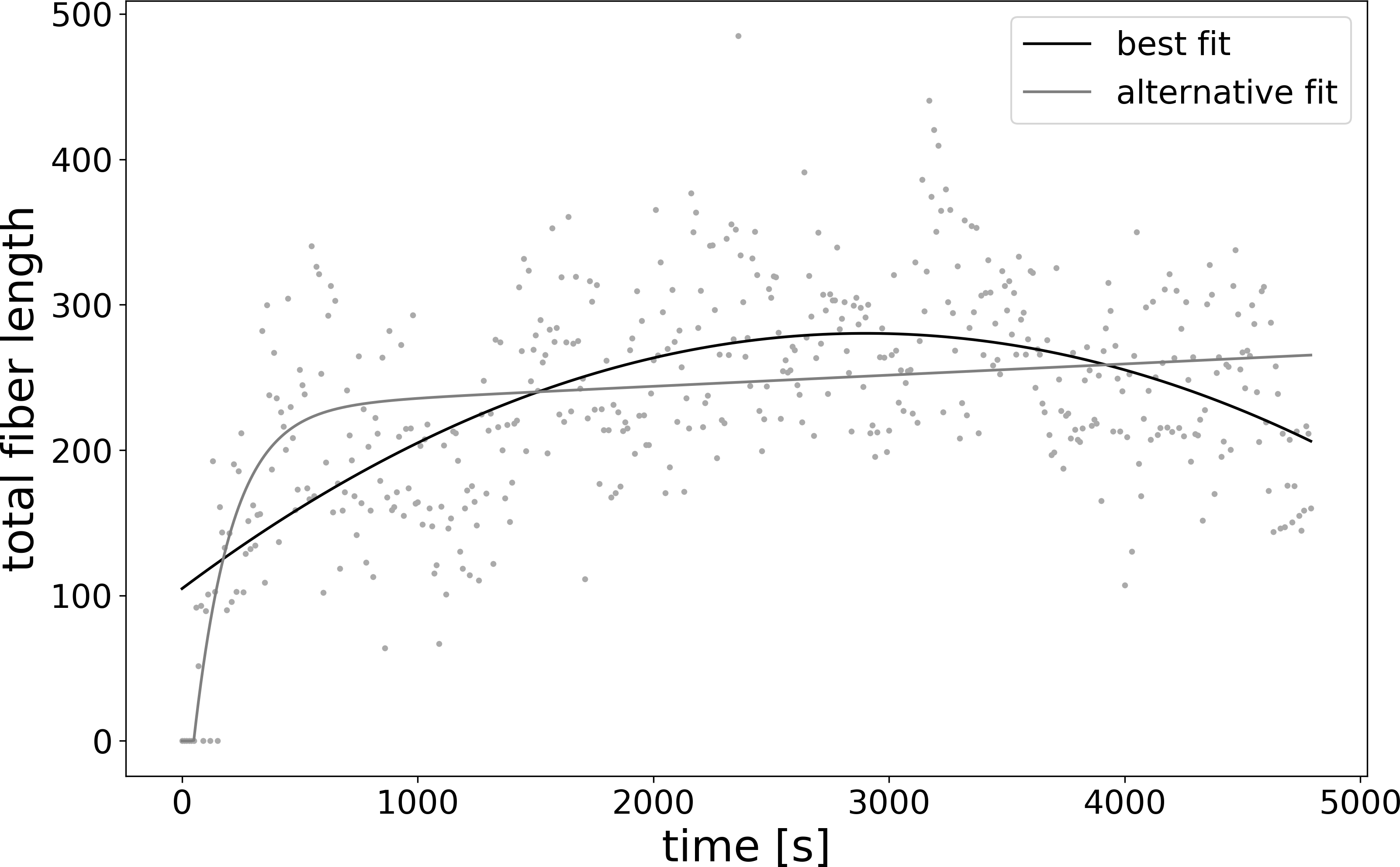}}
  \caption{Plots for two of the time series. Panel~(a) displays a typical example, where the test does not reject the possibility for a second valid optimal fit. The fits are very similar for long times but differ more strongly at short times. Panel~(b) displays an extreme case, in which the best fit leads to extreme parameter values and an unusual curve shape. The test rejects in this case because the secondary cluster is small, however the alternative fit looks much closer to an expected result. In this case, it can be beneficial to restrict the parameter space. In \cite{PEK2019} the restriction $\theta_1 \theta_2 \ge -30$ was used, since the begin of spreading, marked by the initial departure of the curve from $0$, must be within the movie or very briefly before due to the experimental setup.\label{fig:platelets}}
\end{figure}

In Figure \ref{fig:platelets} we show two examples. Panel~(a) displays a typical example where the fact that the test does not reject indicates a non-unique optimum.

\begin{Int}
  In the example in Figure~\ref{fig:platelets}, panel~(b) the best fit violates assumptions of the experiment, namely the fact that the movies are started once the platelet begins spreading on the substrate and thus the begin of spreading can at most lie a few seconds before the start of the time series. Therefore, this fit suggests that restricting parameter space to $\theta_1 \theta_2 \ge -30$ as done by \cite{PEK2019} is necessary here.
  
  The results of \cite{PEK2019} concerning the time scale of actin fiber formation remains mostly unaffected by the result presented here, since only five platelet in the full data set exhibit non-uniqueness in their fits and only for two this remains true if the parameter space is restricted to $\theta_1 \theta_2 \ge -30$. In these two cases, the two candidates for the optimum have very similar time constants $\theta_2$.
\end{Int}

\subsection{Clustering of RNA Residues}

For the third example, we consider a large RNA data set carefully selected for high experimental X-ray precision ($0.3$ nanometers) by \cite{Duarte1998}, updated by \cite{Wadley2007} and analyzed by them and others, for example, \cite{Murray2003,Richardson2008}. Additional examples for applications to centroid clustering are given in Appendix~\ref*{subsec-supp:gauss-mixture}.

Specifically, we revisit the clustering approach presented in \cite{WEMH2023} and pick out four clusters of moderate sizes, namely $182$, $64$, $53$, and $49$ points, respectively, adding up to $n = 348$ points. The data space is $(S^1)^{\times 7}$, i.e. a seven dimensional flat torus defined by the 7 dihedral angles describing a suite. We consider truncated multivariate isotropic normal distributions for a mixture model clustering, both with individual variances per cluster $(\sigma_j)_j$ and with same variance $\sigma_0$ for all clusters. The latter case corresponds to a ``fuzzy'' k-means clustering. The $p$-values for different numbers of clusters and the two approaches are listed in Table~\ref{tab:rna-clusters}.

\begin{table}[!ht]
  \centering
  \caption{Results for the $p$-values for the clustering results with different numbers of clusters and either the same variance $\sigma_0$ or individual cluster variances $(\sigma_j)_j$. The additional flexibility afforded by individual cluster variance leads to less reliability of the cluster results.}
  \begin{tabular}{l*{5}{|c}}
    clusters                        & $k=2$ & $k=3$  & $k=4$ & $k=5$  & $k=6$\\
    \hline
    $p$-value, same variance        & 0.0   & 0.0142 & 0.0   & 0.2558 & 0.255\\
    $p$-value, individual variances & 0.0   & 0.3244 & 0.0   & 1.0    & 1.0 
  \end{tabular}
  \label{tab:rna-clusters}
\end{table}

\begin{figure}[h!]
  \centering
  \subcaptionbox{cluster labels}[0.45\textwidth]{\includegraphics[width=0.45\textwidth]{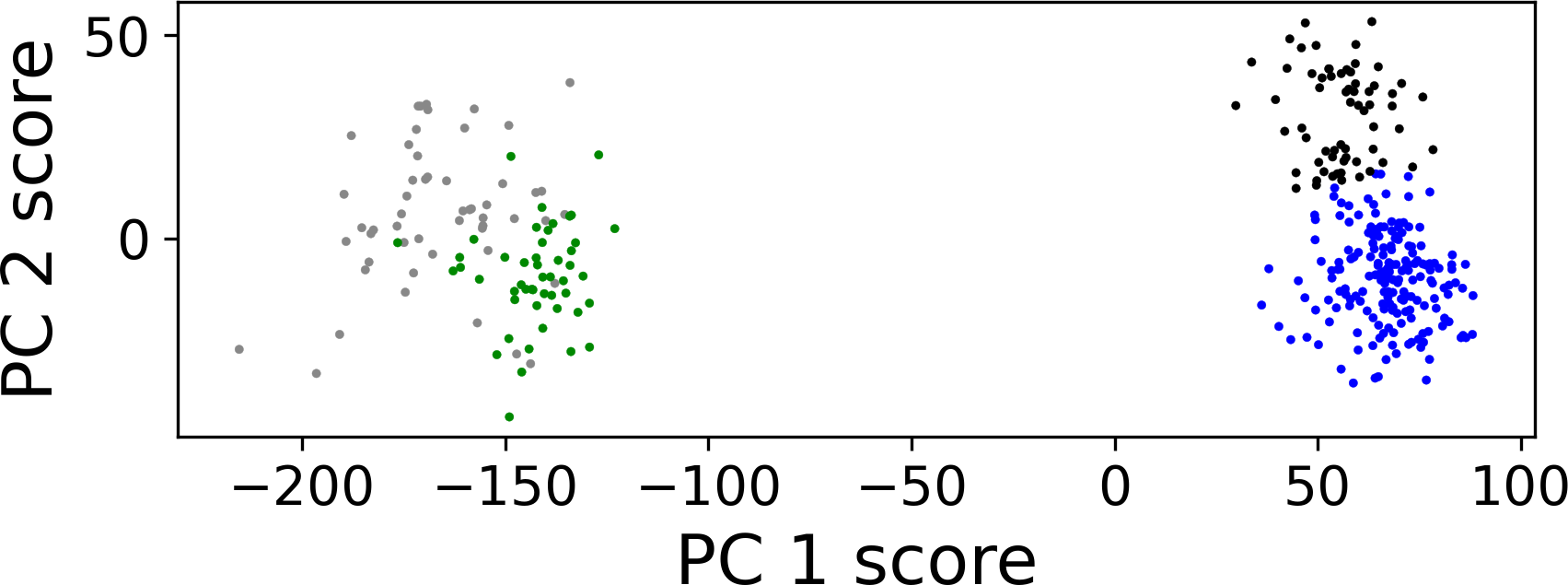}}\\
  \subcaptionbox{$k=4$, same variance}[0.45\textwidth]{\includegraphics[width=0.45\textwidth]{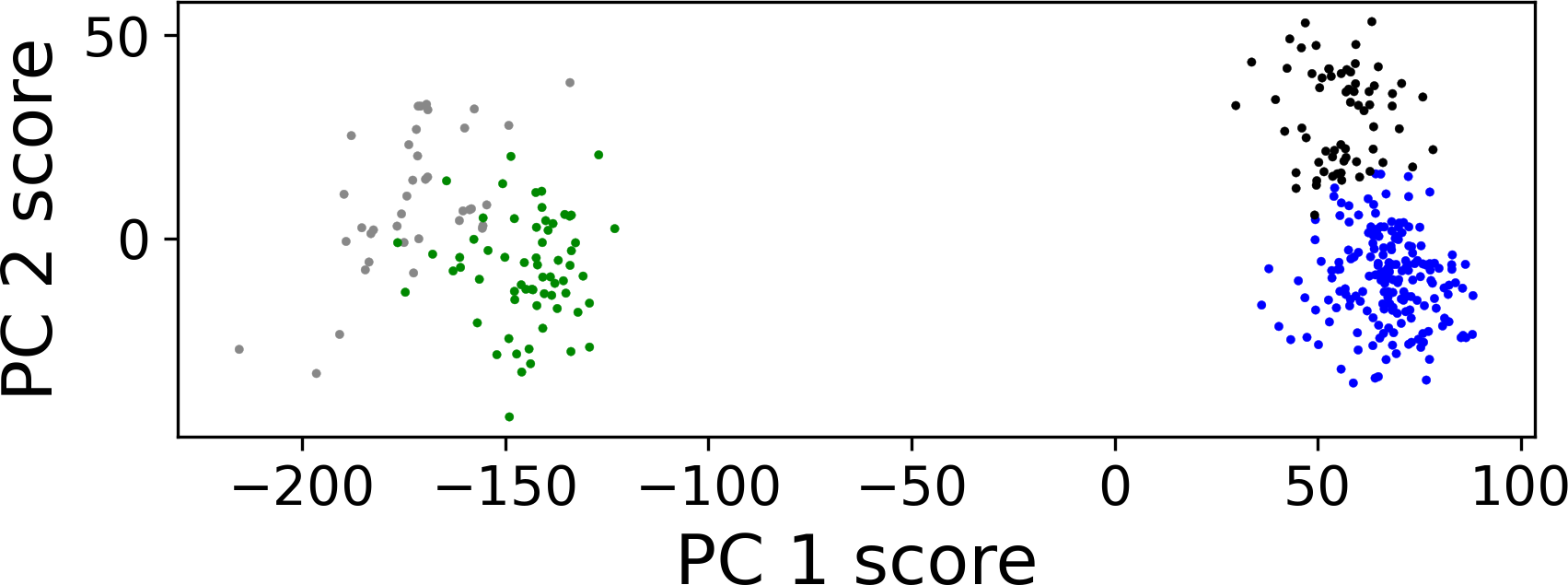}}
  \hspace*{0.02\textwidth}
  \subcaptionbox{$k=4$, individual variances}[0.45\textwidth]{\includegraphics[width=0.45\textwidth]{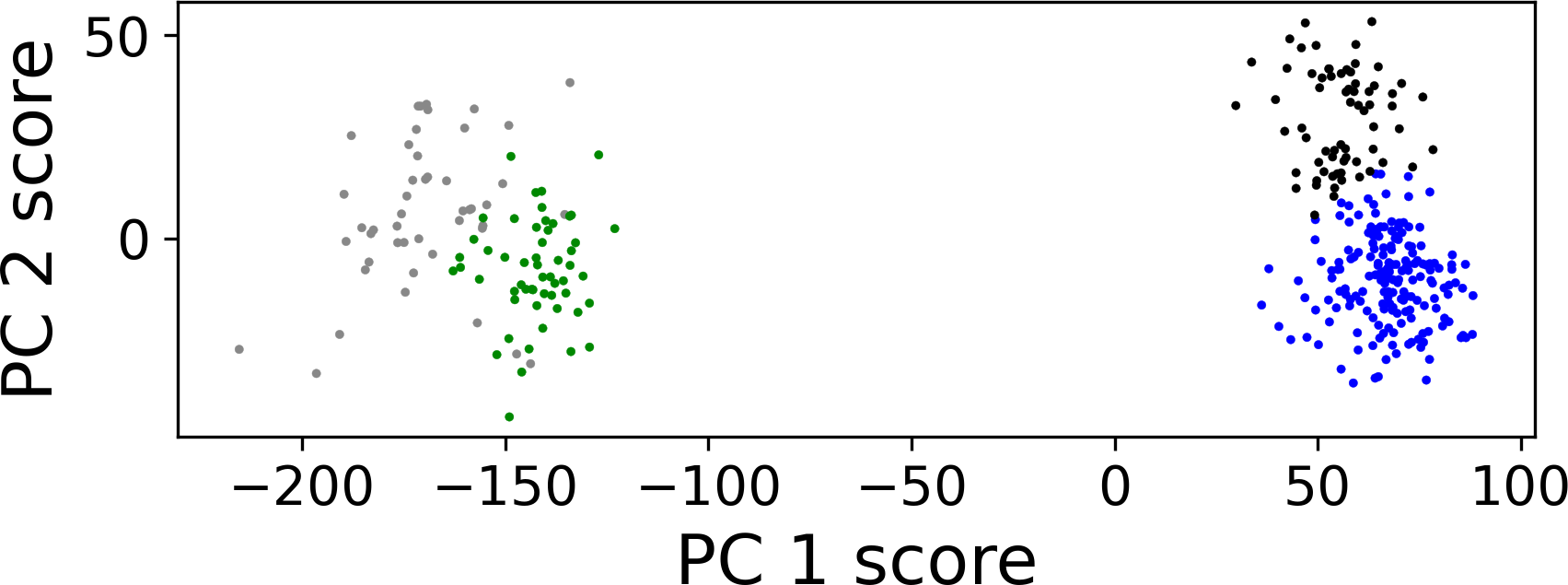}}
  \caption{Scatter plots of the first two tangent principal components of the RNA data. The colors in panel~(a) highlight the three species, while the colors in panels~(b) and (c) show the four clusters in the four-cluster segmentation. The clustering results are in very good agreement with the true labels and the test results indicate that the clustering results are reliable. 
    \label{fig:rna}}
\end{figure}

One can see that for $k=2$ and $k=4$ mixture components the clustering result is unique for either variant of the clustering algorithm. This is due to the fact that the four clusters are arranged in two groups of two clusters, as illustrated in Figure~\ref{fig:rna}. If one considers $3$, $5$, or $6$ mixture components, the null hypothesis of non-uniqueness cannot be rejected, except for same $\sigma_0$ and $k=3$. 

\begin{Int}
  Taken together, these results show that the test for non-uniqueness accurately predicts that the data set consists of four clusters. Remarkably, $3$ clusters yield a less stable fit than $2$ clusters and individual cluster variances generally lead to less stable fits than requiring the same variance for all clusters. This shows that adding more parameters does not necessarily improve the quality of the fit in terms of uniqueness of the results. Unlike in model selection criteria like BIC and AIC, additional parameters are not directly penalized in the test for non-uniqueness. Instead, models which do not achieve unique results with confidence can be considered inherently unreliable for the given application on their own merit alone.
\end{Int}

\section*{Acknowledgments}

The author gratefully acknowledges funding by DFG~CRC~755, project~B8, DFG~CRC~803, project~Z2, and DFG~CRC~1456, Project~B2. I am very grateful to Stephan Huckemann for many helpful discussions and detailed comments to the manuscript, to Sarah K\"oster for permission to use the platelet data and to Yvo Pokern for helpful discussions.

\pagebreak
\appendix

\section{Additional Results for Asymptotics on Metric Spaces}\label{section-supp:aux-metric}

To understand the importance of the local Lipschitz condition \ref*{as:Lipschitz0}, it is necessary to introduce some empirical process theory. The most important notion for the following is the notion of the bracketing entropy of a function class.

\begin{Def} $ $\label{def:bracket-number}
  For a class of functions $\mathcal{F}$ from $Q$ to $\mathbb{R}$, a norm $\| \, \|$ on $\mathcal{F}$ and $\eps > 0$, consider two functions $f_1, f_2 \in \mathcal{F}$ such that $\| f_1 - f_2 \| < \eps$ and call
  \begin{align*}
    [f_1,f_2](\eps, \mathcal{F}, \| \, \|) := \left\{ f \in \mathcal{F} \, : \, \forall_0 q \in Q \quad f_1(q) \le f(q) \le f_2(q) \right\}
  \end{align*}
  an $\eps$ bracket. Then the \emph{bracketing number} $N_{[\, ]}(\eps, \mathcal{F}, \| \, \|)$ is defined as the minimal number of $\eps$ brackets needed such that their union is $\mathcal{F}$. If $\mathcal{F}$ has an envelope function $F$, the \emph{bracketing entropy} is defined as
  \begin{align*}
    J_{[\, ]} (\eta, \mathcal{F}, \| \, \|) := \int_0^\eta \limits \sqrt{1 + \log N_{[\, ]}(\eps\| F \|, \mathcal{F}, \| \, \|) } \, d\eps
  \end{align*}
\end{Def}

The bracketing entropy for a function class is difficult to determine. However, there is a similar concept for general metric spaces, namely the covering entropy, which is usually much simpler to determine.

\begin{Def} $ $\label{def:cover-number}
  The \emph{covering entropy} is defined as
  \begin{align*}
    J (\eta, \mathcal{P}, d) := \int_0^\eta \limits \sqrt{1 + \log N(\eps, \mathcal{P}, d) } \, d\eps
  \end{align*}
\end{Def}

The local Lipschitz condition \ref*{as:Lipschitz0} connects the bracketing number of a function class indexed by a totally bounded metric space to the covering number of this metric space. Consider the function classes $\mathcal{F}^i := \Big\{ \rho^i(p, \cdot) - \rho^i(\mu^i, \cdot), p \in \mathcal{B}_\delta(\mu^i) \Big\}$, then the connection is as follows.

\begin{Theorem}[\cite{vdVW96}, Theorem 2.7.11] $ $\label{theo:numbers}\\
  Under Assumption \ref*{as:Lipschitz0}, we have for any norm $\|\, \|$ on $\mathcal{F}^i$
  \begin{align*}
    N_{[\, ]}(2\eps\| \dot{\rho}^i \|, \mathcal{F}^i, \|\, \|) \le N(\eps, \mathcal{B}_\delta(\mu^i), d_i) \, .
  \end{align*}
\end{Theorem}

For the next theorem we introduce the empirical process and two relevant norms.
\begin{Def}
  Let $Q$ be the data space and $P = \mathbb{P} \circ X^{-1}$ the distribution measure in $Q$ corresponding to the random variable $X$. 
  \begin{itemize}
    \item[a)] For an i.i.d. sample $\{X_1, \dots, X_n\}$ from $P$ let $\mathbb{P}_n := \frac{1}{n} \sum_{j=1}^n \limits \delta_{X_j}$ be the \emph{empirical measure}.
    \item[b)] For any measurable function $f : Q \to \mathbb{R}$ let $\mathbb{G}_n (f) := \sqrt{n} \big( \mathbb{P}_n f - Pf \big)$ be the \emph{empirical process}.
    \item[c)] For any measurable function $f : Q \to \mathbb{R}$ let $\|f\|_{P,r} := \left( \int |f|^r dP \right)^\frac{1}{r}$ be the $L_r(P)$ norm.\\
    For any class $\mathcal{F}$ of measurable functions let $\| \mathbb{G}_n \|_\mathcal{F} := \sup_{f \in \mathcal{F}} \limits |\mathbb{G}_n(f)|$.
  \end{itemize}
\end{Def}

\begin{Theorem}[\cite{vdVW96}, Theorem 2.14.2] $ $\label{theo:emp-proc-bound-raw}\\
  For any class $\mathcal{F}$ of measurable functions with measurable envelope function $F$,
  \begin{align*}
    \mathbb{E} \left[ \| \mathbb{G}_n \|_\mathcal{F} \right] \lesssim J_{[\, ]} (1, \mathcal{F}, L_2(P)) \|F\|_{P,2}
  \end{align*}
\end{Theorem}

Now, we have all the tools to show an auxiliary proposition for Theorem \ref*{theo:CLT-m-variance-metric-spaces}

\begin{Prop}[Preparation for the CLT] \label{prop:CLT-metric}
  Under Assumptions \ref*{as:BPSC}, \ref*{as:finite}, \ref*{as:Lipschitz0} and \ref*{as:entropy} we have, for a measurable selection $\mun^i \in \left\{ x \in \mathcal{B}_\delta(\mu^i) : F_n(x) = \inf_{y \in \mathcal{B}_\delta(\mu^i)} \limits F_n(y) \right\}$
  \begin{align*}
    \sqrt{n} \big(F_n(\mun^i) - F(\mu^i)\big) &= \sqrt{n} \big(F_n(\mu^i) - F(\mu^i)\big) + o_P(1) \, .
  \end{align*}
\end{Prop}

\begin{proof}
  We follow the strategy laid out in \cite{DM18}~Proposition~3.
  \begin{align*}
    \MoveEqLeft[4] \mathbb{P} \left[ \sqrt{n} |F_n(\mun^i) - F_n(\mu^i)| > \eps \right] = \mathbb{P} \left[ \sqrt{n} \left( F_n(\mu^i) - F_n(\mun^i) \right) > \eps \right]\\
    \le& \mathbb{P} \left[ \sqrt{n} \left( F_n(\mu^i) - F_n(\mun^i) - \left( F(\mu^i) - F(\mun^i) \right) \right) > \eps \right]\\
    =& \mathbb{P} \left[ \sqrt{n} \left| F_n(\mun^i) - F(\mun^i) - \left( F_n(\mu^i) - F(\mu^i) \right) \right| > \eps \right]\\
    \le& \mathbb{P} \left[ \sup_{x^i \in \mathcal{B}_\delta(\mu^i)} \limits \sqrt{n} \left| F_n(x^i) - F(x^i) - \left( F_n(\mu^i) - F(\mu^i) \right) \right| > \eps \right] + \mathbb{P} [d(\mun^i, \mu^i) > \delta]
  \end{align*}
  Now, we can apply the Markov inequality, and Assumption \ref*{as:BPSC} to get
  \begin{align*}
    \mathbb{P} \left[ \sqrt{n} |F_n(\mun^i) - F_n(\mu^i)| > \eps \right] \le& \frac{1}{ \eps} \mathbb{E} \left[ \sup_{x^i \in \mathcal{B}_\delta(\mu^i)} \limits \left| \mathbb{G}_n \left( \rho^i(x^i, \cdot) - \rho^i(\mu^i, \cdot) \right) \right| \right] + o_P(1)\\
    \le& \frac{1}{ \eps} \mathbb{E} \left[ \| \mathbb{G}_n \|_{\mathcal{F}^i} \right] + o_P(1) \, .
  \end{align*}
  And finally, using Theorem \ref{theo:emp-proc-bound-raw} and Assumption \ref*{as:entropy}, we get
  \begin{align*}
    \mathbb{P} \left[ \sqrt{n} |F_n(\mun^i) - F_n(\mu^i)| > \eps \right] \le& \frac{1}{ \eps} \mathbb{E} \left[ \| \mathbb{G}_n \|_{\mathcal{F}^i} \right] + o_P(1)\\ \lesssim& \frac{2 \delta}{\eps} \|\dot{\rho}^i\|_{P,2} J_{[\, ]} (1, \mathcal{F}^i, L_2(P)) + o_P(1)\\
    =& \frac{2 \delta}{\eps} \|\dot{\rho}^i\|_{P,2} \int_0^1 \limits \sqrt{1 + \log N_{[\, ]}(2 \eps \delta \| \dot{\rho}^i \|, \mathcal{F}^i, L_2(P)) } \, d\eps  + o_P(1)\\
    \le& \frac{2 \delta}{\eps} \|\dot{\rho}^i\|_{P,2} \int_0^1 \limits \sqrt{1 + \log N(\eps \delta, \mathcal{B}_\delta(\mu^i), d_i ) } \, d\eps + o_P(1) = o_P(1) \, .
  \end{align*}
  This proves the result taking the limit $\delta \to 0$.
\end{proof}

\section{Auxiliary Results for Manifolds}\label{section-supp:aux-mf}

\subsection{An Entropy Result} \label{subsec-supp:entropy}

As a first step in the auxiliary results for the CLT we show a bound for the empirical process. In Lemma \ref{lem:entropy} we derive an entropy bound which is only coarsely sketched in \cite{vdV00} for the special case $\beta = 1$. We set out to make the argument more explicit and generalize it to general $\beta$. With this auxiliary result we can prove a bound on the empirical process in Theorem \ref{theo:emp-proc-bound}.

\begin{Lem} \label{lem:entropy}
  Under Assumptions \ref*{as:local-manifold} and \ref*{as:Lipschitz} we have, with a dimension dependent constant $K_p$, the following entropy condition
  \begin{align*}
    J_{[\, ]} (\eta, \mathcal{F}^i, \| \, \|) \le \int_0^\eta \limits \sqrt{1 + \log K_p + \frac{p}{\beta} \log \left(\frac{2^{1-\beta}}{\eps}\right) } \, d\eps
  \end{align*}
  for any $0 < \eta < \infty$ and any norm $\|\, \|$ on $\mathcal{F}^i := \Big\{ \tau^i(y, \cdot) - \tau^i(0, \cdot), y \in B_\delta(0) \Big\}$
\end{Lem}

\begin{proof}
  Note that $(2\delta)^\beta \| \dot{\tau}^i \|$ is an envelope function of $\mathcal{F}^i$, $\| \, \|^\beta$ is a norm for $0 < \beta \le 1$ and $\dot{\tau}^i$ takes the role of the Lipschitz ``constant'' in Theorem \ref{theo:numbers}. This yields
  \begin{align*}
    J_{[\, ]} (\eta, \mathcal{F}^i, \| \, \|) = \int_0^\eta \limits \sqrt{1 + \log N_{[\, ]}(\eps (2\delta)^\beta \| \dot{\tau}^i \|, \mathcal{F}^i, \| \, \|) } \, d\eps \le \int_0^\eta \limits \sqrt{1 + \log N(\eps (2\delta)^\beta/2, B_\delta(0), \| \, \|^\beta ) } \, d\eps \, .
  \end{align*}
  Next, we see that
  \begin{align*}
    N(\kappa, B_\delta(0), \| \, \|^\beta ) = K_p \left(\frac{\delta}{\kappa^{1/\beta}}\right)^p = K_p \left(\frac{\delta^\beta}{\kappa}\right)^{p/\beta}
  \end{align*}
  with $K_p$ being a dimension dependent constant and thus
  \begin{align*}
    J_{[\, ]} (\eta, \mathcal{F}^i, \| \, \|) \le \int_0^\eta \limits \sqrt{1 + \log K_p + \frac{p}{\beta} \log \left(\frac{2 \delta^\beta}{\eps (2\delta)^\beta}\right) } \, d\eps = \int_0^\eta \limits \sqrt{1 + \log K_p + \frac{p}{\beta} \log \left(\frac{2^{1-\beta}}{\eps}\right) } \, d\eps \, .
  \end{align*}
\end{proof}

Now, we can prove a bound on the empirical process.

\begin{Theorem} \label{theo:emp-proc-bound}
  Under Assumptions \ref*{as:BPSC}, \ref*{as:local-manifold} and \ref*{as:Lipschitz} there is a constant $C$ such that
  \begin{align*}
    \mathbb{E} \left[ \sup_{\|y\| < \delta} \limits | \mathbb{G}_n (\tau^i(y, X) - \tau^i(0, X)) | \right] \le C \delta^\beta
  \end{align*}
\end{Theorem}

\begin{proof}
  The claim follows by combining Theorem \ref{theo:emp-proc-bound-raw} and Lemma \ref*{lem:entropy} since $(2\delta)^\beta \| \dot{\tau}^i \|$ is an envelope function of $\mathcal{F}^i$, thus giving the order $\delta^\beta$ while
  \begin{align*}
    J_{[\, ]} (1, \mathcal{F}^i, L_2(P)) \le \int_0^1 \limits \sqrt{1 + \log K_p + \frac{p}{\beta} \log \left(\frac{2^{1-\beta}}{\eps}\right) } \, d\eps
  \end{align*}
  is independent of $\delta$ and therefore only contributes a constant.
\end{proof}

This result is used extensively in the proofs of auxiliary results to the CLT below.

\pagebreak[0]
\subsection{Auxiliary Results Used in the Proof of Theorem \ref*{theo:CLT-m-variance}}\label{subsec-supp:aux-clt}

We prove two incarnations of a preliminary proposition to Theorem \ref*{theo:CLT-m-variance}. Under the additional Assumption \ref{as:Taylor} on the Fr\'echet function Proposition \ref{prop:CLT-manifold1} holds, which highlights an explicit order of convergence. Proposition \ref{prop:CLT-manifold2} holds without Assumption \ref{as:Taylor} but does not give a specific order of convergence.

\begin{As}[Smooth Fr\'echet Function]\label{as:Taylor}
  Under Assumption \ref*{as:finite}, assume for every $\mu^i \in E$ a rotation matrix $R_i \in SO(p)$ and $T_{i,1},\ldots,T_{i,p} \neq 0$. Furthermore, assume that the Fr\'echet function admits the power series expansion 
  \begin{align}
    G^i(x) &= G^i(0) + \sum_{j=1}^p \limits T_{i,j} |(R_i x)_{j}|^{r_i} + o(\|x\|^{r_i})\,.\label{eq:power_series_full}
  \end{align}
  where the same $r_i$ are assumed in Assumption \ref*{as:local-manifold}.
\end{As}
Note that for $r_i = 2$ Equation \eqref{eq:power_series_full} encompasses every possible covariance whereas for $r_i > 2$ this represents a restriction to easily tractable tensors.

\begin{Prop}[Preparation for CLT -- convergence rate] \label{prop:CLT-manifold1}
  Under Assumptions \ref*{as:BPSC}, \ref*{as:finite}, \ref*{as:local-manifold}, \ref*{as:Lipschitz} and \ref{as:Taylor} we have, for a measurable selection $\nun^i$ of $\delta$-local sample $i$-descriptors
  \begin{align*}
    \sqrt{n} (G^i_n(\nun^i) - G^i(0)) = \sqrt{n} (G^i_n(0) - G^i(0)) + \mathcal{O}_P\left(n^{-\beta/(2(r_i-\beta))}\right) \, .
  \end{align*}
\end{Prop}

\begin{proof}
  In \cite{EH19}, it was shown that the assumptions of the theorem imply an asymptotic rate $\nun = \mathcal{O}_P(n^{-1/(2(r_i-1))})$. Thus we see
  \begin{align*}
    \sqrt{n} (G^i_n(\nun) - G^i(0)) =& \sqrt{n} (G^i_n(0) - G^i(0)) + \sqrt{n} (G^i(\nun) - G^i(0))\\
    &+ \sqrt{n} (G^i_n(\nun) - G^i(\nun) - (G^i_n(0) - G^i(0))) \, .
  \end{align*}
  The second term yields, using Lemma 5.52 in \cite{vdV00} or correspondingly Lemma 2.9 in \cite{EH19}
  \begin{align*}
    \sqrt{n} (G^i(\nun) - G^i(0)) = \sqrt{n} O(\|\nun\|^{r}) = \mathcal{O}_P \left(n^{1/2-r_i/(2(r_i-\beta))}\right) = \mathcal{O}_P\left(n^{-\beta/(2(r_i-\beta))}\right) \, .
  \end{align*}
  For the third term we get
  \begin{align*}
    \MoveEqLeft[4]\mathbb{P} \left[ \left|\sqrt{n} (G^i_n(\nun) - G^i(\nun) - (G^i_n(0) - G^i(0)))\right| > \eps \right]\\
    \le& \mathbb{P} \left[ \sup_{|x| \le \delta} \left|\sqrt{n} (G^i_n(x) - G^i(x) - (G^i_n(0) - G^i(0))) \right| > \eps \right] + \mathbb{P}\left[|\nun|\ge \delta \right] \, .
  \end{align*}
  Now, we can apply the Markov inequality, Theorem \ref{theo:emp-proc-bound} from this appendix, and Lemma 2.9 of \cite{EH19} to get
  \begin{align*}
    \mathbb{P} \left[ \left|\sqrt{n} (G^i_n(\nun) - G^i(\nun) - (G^i_n(0) - G^i(0)))\right| > \eps \right] \le& C \delta^\beta + \mathbb{P}\left[|\nun|\ge \delta \right] = \mathcal{O}_P\left(n^{-\beta/(2(r_i-\beta))}\right) \, .
  \end{align*}
  The claim follows at once.
\end{proof}

\begin{Prop}[Preparation for CLT -- general] \label{prop:CLT-manifold2}
  Under Assumptions \ref*{as:BPSC}, \ref*{as:finite}, \ref*{as:local-manifold} and \ref*{as:Lipschitz} we have, for a measurable selection $\nun^i$ of $\delta$-local sample $i$-descriptors
  \begin{align*}
    \sqrt{n} (G^i_n(\nun^i) - G^i(0)) = \sqrt{n} (G^i_n(0) - G^i(0)) + o_P\left(1\right) \, .
  \end{align*}
\end{Prop}

\begin{proof}
  We use the analogous argument to the proof of Proposition \ref{prop:CLT-metric} to show, using the Markov inequality
  \begin{align*}
    \mathbb{P} \left[ \left|\sqrt{n} (G^i_n(\nun) - G^i_n(0))\right| > \eps \right] \le& \frac{1}{\epsilon} \mathbb{E} \left[ \sup_{\|y\| < \delta} \limits | \mathbb{G}_n (\tau^i(y, X) - \tau^i(0, X)) | \right] + \mathbb{P}\left[|\nun^i|\ge \delta \right] \, .
  \end{align*}
  At this point we exploit the additional structure present in the finite dimensional manifold setting by applying Theorem \ref{theo:emp-proc-bound} and Assumption \ref*{as:BPSC} and taking the limit $\delta \to 0$ to get
  \begin{align*}
    \mathbb{P} \left[ \left|\sqrt{n} (G^i_n(\nun) - G^i_n(0))\right| > \eps \right] \le& \frac{C}{\epsilon} \delta^\beta + \mathbb{P}\left[|\nun^i|\ge \delta \right] = o_P(1) \, .
  \end{align*}
  This proves the claim.
\end{proof}

\begin{Rm}
  From Proposition \ref{prop:CLT-manifold2} it follows that Assumption \ref{as:Taylor} is not necessary for the results we aim to prove here. However, Proposition \ref{prop:CLT-manifold1} highlights that the asymptotic rate of the m-variance is not affected by smeariness, discussed in \cite{EH19}. Instead, the rate remains $n^{-1/2}$, even if the order of the Fr\'echet function $r_i$ in Assumption \ref{as:Taylor} or the H\"older continuity order $\beta$ in Assumption \ref*{as:Lipschitz} depart from their standard values $r_i=2$ and $\beta=1$. This is an interesting side result in its own right. Furthermore, the rate shown in Proposition \ref{prop:CLT-manifold1} may be regarded as a step in the direction of a Berry-Esseen type result for the m-variance.
\end{Rm}

The proof of Theorem \ref*{theo:CLT-m-variance} is in the main text.

\subsection{Auxiliary Result Used in the Proof of Theorem \ref*{theo:varvarCLT}}

\begin{Prop} \label{prop:varvarSimp}
  Under Assumptions \ref*{as:BPSC}, \ref*{as:finite}, \ref*{as:local-manifold} and \ref*{as:Lipschitz2}, for a measurable selection of $\delta$-local sample $i$-descriptors $\nun^i$ and for all combinations of $i,j \in \{1,\dots, m\}$
  \begin{align*}
    \sqrt{n} \big(G^{ij}_n(\nun^i, \nun^j) - G^{ij}(0,0)\big) = \sqrt{n} \big(G^{ij}_n(0,0) - G^{ij}(0,0)\big) + o_P(1)
  \end{align*}
\end{Prop}

\begin{proof}
  We follow the strategy laid out in \cite{DM18}~Proposition~3.
  \begin{align*}
    \MoveEqLeft[4] \mathbb{P} \left[ \sqrt{n} |G^{ij}_n(\nun^i, \nun^j) - G^{ij}_n(0,0)| > \eps \right] = \mathbb{P} \left[ \sqrt{n} \left( G^{ij}_n(0,0) - G^{ij}_n(\nun^i, \nun^j) \right) > \eps \right]\\
    \le& \mathbb{P} \left[ \sqrt{n} \left( G^{ij}_n(0,0) - G^{ij}_n(\nun^i, \nun^j) - \left( G^{ij}(0,0) - G^{ij}(\nun^i, \nun^j) \right) \right) > \eps \right]\\
    =& \mathbb{P} \left[ \sqrt{n} \left| G^{ij}_n(\nun^i, \nun^j) - G^{ij}(\nun^i, \nun^j) - \left( G^{ij}_n(0,0) - G^{ij}(0,0) \right) \right| > \eps \right]\\
    \le& \mathbb{P} \left[ \sup_{|x^i| \le \delta^i, \, |x^j| \le \delta^j} \limits \sqrt{n} \left| G^{ij}_n(x^i, x^j) - G^{ij}(x^i, x^j) - \left( G^{ij}_n(0,0) - G^{ij}(0,0) \right) \right| > \eps \right]\\
    &+ \mathbb{P} [|\nun^i| \ge \delta_i] + \mathbb{P} [|\nun^j| \ge \delta_j]
  \end{align*}
  Now, we can apply the Markov inequality, Theorem \ref{theo:emp-proc-bound} from this appendix, Lemma 2.9 in \cite{EH19} and Assumption \ref*{as:BPSC} to get
  \begin{align*}
    \mathbb{P} \left[ \sqrt{n} |G^{ij}_n(\nun^i, \nun^j) - G^{ij}_n(0,0)| > \eps \right] \le& C (\delta_i + \delta_j)^\beta + \mathbb{P} [|\nun^i| \ge \delta_i] + \mathbb{P} [|\nun^j| \ge \delta_j] = o_P(1) \, .
  \end{align*}
  This proves the claim.
\end{proof}

The proof of Theorem \ref*{theo:varvarCLT} is in the main text.

\subsection{Proof of Theorem \ref*{theo:varvarCLT}}\label{subsec-supp:cov-clt}

We follow the strategy laid out in \cite{DM18}~Proposition~4. Using $q \in \mathbb{R}^m$ we denote
\begin{align*}
  D_n^{v}(q,X) :=& \sum_{i=1}^m \limits \sum_{j=1}^m \limits v_i v_j G_n^{ij}(q_i,q_j,X)\\
  D^{v}(q,X) :=& \sum_{i=1}^m \limits \sum_{j=1}^m \limits v_i v_j G^{ij}(q_i,q_j,X)
\end{align*}
Writing $\nun := (\nun^1 \dots \nun^m)^T$, note that
\begin{align*}
  v^T \textnormal{Cov}[\tau(0,X)] v &= \mathbb{E}\left[(v^T \tau(0, X))^2\right] - \mathbb{E}\left[v^T \tau(0, X)\right]^2\\
  &= D^{v}(0,X) - \left( \sum_{i=1}^m \limits v_i G^i(0,X) \right)^2\\
  v^T \textnormal{Cov}[\tau_n^*(0,X^*)] v &= \frac{1}{n} \sum_{j=1}^n \limits \left( v^T \tau(\nun, X_j) \right)^2 - \left( \frac{1}{n} \sum_{j=1}^n \limits \left( v^T \tau(\nun, X_j) \right) \right)^2\\
  &= D_n^{v}(\nun,X) - \left( \sum_{i=1}^m \limits v_i G_n^i(\nun^i,X) \right)^2 \, .
\end{align*}
Therefore, we consider a two-dimensional setting and we get from Proposition \ref{prop:varvarSimp} and Corollary \ref*{cor:CLT-m-variance2}
\begin{align*}
  \mathcal{W}:=&\sqrt{n}
  \begin{pmatrix}
    D_n^{v}(\nun,X) &-& D^{v}(0,X)\\
    \sum_{i=1}^m \limits v_i G_n^i(\nun^i,X) &-& \sum_{i=1}^m \limits v_i G^i(0,X)
  \end{pmatrix}\\
  =&
  \sqrt{n}
  \begin{pmatrix}
    D_n^{v}(0,X) &-& D^{v}(0,X)\\
    \sum_{i=1}^m \limits v_i G_n^i(0,X) &-& \sum_{i=1}^m \limits v_i G^i(0,X)
  \end{pmatrix}
  + o_P(1) \inD \mathcal{N}\left( 0 , \Sigma_{v,2} \right) \, ,
\end{align*}
with convergence by the standard CLT.

Here, writing $C^{v} := \textnormal{Cov}\left[\left(v^T \tau(0, X) \right)^2, v^T \tau(0, X)\right]$, the covariance is
\begin{align*}
  \Sigma_{v,2} = 
  \begin{pmatrix}
    \textnormal{Var}\left[\left(v^T \tau(0, X) \right)^2\right] & C^{v}\\
    C^{v} & \textnormal{Var}\left[v^T \tau(0, X)\right]
  \end{pmatrix} \, .
\end{align*}
With the function $f (x, y) := x - y^2$ we have
\begin{align*}
  v^T \textnormal{Cov}[\tau(0,X)] v = f\left( \mathbb{E}\left[(v^T \tau(0, X))^2\right], \mathbb{E}\left[v^T \tau(0, X)\right] \right)\, ,
\end{align*}
and thus asymptotic normality of $v^T \textnormal{Cov}[\tau_n^*(0,X^*)] v$ follows with the delta method, where the Variance is
\begin{align*}
  W_v &= \begin{pmatrix} 1, & -2 \mathbb{E}\left[v^T \tau(0, X)\right] \end{pmatrix} \Sigma_{v,2} \begin{pmatrix} 1 \\ -2 \mathbb{E}\left[v^T \tau(0, X)\right] \end{pmatrix}\\
  &= \textnormal{Var}\left[\left(v^T \tau(0, X) \right)^2\right] - 4 \|v\|_1 V C^v + 4 \|v\|_1^2 V^2 \textnormal{Var}\left[v^T \tau(0, X)\right] \, .
\end{align*}
where the last equality follows from the fact that $\mathbb{E}\left[v^T \tau(0, X)\right] = \|v\|_1 V$.

\subsection{Proof of Theorem \ref*{theo:quantileLLN}}\label{subsec-supp:quantiles}

Equation (\ref*{eq:quant1}) follows from Theorems \ref*{theo:CLT-m-variance} and \ref{theo:normal-quantiles} and Equation (\ref*{eq:quant1b}) follows from Theorems \ref*{theo:CLT-m-variance} and \ref{theo:normal-quantiles-m4}. From Corollary \ref*{cor:CLT-boot-m-variance} we have for any fixed sampling sequence $X_1 , \dots$ with $\nun^i$ and $\nun^j$ for $i,j \in \{1, \dots, m\}$ with $i \neq j$
\begin{align*}
  \lim_{n \to \infty} \limits 2 \cdot \mathbb{P} \left(V_{n,n}^{*,j} - V_{n,n}^{*,i} - (\Vn^j - \Vn^i) < \sqrt{\frac{\widehat{W}^{ji}_n}{n}} q\left(\frac{\alpha}{2} \right) \right) = \alpha \, .
\end{align*}
Note that in this equation and all the following, the probability is taken over all samples with the distribution measure of $X$ and the corresponding bootstrap samples with their discrete uniform bootstrap measure. In consequence, the probability is to be understood as a deterministic value. We can condition on any subset of sampling sequences without loss of generality to get
\begin{align*}
  \lim_{n \to \infty} \limits 2 \cdot \mathbb{P} \left(V_{n,n}^{*,j} - V_{n,n}^{*,i} - (\Vn^j - \Vn^i) < \sqrt{\frac{\widehat{W}^{ji}_n}{n}} q\left(\frac{\alpha}{2} \right) \, \middle| \, \Vn \in A^{\alpha,\eps}_{n,i} \right) = \alpha
\end{align*}
for any fixed $\eps > 0$. Now, for all $k \in \mathbb{N}$ define $n_k$ such that
\begin{align*}
  \forall n \ge n_k : \, \left| 2 \cdot \mathbb{P} \left(V_{n,n}^{*,j} - V_{n,n}^{*,i} - (\Vn^j - \Vn^i) < \sqrt{\frac{\widehat{W}^{ji}_n}{n}} q\left(\frac{\alpha}{2} \right) \, \middle| \, \Vn \in A^{\alpha,2^{-1-k}}_{n,i} \right) - \alpha \right| < 2^{-1-k}
\end{align*}
where the convergence for fixed $\eps$ guarantees the existence of these $n_k$. Now define
\begin{align*}
  \eps_n = \left\{ \begin{array}{ll} 1/2 & \text{for } n < n_1 \\ 2^{-1-k} &  \text{for } n_k \le n < n_{k+1} \end{array} \right.
\end{align*}
and note that
\begin{align*}
  \lim_{n \to \infty} \limits 2 \cdot \mathbb{P} \left(V_{n,n}^{*,j} - V_{n,n}^{*,i} - (\Vn^j - \Vn^i) < \sqrt{\frac{\widehat{W}^{ji}_n}{n}} q\left(\frac{\alpha}{2} \right) \, \middle| \, \Vn \in A^{\alpha,\eps_n}_{n,i} \right) = \alpha \, .
\end{align*}
and then deduce from Theorem \ref*{theo:varvarCLT}
\begin{align*}
  \alpha = & \lim_{n \to \infty} \limits 2 \cdot \mathbb{P} \left(V_{n,n}^{*,j} - V_{n,n}^{*,i} - (\Vn^j - \Vn^i) < \sqrt{\frac{\widehat{W}^{ji}_n}{n}} q\left(\frac{\alpha}{2} \right) \, \middle| \, \Vn \in A^{\alpha,\eps_n}_{n,i} \right)\\
  \le& \lim_{n \to \infty} \limits 2 \cdot \mathbb{P} \left(V_{n,n}^{*,j} - V_{n,n}^{*,i} < \frac{\sqrt{\widehat{W}^{ji}_n} - \sqrt{W^{ji}}}{\sqrt{n}} q\left(\frac{\alpha}{2} \right) \, \middle| \, \Vn \in A^{\alpha,\eps_n}_{n,i} \right)\\
  =& \lim_{n \to \infty} \limits 2 \cdot \mathbb{P} \left(V_{n,n}^{*,j} - V_{n,n}^{*,i} < \mathcal{O}_P(n^{-1}) \, \middle| \, \Vn \in A^{\alpha,\eps_n}_{n,i} \right)\\
  =& \lim_{n \to \infty} \limits 2 \cdot \mathbb{P} \left(V_{n,n}^{*} \notin A^{1}_{n,ij} \, \middle| \, \Vn \in A^{\alpha,\eps_n}_{n,i} \right)\\
  \le& \lim_{n \to \infty} \limits 2 \cdot \mathbb{P} \left(V_{n,n}^{*} \notin A^{1}_{n,i} \, \middle| \, \Vn \in A^{\alpha,\eps_n}_{n,i} \right)
\end{align*}
which yields equation (\ref*{eq:quant2}).

\subsection{Test Size for Fr\'echet Function with \texorpdfstring{$m$}{m} Minima}

Consider a random vector $X \sim N(0,\Sigma)$ and for any $i,j \in \{1, 2, \dots, m\}$ with $i \neq j$ let $e_{ij} := (e_i -e_j)$. Define the following sets
\begin{align*}
  A^\alpha_{ij} &= \left\{x \in \mathbb{R}^m : e_{ij}^T x \le \left( e_{ij}^T \Sigma e_{ij} \right)^{1/2} q\left(\frac{\alpha}{2} \right) \right\} & A^\alpha_i :=& \bigcap_{\substack{j= 1 \\ j \neq i}}^m \limits A^\alpha_{ij} \, .
\end{align*}

\begin{Theorem} \label{theo:normal-quantiles}
  Let $X \sim N(0,\Sigma)$ with $\Sigma$ being of full rank be a multivariate normal random vector in $\mathbb{R}^m$ for $m \in \{2,3\}$. Then for all $\alpha \in [0,1]$
  \begin{align*}
    \sum_{i=1}^m \limits \mathbb{P} \left( A^{\alpha}_i \right) \le \alpha
  \end{align*}
\end{Theorem}

\begin{proof}
  For $m=2$ we get $\sum_{i=1}^2 \limits \mathbb{P} \left( A^{\alpha}_i \right) = \mathbb{P} \left( A^{\alpha}_{12} \right) + \mathbb{P} \left( A^{\alpha}_{21} \right) = \alpha$.
  
  For $m=3$ we get
  \begin{align*}
    \sum_{i=1}^3 \limits \mathbb{P} \left( A^{\alpha}_i \right) =& \mathbb{P} \left( A^{\alpha}_{12} \cap A^{\alpha}_{13} \right) + \mathbb{P} \left( A^{\alpha}_{21} \cap A^{\alpha}_{23} \right) + \mathbb{P} \left( A^{\alpha}_{13} \cap A^{\alpha}_{23} \right)\\
    =& \mathbb{P} \left( A^{\alpha}_{12} \cap A^{\alpha}_{13} \right) + \mathbb{P} \left( A^{\alpha}_{21} \cap A^{\alpha}_{13} \cap A^{\alpha}_{23} \right) + \mathbb{P} \left( (A^{\alpha}_{21} \cup A^{\alpha}_{13} ) \cap A^{\alpha}_{23} \right)\\
    =& \mathbb{P} \left( (A^{\alpha}_{12} \cup (A^{\alpha}_{21} \cap A^{\alpha}_{23})) \cap A^{\alpha}_{13} \right) + \mathbb{P} \left( (A^{\alpha}_{21} \cup A^{\alpha}_{13} ) \cap A^{\alpha}_{23} \right)\\
    \le& \mathbb{P} \left( A^{\alpha}_{13} \right) + \mathbb{P} \left( A^{\alpha}_{23} \right) = \alpha \, .
  \end{align*}
\end{proof}

\begin{Theorem} \label{theo:normal-quantiles-m4}
  Let $X \sim N(0,\Sigma)$ with $\Sigma$ being of full rank be a multivariate normal random vector in $\mathbb{R}^m$ for $m \ge 4$. Then for sufficiently small $\alpha \ge 0$
  \begin{align*}
    \sum_{i=1}^m \limits \mathbb{P} \left( A^\alpha_i \right) \le \alpha \, ,
  \end{align*}
\end{Theorem}

\begin{proof}
  Define the functions
  \begin{align*}
    f_i : \alpha \mapsto \frac{1}{(2\pi)^{\frac{m}{2}} (\det\Sigma)^{1/2}}\int_{A^\alpha_{i}} \limits \exp \left( - \frac{1}{2} x^T \Sigma^{-1} x \right) \, dx \, .
  \end{align*}
  The claim follows, if $f_i'(0) = 0$ for all $i$, therefore we will show this. In order to simplify calculations, we make a basis transform to ``homoscedastic coordinates''. Let $\widetilde{v}_{ij} := \Sigma^{1/2} e_{ij}$, $v_{ij} := \frac{\widetilde{v}_{ij}}{|\widetilde{v}_{ij}|}$ and $\widetilde{y} := \Sigma^{-1/2} x$. Let $e_0 := \frac{1}{\sqrt{m}} \sum_{i=1}^m \limits e_i$ and define
  \begin{align*}
    H_i :=& \begin{pmatrix} \Sigma^{1/2} e_0 & v_{i1} & \dots & v_{i(i-1)} & v_{i(i+1)} & \dots & v_{im} \end{pmatrix}^T & w_i :=& H_i^{-1} \begin{pmatrix} 0 & 1 & \dots & 1 \end{pmatrix}^T \, .
  \end{align*}
  This allows us to define $y := \widetilde{y} - w_i q\left(\frac{\alpha}{2} \right)$. Then, the following sets are equivalent to the $A^\alpha_{ij}$ defined above
  \begin{align*}
    B^\alpha_{ij} :=& \left\{\widetilde{y} \in \mathbb{R}^m : v_{ij}^T \widetilde{y} \le q\left(\frac{\alpha}{2} \right) \right\} & B^\alpha_i :=& \bigcap_{\substack{j= 1 \\ j \neq i}}^m \limits B^\alpha_{ij} \, .
  \end{align*}
  We note that $\widetilde{y} \in B^\alpha_{i} \, \Leftrightarrow \, y \in B^1_{i}$. We can now introduce shorthand notation and express the above defined functions $f_i$ as
  \begin{align*}
    \exp_i(y,\alpha) :=& \exp \left( - \frac{1}{2} \left|y + w_i q\left(\frac{\alpha}{2} \right) \right|^2 \right) & f_i(\alpha) =& \frac{1}{(2\pi)^{\frac{m}{2}}}\int_{B^1_{i}} \limits \exp_i(y,\alpha) \, dy \,.
  \end{align*}
  In the following we suppress the arguments of $q$ and $\exp_i$ to calculate the derivative with respect to $\alpha$
  \begin{align*}
    f_i'(\alpha) &= - \frac{q'}{2(2\pi)^{\frac{m}{2}}} \int_{B^1_{i}} \limits w_i^T \left(y + w_i q \right) \exp_i \, dz
  \end{align*}
  To simplify this expression we define the boundary of $B^1_i$ as
  \begin{align*}
    \partial B^1_{ij} :=& \big\{y \in B^1_{i} : v_{ij}^T y = 0 \big\} \, , & \partial B^1_{i} :=& \bigcup_{\substack{j= 1 \\ j \neq i}}^m \partial B^1_{ij}  \, ,
  \end{align*}
  and we use the fact that for every $j \neq i$ we have $w_i^T v_{ij} = 1$. Then, we get by applying Gauss' integral theorem, substituting $y = -q z$ (note that $q < 0$ for $\alpha < 1$) and using $q' = (2\pi)^{1/2} \exp \left( \frac{1}{2} q^2 \right)$ we get
  \begin{align*}
    f_i'(\alpha) &= \frac{q'}{2(2\pi)^{\frac{m}{2}}} \int_{B^1_{i}} \limits \textnormal{div}_y \big( w_i \exp_i \big) \, dy = \frac{q'}{2(2\pi)^{\frac{m}{2}}} \int_{\partial B^1_{i}} \limits \exp_i \, dy\\
    &= \frac{(-q)^{m-1}}{2(2\pi)^{\frac{m-1}{2}}} \int_{\partial B^1_{i}} \limits \exp \left( - \frac{1}{2} \left( (z -w_i)^2 - 1 \right) q^2 \right) dz \, .
  \end{align*}
  Let $p_i := \argmin_{p \in B_i^1} \limits |p-w_i|$ and note that we have $|p_i-w_i|^2 =: 1 + \delta_i > 1$. Furthermore, let
  \begin{align*}
    C_i :=& \left\{ z \in \mathbb{R}^m : (w_i - p_i)^T z \le 0 \right\} & \partial C_i :=& \left\{ z \in \mathbb{R}^m : (w_i - p_i)^T z = 0 \right\}
  \end{align*}
  where, according to Lemma \ref{lem:quantile-aux}, $B_i \subset C_i$ and $p_i \in \partial C_i$ and for $z \in \partial C_i$, one has $(z-w_i)^2 = (z-p_i)^2 + (p_i-w_i)^2$. Then we calculate,
  \begin{align*}
    f_i'(\alpha) \le& \frac{(-q)^{m-1}}{2(2\pi)^{\frac{m-1}{2}}} \int_{\partial C^1_{i}} \limits \exp \left( - \frac{1}{2} \left( (z -w_i)^2 - 1 \right) q^2 \right) dz\\
    =& \frac{(-q)^{m-1}}{2(2\pi)^{\frac{m-1}{2}}} \int_{\partial C^1_{i}} \limits \exp \left( - \frac{1}{2} \left( ((z-p_i)^2 + (p_i-w_i)^2 - 1 \right) q^2 \right) dz\\
    =& \frac{1}{2} \exp \left( - \frac{1}{2} \left( (p_i-w_i)^2 - 1 \right) q^2 \right) = \frac{1}{2} \exp \left( - \frac{1}{2} \delta_i q^2 \right) \, .
  \end{align*}
  Now, one can immediately read off $\lim_{\alpha \to 0} \limits f_i'(\alpha) = \lim_{q \to - \infty} \limits f_i'(\alpha) = 0$, which proves the claim.
\end{proof}

\begin{Lem} \label{lem:quantile-aux}
  Consider the following sets defined in Theorem \ref{theo:normal-quantiles-m4}
  \begin{align*}
    B^1_{ij} :=& \left\{z \in \mathbb{R}^m : v_{ij}^T z \le 0 \right\} & B^1_i :=& \bigcap_{\substack{j= 1 \\ j \neq i}}^m \limits B^1_{ij}
  \end{align*}
  and note that for all $i \in \{1, \dots, m \}$ the set $B_i^1$ is convex and if $v \in B_i^1$ then $\lambda v \in B_i^1$ for every $\lambda \ge 0$. Now, for $w_i \notin B_i^1$ let $p_i := \argmin_{p \in B_i^1} \limits |p-w_i|$ and define  
  \begin{align*}
    C_i :=& \left\{ z \in \mathbb{R}^m : (w_i - p_i)^T z \le 0 \right\} & \partial C_i :=& \left\{ z \in \mathbb{R}^m : (w_i - p_i)^T z = 0 \right\} \, .
  \end{align*}
  Then we have
  \begin{enumerate}[(i)]
    \item $(w_i - p_i)^T p_i = 0$ and thus $p_i \in \partial C_i$
    \item $B_i \subset C_i$
    \item For $z \in \partial C_i$ one has $(z-w_i)^2 = (z-p_i)^2 + (p_i-w_i)^2$.
  \end{enumerate}
\end{Lem}

\begin{proof} $ $
  \begin{enumerate}[(i)]
    \item The proof is done by contradiction. As first case, assume $(w_i - p_i)^T p_i > 0$. Then, for $\lambda > 1$, let
    \begin{align*}
      |w_i - \lambda p_i|^2 - |w_i - p_i|^2 = (\lambda^2 - 1) p_i^T p_i -2 (\lambda - 1) w_i^T p_i = (\lambda - 1) \left( (\lambda - 1) p_i^T p_i - 2 (w_i - p_i)^T p_i \right)
    \end{align*}
    Therefore,
    \begin{align*}
      |w_i - \lambda p_i|^2 < |w_i - p_i|^2 \quad \Leftrightarrow \quad 1 < \lambda < 1 + 2 \frac{(w_i - p_i)^T p_i}{p_i^T p_i} \, ,
    \end{align*}
    which can be satisfied. Therefore, $p_i \neq \argmin_{p \in B_i^1} \limits |p-w_i|$ and we have a contradiction.
    
    As second case, let $(w_i - p_i)^T p_i < 0$. Then, for $\lambda < 1$ we get by the same computation
    \begin{align*}
      |w_i - \lambda p_i|^2 < |w_i - p_i|^2 \quad \Leftrightarrow \quad 1 + 2 \frac{(w_i - p_i)^T p_i}{p_i^T p_i} < \lambda < 1 \, ,
    \end{align*}
    which again can be satisfied, such that we get a contradiction.
    \item The proof is done by contradiction. Assume $z \in B_i^1$ such that $z \notin C_i$. Because $B_i^1$ is convex, we know that for every $\lambda \in [0,1]$ we have $p_\lambda := (1- \lambda) p_i + \lambda z \in B_i^1$. Now, note that
    \begin{align*}
      |w_i - p_\lambda|^2 - |w_i - p_i|^2 =& |w_i - p_i + \lambda (p_i -z)|^2 - |w_i - p_i|^2 = 2 \lambda (w_i - p_i)^T (p_i - z) + \lambda^2 |p_i - z|^2\\
      =& \lambda^2 |p_i - z|^2 - 2 \lambda (w_i - p_i)^T z \, .
    \end{align*}
    In consequence,
    \begin{align*}
      |w_i - p_\lambda|^2 < |w_i - p_i|^2 \quad \Leftrightarrow \quad \lambda < 2 \frac{(w_i - p_i)^T z}{|p_i - z|^2} \, ,
    \end{align*}
    which can be satisfied. Thus, we have a contradiction.
    \item The proof is done by direct calculation.
    \begin{align*}
      |z-p_i|^2 + |p_i-w_i|^2 - |z-w_i|^2 =& |z|^2 - 2 p_i^T z + |p_i|^2 + |w_i|^2 - 2 w_i^T p_i + |p_i|^2 - |z|^2 + 2 w_i^T z - |w_i|^2\\
      =& 2 (w_i - p_i)^T z - 2 (w_i - p_i)^T p_i = 0
    \end{align*}
  \end{enumerate}
  \vspace*{-\baselineskip}
  \phantom{}
\end{proof}

\newpage
\section{Additional Background on Applications}\label{section-supp:application}

\subsection{Rome Wind Directions}\label{subsec-supp:wind}

Here, we show some more illustrations to give a more complete picture of the wind direction data set. This is still far from a full analysis of the data set, but it serves to illustrate the points made in the article.

\begin{figure}[h!]
  \centering
  \includegraphics[width=0.2\textwidth]{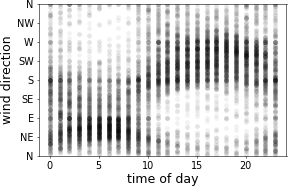}
  \includegraphics[width=0.2\textwidth]{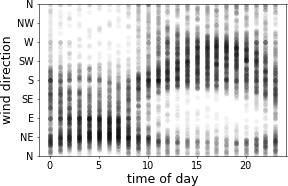}
  \includegraphics[width=0.2\textwidth]{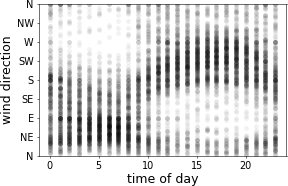}
  \includegraphics[width=0.2\textwidth]{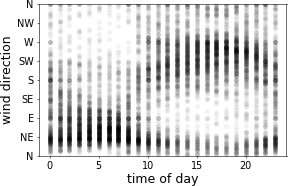}\\
  \includegraphics[width=0.2\textwidth]{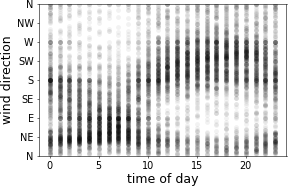}
  \includegraphics[width=0.2\textwidth]{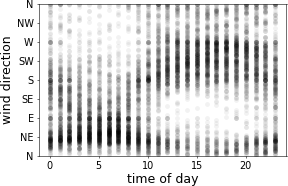}
  \includegraphics[width=0.2\textwidth]{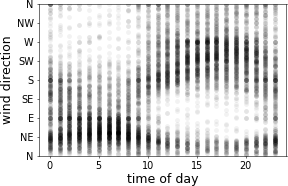}
  \includegraphics[width=0.2\textwidth]{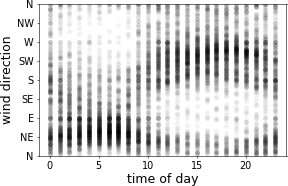}\\
  \includegraphics[width=0.2\textwidth]{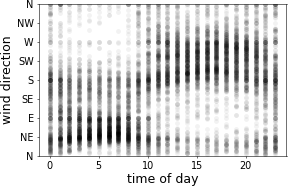}
  \includegraphics[width=0.2\textwidth]{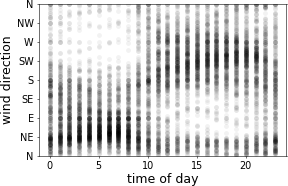}
  \includegraphics[width=0.2\textwidth]{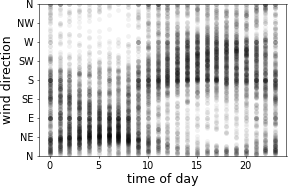}
  \includegraphics[width=0.2\textwidth]{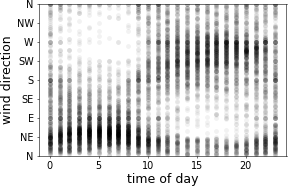}\\
  \includegraphics[width=0.2\textwidth]{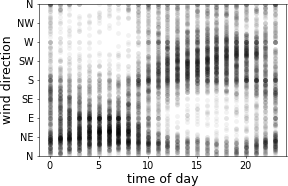}
  \includegraphics[width=0.2\textwidth]{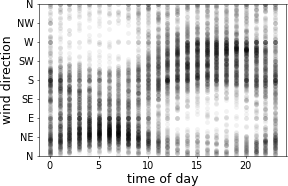}
  \includegraphics[width=0.2\textwidth]{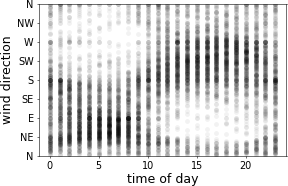}
  \includegraphics[width=0.2\textwidth]{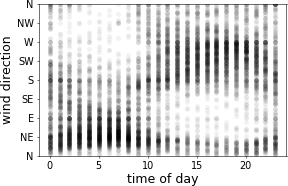}\\
  \includegraphics[width=0.2\textwidth]{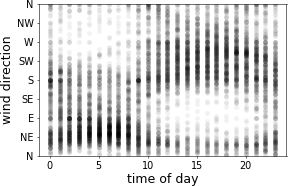}
  \includegraphics[width=0.2\textwidth]{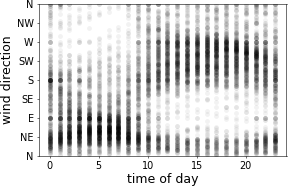}
  \includegraphics[width=0.2\textwidth]{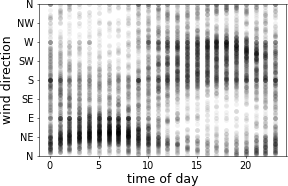}
  \includegraphics[width=0.2\textwidth]{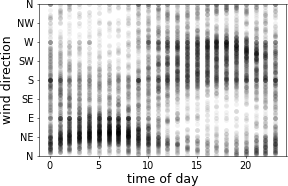}
  \caption{Daily wind pattern illustrated for all years from 2000 to 2019. The plot clearly demonstrates that during night an morning the predominant wind direction is from the northeast, whereas in the afternoon and evening, wind comes predominantly from the southwest. Overall, the southwestern wind covers a longer share of the day, which illustrates that the overall predominant wind direction is still southwest, as expected in a temperate zone. A slight seasonal effect is visible in the fact that for some days, wind in the afternoon still comes from the northeast.\label{fig:wind-daily}}
\end{figure}

\begin{figure}[h!]
  \centering
  \includegraphics[width=0.2\textwidth]{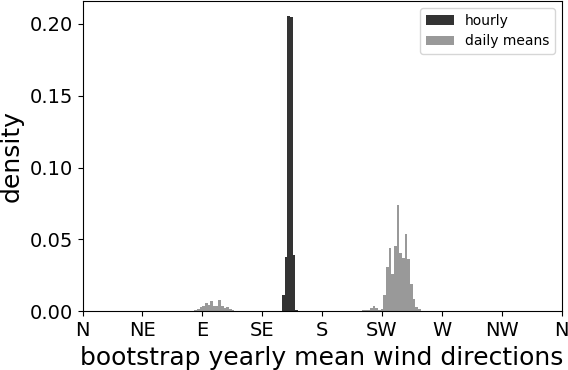}
  \includegraphics[width=0.2\textwidth]{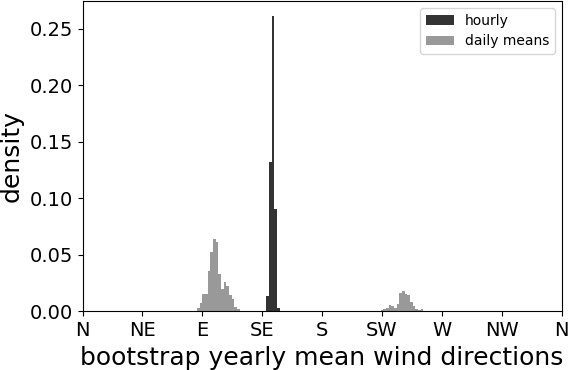}
  \includegraphics[width=0.2\textwidth]{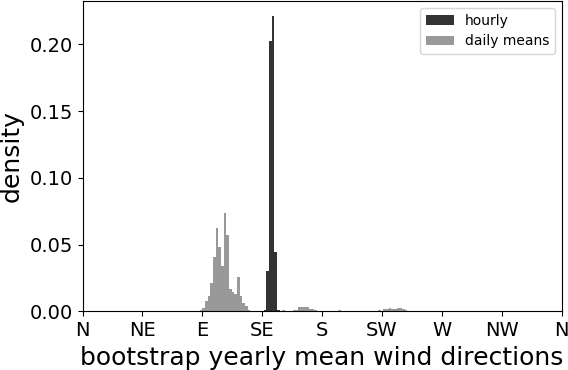}
  \includegraphics[width=0.2\textwidth]{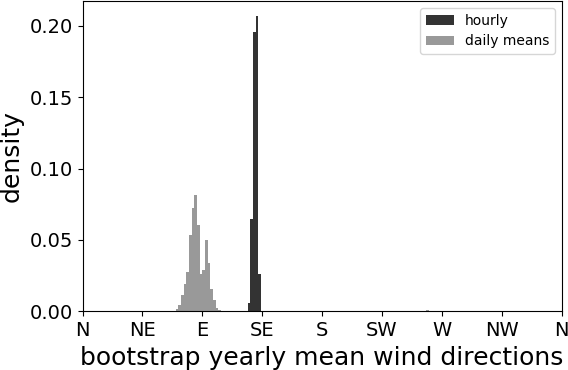}\\
  \includegraphics[width=0.2\textwidth]{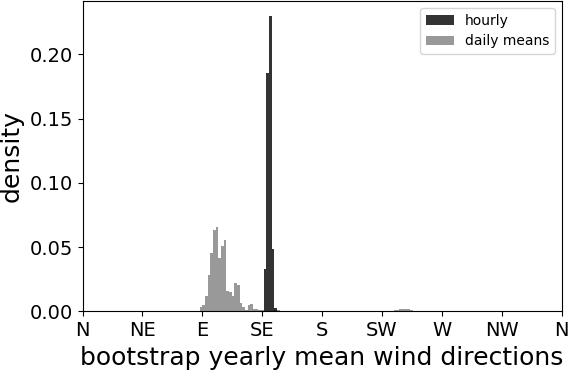}
  \includegraphics[width=0.2\textwidth]{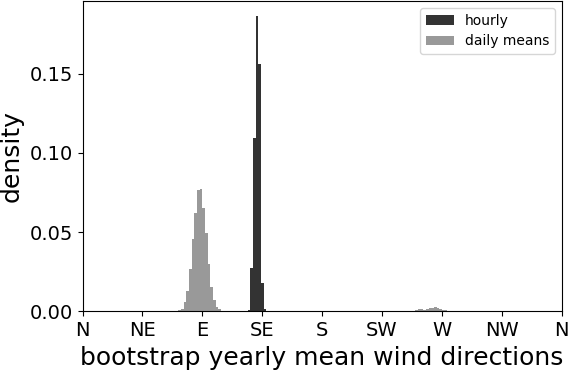}
  \includegraphics[width=0.2\textwidth]{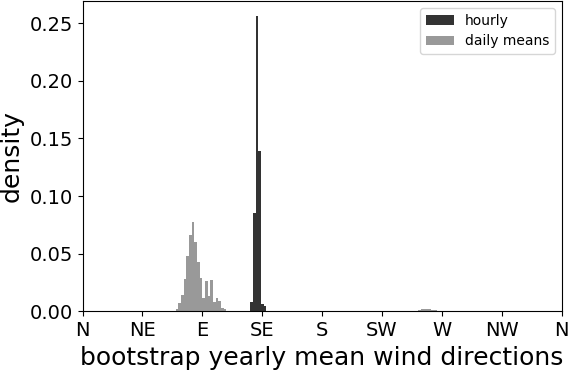}
  \includegraphics[width=0.2\textwidth]{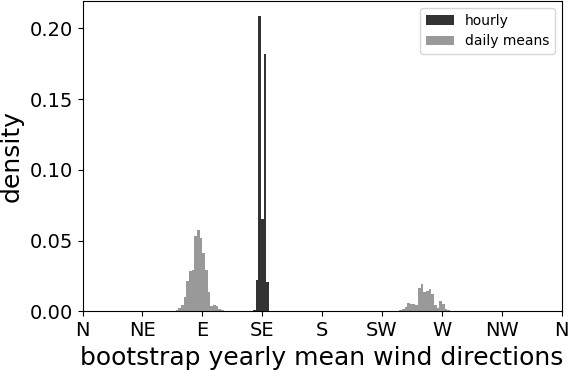}\\
  \includegraphics[width=0.2\textwidth]{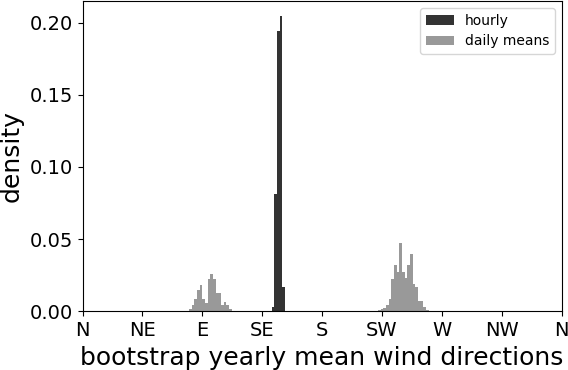}
  \includegraphics[width=0.2\textwidth]{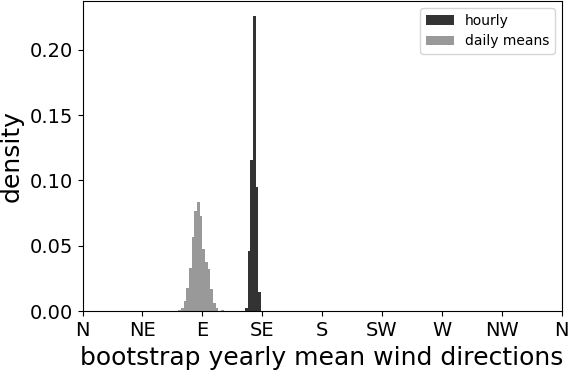}
  \includegraphics[width=0.2\textwidth]{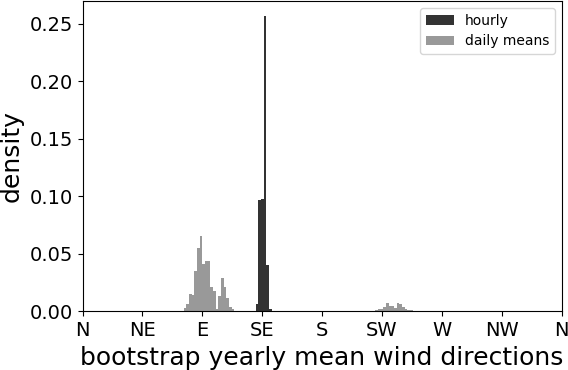}
  \includegraphics[width=0.2\textwidth]{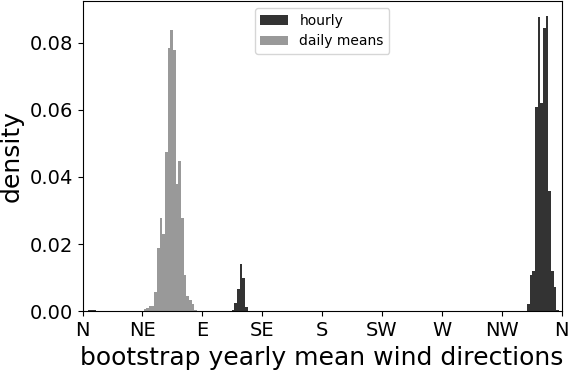}\\
  \includegraphics[width=0.2\textwidth]{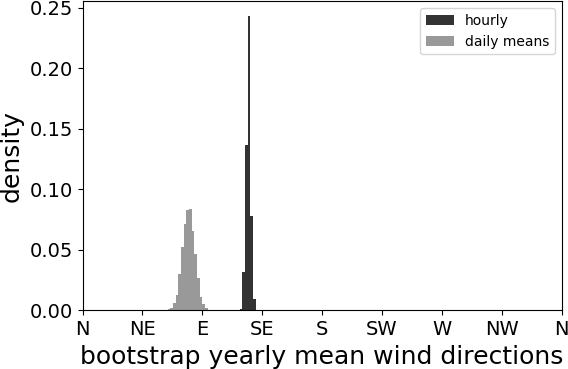}
  \includegraphics[width=0.2\textwidth]{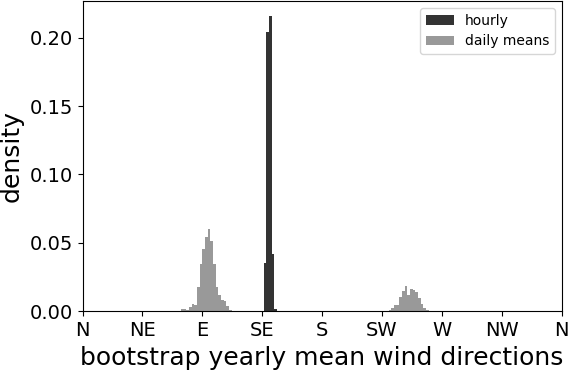}
  \includegraphics[width=0.2\textwidth]{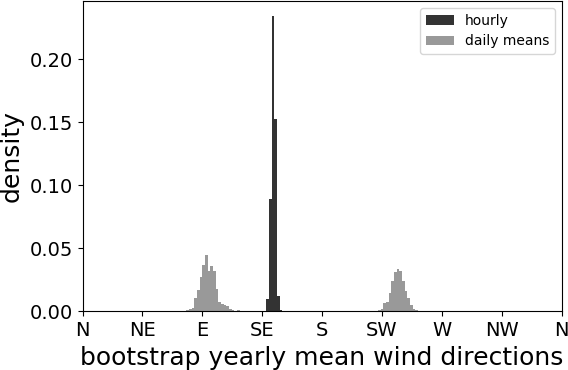}
  \includegraphics[width=0.2\textwidth]{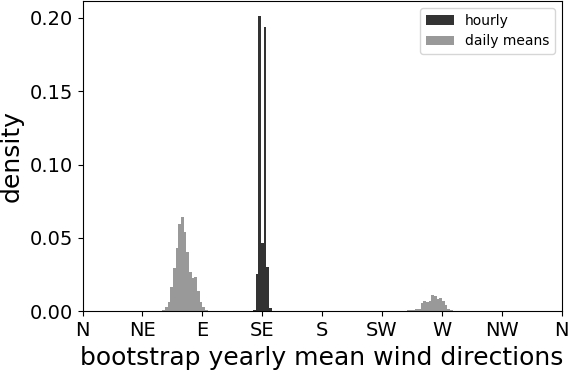}\\
  \includegraphics[width=0.2\textwidth]{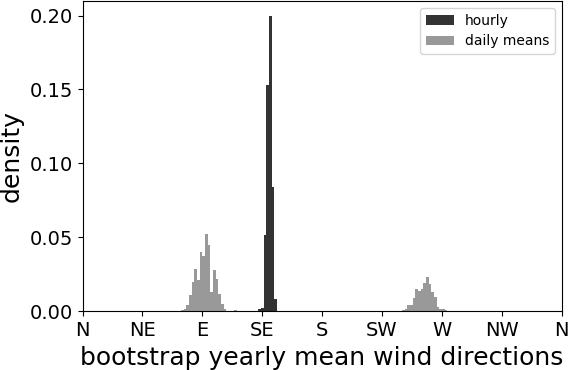}
  \includegraphics[width=0.2\textwidth]{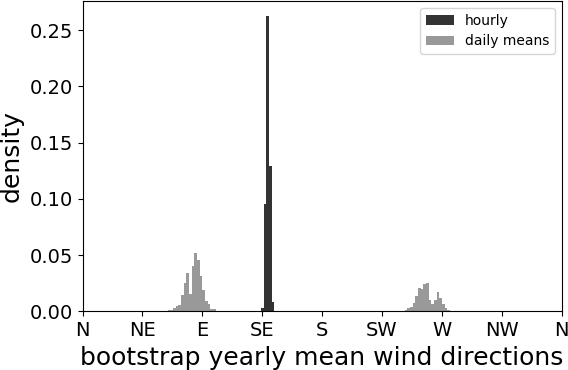}
  \includegraphics[width=0.2\textwidth]{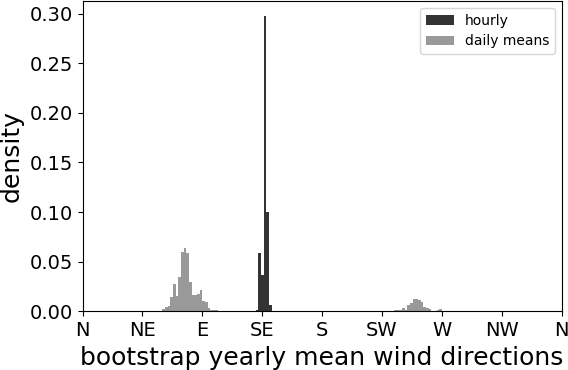}
  \includegraphics[width=0.2\textwidth]{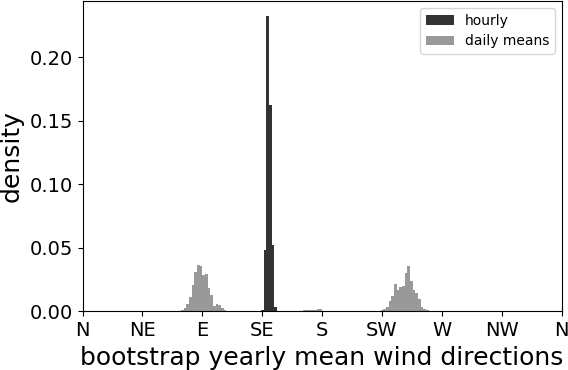}
  \caption{Bootstrap mean wind directions for all years from 2000 to 2019. The dark histograms indicate the means of hourly wind directions while the lighter histograms indicate means of the daily mean wind directions. One can clearly see that the distribution of means of the hourly wind directions is unimodal for all years except 2011. The means of daily means are much more prone to exhibiting a multimodal distribution, indicating non-uniqueness. \label{fig:wind-bootstrap}}
\end{figure}

\newpage
\subsection{Nesting Sea Turtles}\label{subsec-supp:turtles}

As another classic example of the mean on the circle, we consider the turtle data set presented by \cite{Stephens1969} and \cite[p. 9]{MJ00}, which gives the compass directions under which $n=76$ female turtles leave their nests after egg laying. This data set was discussed above as an example of finite sample smeariness, so we expect the test not to reject.

\begin{figure}[h!]
  \centering
  \subcaptionbox{histogram of turtle directions}[0.45\textwidth]{\includegraphics[width=0.4\textwidth]{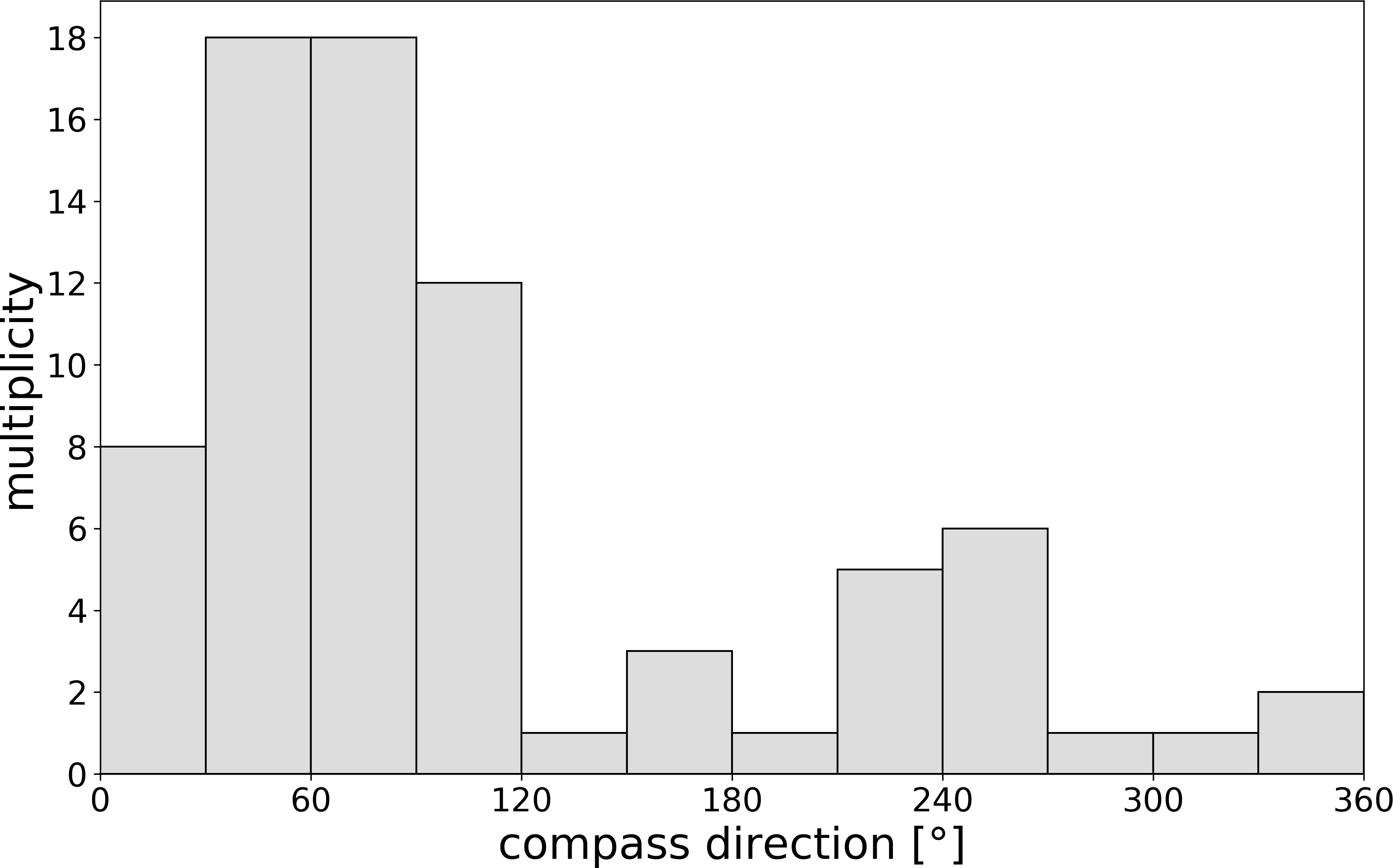}}
  \hspace*{0.02\textwidth}
  \subcaptionbox{histogram of bootstrap mean distances $d_j$}[0.45\textwidth]{\includegraphics[width=0.4\textwidth]{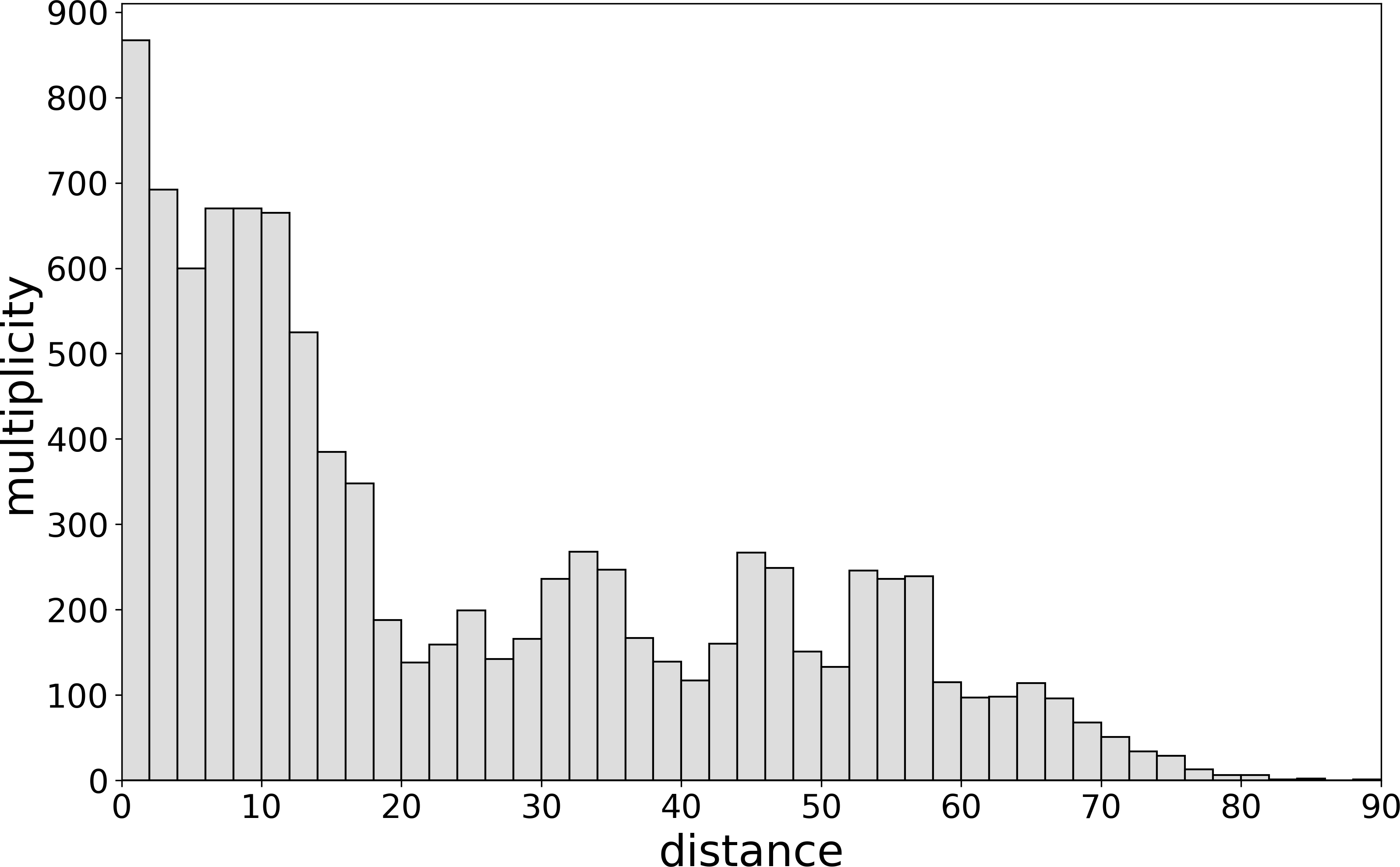}}
  \caption{Illustration of the turtle data and the results of the bootstrap with $B=10\,000$. Panel~(a) shows a pronounced bimodal structure of the data and panel~(b) shows that this leads to several secondary modes in the distribution of bootstrap means. This already indicates that the test will likely not reject the null hypothesis for this data set. Due to the small $n$, the number of minima $m$ of the underlying population cannot be easily determined in this case.\label{fig:turtles}}
\end{figure}

The data and the bootstrap histogram used for the hypothesis test are illustrated in Figure \ref{fig:turtles}. Indeed the hypothesis test with $B=10\,000$ yields a p-value of $0.7868$, which means that the hypothesis that the mean is not unique cannot be rejected.

\begin{Int}
  The data set was previously investigated for finite sample smeariness, as defined by \cite{HEH19}, in \cite{EH19} where it was shown to exhibit very pronounced finite sample smeariness. Indeed, due to the fact that the limiting distribution of the mean in case of smeariness is not unimodal, the test for non-uniqueness of the mean can be expected not to reject in case of finite sample smeariness. This is desirable, since smeariness occurs generically as a boundary case between distributions with a unique mean and distributions with non-unique mean. The data set is therefore compatible with a smeary mean of the population as well as with non-unique means.
\end{Int}

\subsection{Gaussian Mixture Clustering}\label{subsec-supp:gauss-mixture}

For another example of clustering, we consider centroid based clustering of data on $\mathbb{R}^2$ and $\mathbb{R}^4$. In the examples presented here, we use a Gaussian mixture model and apply an EM algorithm for the optimization. We use two classic data sets which are provided by the GNU R \texttt{datasets} package.

The \texttt{faithful} data set from \cite{AB1990} contains eruption duration and preceding inactivitydisplayed in $\mathbb{R}^2$ of $n=272$ eruptions of the Old Faithful geyser in Yellowstone national park. The data set is famously bimodal and the fit of a two-cluster mixture model displayed in Figure~\ref{fig:faithful} clearly rejects the hypothesis of non-uniqueness with a p-value of $p_d = 0$ with $B = 10\,000$. However, when fitting a three-cluster mixture model, which can be considered overfitting, the test no longer reject, with p-value $p_d = 0.6804$. This reinforces the notion that three clusters cannot be uniquely fit to the data in a meaningful way.

\begin{figure}[h!]
  \centering
  \includegraphics[width=0.5\textwidth]{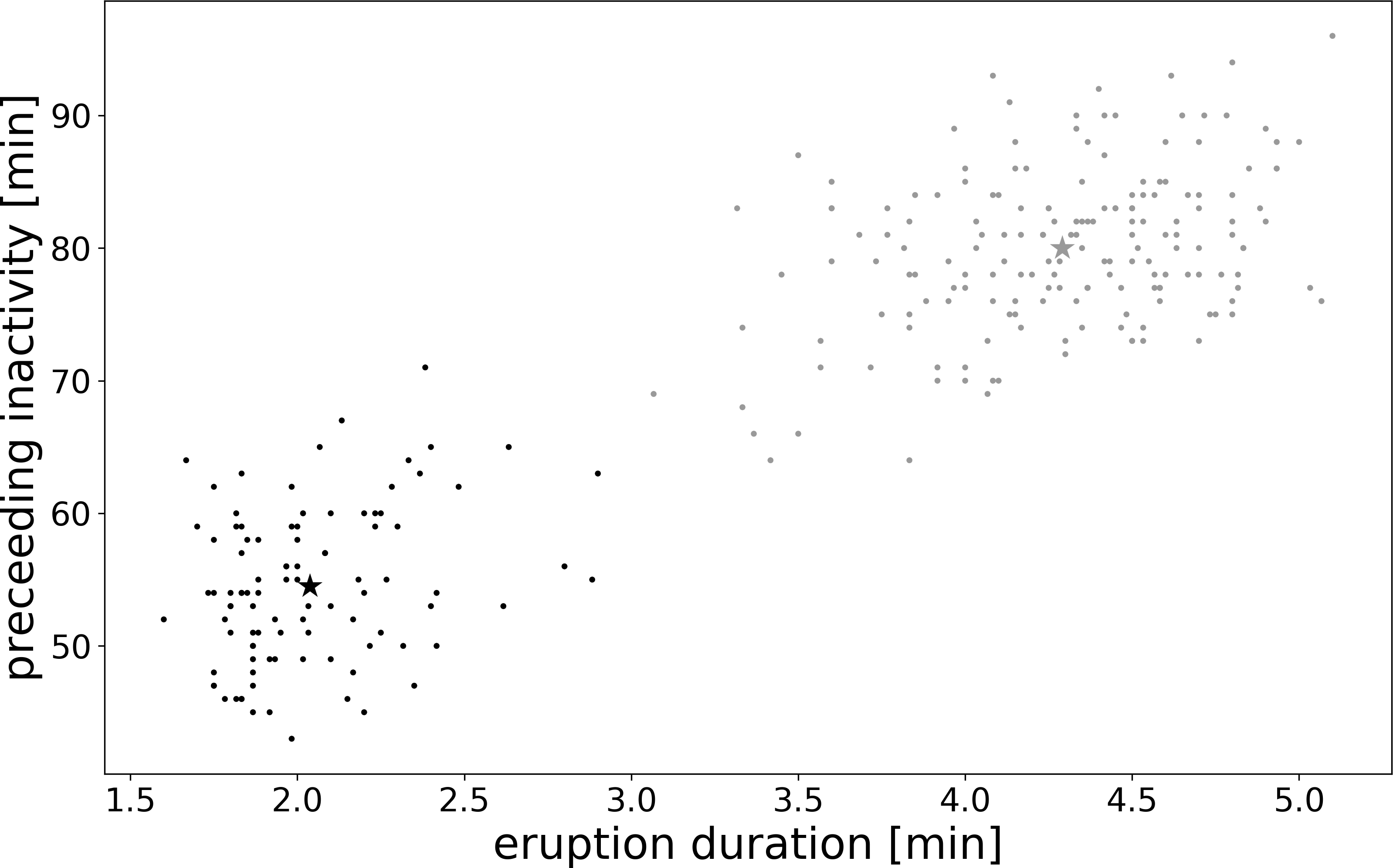}
  \caption{Scatter plot of old faithful eruption data with colors indicating the two-cluster segmentation and the stars indicating cluster centers.\label{fig:faithful}}
\end{figure}

The \texttt{iris} data set collected by \cite{A1935} and analyzed by \cite{F1936} is a widely used benchmark and illustration data set for classification methods. It contains measurements of $4$ features of $50$ iris blossoms for each of three different species, leading to an overall sample size of $n=150$ of data in $\mathbb{R}^4$. One of the three species forms a clearly distinct cluster, which can be easily identified. However, the clusters of the two other species have some overlap and it is not immediately clear from the data, whether three clusters should be fit to the data. A two-cluster model fits the data very well and our test rejects with a p-value of $0$ with $B = 10\,000$. However, when fitting three clusters, as shown in Figure~\ref{fig:iris}, although the results reproduce the true classes well, we find p-value $p_d = 1$, so the test does not reject the possibility of multiple minima. This indicates that the three-cluster Gaussian mixture segmentation of the data set is unreliable and more data are needed for a unique result.

\begin{figure}[h!]
  \centering
  \subcaptionbox{species labels}[0.45\textwidth]{\includegraphics[width=0.45\textwidth]{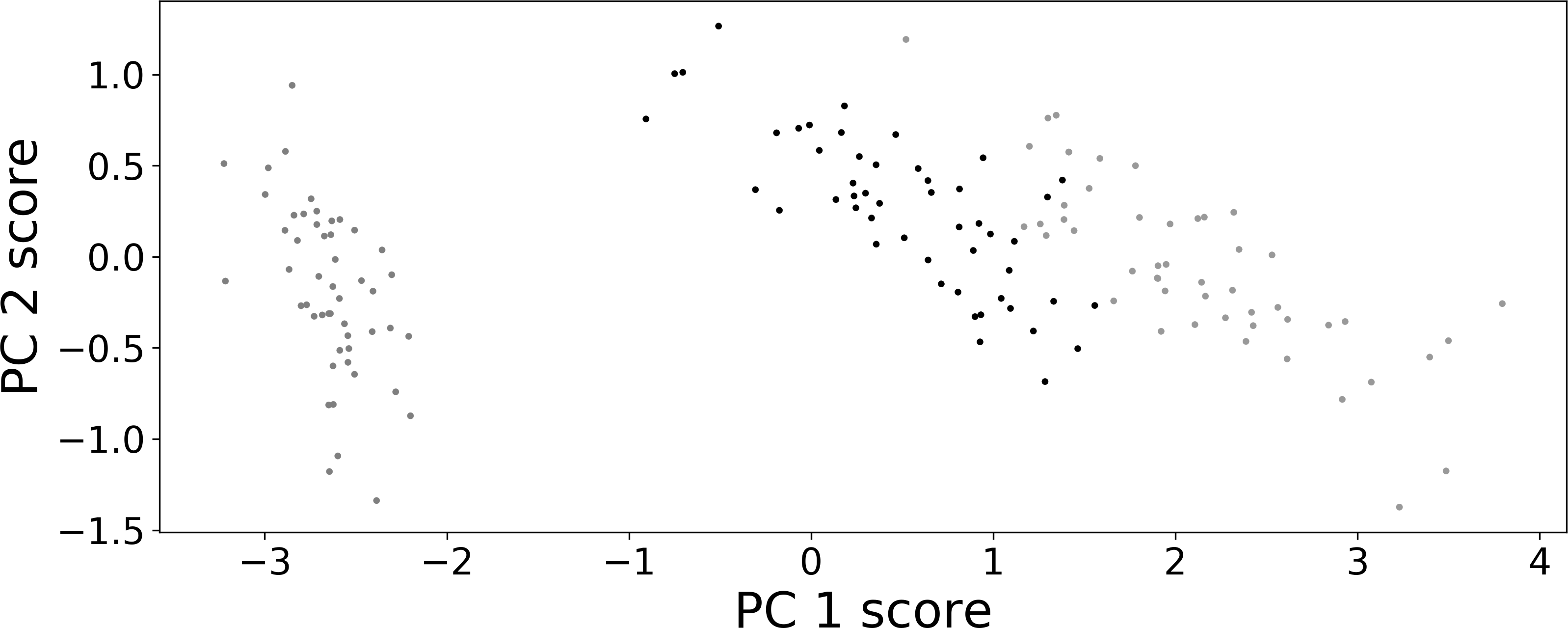}}
  \hspace*{0.02\textwidth}
  \subcaptionbox{Gaussian mixture clusters}[0.45\textwidth]{\includegraphics[width=0.45\textwidth]{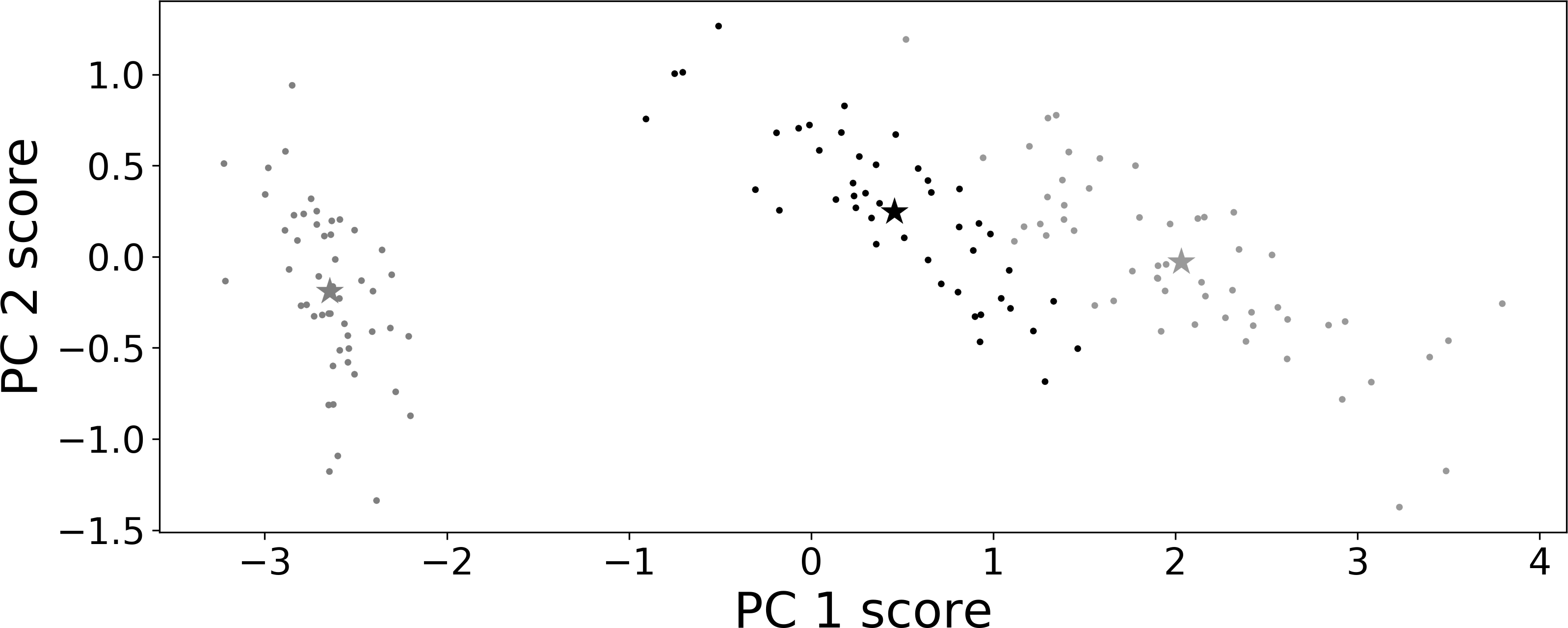}}
  \caption{Scatter plots of the first two principal components of the iris blossom data. The colors in panel~(a) highlight the three species, while the colors in panel~(b) show the three clusters and their centers in the three-cluster segmentation. The clustering results are in very good agreement with the true labels. However, the fact that our hypothesis test does not reject the hypothesis that another clustering result may be equally valid calls the significance of this result into question.
    \label{fig:iris}}
\end{figure}

\begin{Int}
  For the simple Gaussian mixture clustering, which can be interpreted as an unsupervised classification algorithm, we find that the result assuming two components yields a unique result for both data sets. Since the \texttt{faithful} data set consists of two modes, the results that the fit with three components is ambiguous is expected. The \texttt{iris} data set, however, consists of three subsets and a clustering assuming three components reflects the true labels surprisingly well. However, the test for non-uniqueness leads one to question the reliability of the result, since the existence of an alternate ``best fit'' for the population is not ruled out. This is an example for how the test for non-uniqueness of descriptors can be used to indirectly determine the optimal number of clusters.
\end{Int}

\bibliographystyle{Chicago}
\bibliography{bibliography}

\end{document}